\documentclass[12pt]{amsart}
\usepackage{amsmath,amssymb,amsthm,tikz,tikz-cd,color,xcolor,mathdots,amscd}
\usepackage[centertableaux]{ytableau}
\usepackage[all]{xy}
\usepackage[margin=1in]{geometry}
\usepackage{placeins}
\usetikzlibrary{shapes,arrows,svg.path}
\usepackage{booktabs}

\usepackage[colorlinks=true, pdfstartview=FitV, linkcolor=blue, citecolor=blue, urlcolor=blue]{hyperref}

\newif\ifexamples
\examplestrue

\theoremstyle{plain}
\newtheorem{theorem}{Theorem}[section]
\newtheorem{proposition}[theorem]{Proposition}
\newtheorem{corollary}[theorem]{Corollary}
\newtheorem{lemma}[theorem]{Lemma}
\newtheorem{conjecture}[theorem]{Conjecture}

\theoremstyle{definition}

\theoremstyle{remark}
\newtheorem{remark}[theorem]{Remark}
\newtheorem{example}[theorem]{Example}
\numberwithin{equation}{section}

\definecolor{darkgreen}{RGB}{0,180,0}

\definecolor{darkred}{rgb}{0.7,0,0} 
\definecolor{darkblue}{rgb}{0,0,0.7} 
\newcommand{\defn}[1]{{\color{darkred}\emph{#1}}} 

\definecolor{UQgold}{RGB}{196, 158, 54} 
\definecolor{UQpurple}{RGB}{73, 7, 94} 

\newcommand{\ZZ}{\mathbb{Z}}

\newcommand{\CC}{\mathbb{C}}
\newcommand{\cc}{\mathbf{c}}
\newcommand{\xx}{\mathbf{x}}
\newcommand{\zz}{\mathbf{z}}
\newcommand{\z}{\mathbf{z}}

\newcommand{\mcT}{\mathcal{T}}

\newcommand{\ibar}{\overline{\imath}}

\newcommand{\bon}{\overline{1}}
\newcommand{\btw}{\overline{2}}

\newcommand{\bn}{\overline{n}}

\newcommand{\fgl}{\mathfrak{gl}}
\newcommand{\agl}{\widehat{\fgl}}
\newcommand{\states}{\mathfrak{S}}  
\newcommand{\Dop}{\pi}  
\newcommand{\Aop}{\overline{\Dop}}  
\newcommand{\univR}{\mathcal{R}}  

\newcommand{\wt}{\operatorname{wt}}  
\newcommand{\ch}{\operatorname{ch}}  

\newcommand{\GL}{\operatorname{GL}}
\newcommand{\Sp}{\operatorname{Sp}}
\newcommand{\SO}{\operatorname{SO}}

\newcommand{\key}{\operatorname{\mathbf{key}}}
\newcommand{\PP}{\mathcal{P}} 
\newcommand{\king}{\mathcal{K}} 
\newcommand{\sundaram}{\mathcal{S}} 
\newcommand{\uqsg}{U_q\bigl(\widehat{\mathfrak{gl}}(2n|1)\bigr)} 

\usepackage{listings}
\lstdefinelanguage{Sage}[]{Python}
{morekeywords={False,sage,True},sensitive=true}
\definecolor{dblackcolor}{rgb}{0.0,0.0,0.0}
\definecolor{dbluecolor}{rgb}{0.01,0.02,0.7}
\definecolor{dgreencolor}{rgb}{0.2,0.4,0.0}
\definecolor{dgraycolor}{rgb}{0.30,0.3,0.30}

\usepackage[colorinlistoftodos]{todonotes}

\begin{document}
\title[Quasi-solvable models for Demazure atoms]{Quasi-solvable lattice models for $\Sp_{2n}$ and $\SO_{2n+1}$ Demazure atoms and characters}

\author{Valentin Buciumas}
\address[V.~Buciumas]{Department of Mathematical and Statistical Sciences, 
University of Alberta, 
Edmonton, AB T6G 2G1, 
Canada}
\email{valentin.buciumas@gmail.com}
\urladdr{https://sites.google.com/site/valentinbuciumas/}

\author{Travis Scrimshaw}
\address[T.~Scrimshaw]{School of Mathematics and Physics, 
The University of Queensland, 
St.\ Lucia, QLD, 4072, 
Australia}
\email{tcscrims@gmail.com}
\urladdr{https://people.smp.uq.edu.au/TravisScrimshaw/}


\begin{abstract}
We construct colored lattice models whose partition functions represent symplectic and odd orthogonal Demazure characters and atoms. 
We show that our lattice models are not solvable, but we are able to show the existence of sufficiently many solutions of the Yang--Baxter equation that allows us to compute functional equations for the corresponding partition functions.
From these functional equations, we determine that the partition function of our models are the Demazure atoms and characters for the symplectic and odd orthogonal Lie groups.
We coin our lattice models as quasi-solvable.
We use the natural bijection of admissible states in our models with Proctor patterns to give a right key algorithm for reverse King tableaux and Sundaram tableaux.
\end{abstract}

\maketitle

\section{Introduction}
\label{sec:intro}

The Yang--Baxter equation, also known as the star--triangle equation from its description in electrical networks~\cite{Kennelly99}, was first applied to two-dimensional statistical mechanical models by McGuire~\cite{McGuire64} to reduced the interaction of multiple particles down to pairwise scattering.
This was subsequently generalized to multiple species of particles by Yang~\cite{Yang67}.
Independently, Baxter also used the Yang--Baxter equation to show transfer matrices commute and solve the eight-vertex model~\cite{Baxter71,Baxter72} in what is now known as the train argument.
For additional history, we refer the reader to~\cite[Ch.~13]{McCoy10}.
The Yang--Baxter equation has since appeared in many diverse mathematical contexts beyond its origin in statistical mechanics.
For example, solutions to the Yang--Baxter equation (which correspond to Reidemeister III move) lead to knot invariants such as the Jones polynomial; see~\cite{Jones89,Turaev88} for a lattice model approach to the Jones polynomial.
A more recent application that has received significant attention is in the study of probabilistic models, where solutions to the Yang--Baxter equation control the dynamics such as in~\cite{AGS19,Borodin17,BorodinWheelerColored,BW20polymer,KMMO16,KMO15,KMO16,KMO16II,MS13,MS14,MS20}.

A lattice (or vertex) model is a finite grid where the edges are labeled and satisfy some local conditions around vertices, usually with some additional boundary conditions.
These local conditions are also assigned weights called Boltzmann weights, and the collection of these local conditions with their Boltzmann weights is called an $L$-matrix.
We extend the notation of a Boltzmann weight to any valid labeling of the grid, called a state of the model, by taking the product of all of the Boltzmann weights of the vertices.
We call a lattice model integrable or solvable if there exits an additional crossing called an $R$-matrix that can move past a pair of vertices, which is the $RLL$ form of the Yang--Baxter equation.
The Yang--Baxter equation then implies certain functional equations that are subsequently usually solved to determine the partition function of the lattice model, the generating function of the states of the model.

A now classical approach has been to construct a solvable lattice model based on the $R$-matrix isomorphism for $U_q(\agl_2)$-modules such that the partition function is a certain special function, such as a (symmetric) Grothendieck polynomial, a (spin) Hall--Littlewood polynomial, or a spherical Whittaker function (see, for example,~\cite{BBF11,BorodinWheelerColored,CGKM20,WZJ16,WZJ19}).
Indeed, the quantum groups structure ensures the $R$-matrix for the model satisfies the Yang--Baxter equation.
This approach has many fruitful consequences, yielding often simple proofs of certain combinatorial identities that are otherwise hard to prove directly; for example, Kuperberg's proof~\cite{Kuperberg96} counting the number of alternating sign matrices~\cite{Zeilberger96}.

A more recent approach has been instead use $U_q(\agl_n)$-representations to build solvable lattice models, which has the effect of introducing colors to the lattice model.
Moreover, the Yang--Baxter equation in terms of the solvable lattice model can be restated as a purely algebro-combinatorial statement without reference to a quantum group representation.
Both of these ideas have been quite fruitful by allowing authors to develop colored lattice models associated to certain special functions to break them into more atomic pieces and study their properties.
The first example is by Borodin and Wheeler~\cite{BorodinWheelerColored} with nonsymmetric spin Hall--Littlewood polynomials, which are the atoms of spin Hall--Littlewood polynomials~\cite{Borodin17,BP18}.
They also obtained the analogous result for Macdonald polynomials $P_{\lambda}$ in~\cite{BorodinWheelernsMac}, along with Garbali and Wheeler~\cite{GW18} for modified Macdonald polynomials $\widetilde{H}_{\lambda}$.
The first author and coauthors expanded on this by directly colorizing the classic five-vertex model for Schur functions in~\cite{BBBG19} to give $\GL_r$ Demazure atoms and keys.
The K-theory analog of this was done by the authors and Weber in~\cite{BSW20} using the model by Motegi and Sakai for Grothendieck polynomials~\cite{MS13,MS14} to give the first proof of a combinatorial interpretation of Lascoux atoms~\cite{BSW20}.

The goal of this paper is to do the analogous colorization of the model from~\cite{Gray17,Ivanov12} specialized for $\Sp_{2n}$ characters and a similar model for $\SO_{2n+1}$ characters.
This model is based off two different types of rows with Boltzmann weights called $\Gamma$-weights and $\Delta$-weights in~\cite{BBF11} with an additional U-turn ``vertex'' between pairs of successive rows called a $K$-matrix.
Such models would be solvable if there is an $R$-matrix for any pair of rows (so four $R$-matrices in total: $\Gamma\Gamma$, $\Gamma\Delta$, $\Delta\Gamma$, and $\Delta\Delta$) satisfying the Yang--Baxter equation.
However, our models are not solvable: for any natural colored version of the models in~\cite{Gray17,Ivanov12}, we can only construct at most three $R$-matrices that satisfy the Yang--Baxter equation.
Therefore, there does not exist a natural solvable colored analog of the models~\cite{Gray17,Ivanov12}.
Despite this setback, we are able to use the three solutions of the Yang--Baxter equation to compute the partition function of our models explicitly and show it is equal to a Demazure atom for $\Sp_{2n}$ and $\SO_{2n+1}$.
Because of this, we call our model \defn{quasi-solvable}.
Subsequently, our proofs required novel techniques to address this deficiency, which will likely be useful in generalizing this to other colored U-turn lattice models such as for Iwahori Whittaker functions and Hall--Littlewood polynomials for the symplectic group.

Before describing our methodology, let us discuss Demazure atoms and characters for the simple Lie groups $G = \Sp_{2n}, \SO_{2n+1}$.
The characters of irreducible finite-dimensional highest weight representations of $G$ with highest weight $\lambda$ are certain polynomials in $\ZZ[\zz^{\pm 1}]$, where $\zz = (z_1, \dotsc, z_n)$, that are invariant under the corresponding Weyl group $W$ of signed permutations (also known as the hyperoctohedral group).
By a classic formula of Demazure~\cite{Andersen85,Demazure74,Demazure74II,LMS79} (see also~\cite{L95-3,K93}), these can be described by applications of (isobaric) divided difference operators corresponding to any reduced word of the longest element $w_0 \in W$.
Since the divided difference operators satisfy the corresponding braid group relations, Matsumoto's theorem~\cite{Matsumoto64} implies we can define a partial character for any $w \in W$.
These partial characters are called Demazure characters $D_w(\lambda)$ and are characters for certain representations of the (standard) Borel subgroup $B \subseteq G$.
The divided difference operators $\Dop_i$ corresponding to the $i$-th simple reflection are also projections, and so we can define new operators $\Aop_i = \Dop_i - 1$ that also satisfy the braid relations.
These give rise to smaller polynomials called Demazure atoms $A_w(\lambda)$~\cite{LS82,LS83,LS90,Mason09} that encode the change as the length of the Weyl group element increases.
More precisely, a Demazure character is a sum of the atoms
\begin{equation}
\label{eq:Demazure_atom_sum_intro}
D_w(\lambda) = \sum_{v \leq w} A_v(\lambda),
\end{equation}
where $\leq$ is the (strong) Bruhat order.

Now we turn to our proofs.
A standard technique to produce the functional equations for the partition function of solvable (colored) lattice models is the train argument.
The train argument consists of adding an $R$-matrix to a pair of rows and then passing it through to the other side by repeated use of the Yang--Baxter equation.
This is what was used to show the functional equations satisfied the divided difference operator relations in~\cite{BBBG19,BFHTW20,BSW20,BS20} and the $z_i \leftrightarrow z_j$ symmetry in~\cite{EKLP92,Gray17,Ivanov12}.
In order to produce functional equations for U-turn lattice models like for the uncolored model, there are two additional type of relations needed.
The first is the reflection equation with a pair of $R$-matrices and $K$-matrices in a type BC braid relation $KRKR = RKRK$, which underlies the computations for the $z_i \leftrightarrow z_{i+1}$ (or type A) symmetry in~\cite{Gray17,Ivanov12}.
The other relation is the fish equation involving a single $R$-matrix and $K$-matrix, which is what was used in~\cite{Gray17,Ivanov12} to show the partition function satisfied the $z_i \leftrightarrow z_i^{-1}$ (or type BC) symmetry.

However, since we are unable to freely pass through one of our $R$-matrices, a novel approach is required.
To show the type A divided difference operator relations, our initial step is following~\cite{Gray17,Ivanov12} by applying a block $R$-matrix consisting of our four types of $R$-matrices on the left side of the mode.
We use the train argument for the three types of $R$-matrices that satisfy the Yang--Baxter equation to pass them to the right side of our model.
Then we use the reflection equation to bring two $R$-matrices together, which we can then remove by applying the unitary relation $R^2 = \beta \cdot 1$ for some constant $\beta$.
Finally, we pass the remaining $R$-matrix back to the left side and apply the corresponding unitary relation.
The result is the desired functional equation.

To obtain the type BC divided difference operator relation, we additionally have to get around the obstruction that the fish equation does not hold in the colored model.
We achieve this by first following~\cite{Ivanov12} and using a straightforward bijection of states to change the bottom row of our model from $\Delta$-weights to $\Gamma$-weights.
We then apply the train argument to bring the $\Gamma\Gamma$ $R$-matrix to the right, where we perform a direct analysis of the possible cases to obtain our functional equation.

Because our atom model is based off the atom colored models of~\cite{BBBG19,BSW20}, we apply the same small tweak to the colored model in~\cite{BSW20} to obtain a model for Demazure characters of $G = \Sp_{2n}, \SO_{2n+1}$.
For this tweaked model, the proof is entirely analogous to the case for the Demazure atoms.
Furthermore, the same combinatorial proof of~\cite[Thm.~3.6]{BSW20} of changing how the paths interact recovers Equation~\eqref{eq:Demazure_atom_sum_intro} (which can be turned around to compute the partition function combinatorially).
Consequently, we also obtain a new uncolored model for characters of irreducible $\SO_{2n+1}$-representations.

One application of our model is the computation of a (right) key for tableaux used in combinatorial descriptions of characters of irreducible $G$-representations.
More specifically, it is known~\cite{Gray17,Ivanov12} that the states of the uncolored model are in natural bijection with Proctor patterns~\cite{Proctor94}, which are naturally in bijection with certain tableaux.
However, the weight in our model has been twisted by $i \leftrightarrow n+1- i$ with the natural one from the Proctor patterns.
Thus our model naturally gives a key algorithm on \emph{reverse} King tableaux~\cite{King75,King76} for $G = \Sp_{2n}$ and Sundaram tableaux~\cite{Sundaram90} for $G = \SO_{2n+1}$.
This is in parallel to the models in~\cite{BBBG19,BSW20} with Gelfand--Tsetlin patterns and reverse semistandard tableaux because we can only use one particular ordering of the spectral parameters.
We conjecture that our key map agrees with the representation-theoretic description coming from the Kashiwara crystal structure~\cite{BS17,K93} described in~\cite{JL19}, the key map on Kashiwara--Nakashima (KN) tableaux~\cite{KN94} given in~\cite{Santos19}, and with the crystal structure defined recently on King tableaux by Lee~\cite{Lee19}.
A salient ingredient is a weight-preserving bijection from reverse King (resp.\ Sundaram) tableaux to regular King (resp.\ Sundaram) tableaux.
To relate the KN tableaux with the King/Sundaram tableaux, we expect a combination of the Sheats bijection~\cite{Sheats99} and repeated tableau switching algorithm iterations~\cite{BSS96} would yield the desired crystal isomorphism.

Our solutions to the Yang--Baxter equation have a partial interpretation in terms of $R$-matrices of the quantum supergroup $\uqsg$ (\cite{BazhanovShadrikov}, see also~\cite{Kojima13}) corresponding to the evaluation representation and its dual in the limit $q\to0$ (see Section~\ref{sec:quantum_groups}).
This interpretation gives evidence for the quantum generalization of our results to models representing Iwahori Whittaker functions and nonsymmetric Hall--Littlewood polynomials for the symplectic and odd orthogonal groups. 
In type A, a similar interpretation for lattice model $R$-matrices lead to a relation between quantum group $R$-matrices and $p$-adic intertwining integrals. 
We are able to show that two of our solutions to the Yang--Baxter equation come from $q \to 0$ limits of $\uqsg$ $R$-matrices.
We also explain why we are unable to give a complete quantum supergroup interpretation of our $R$-matrices. 
 
Generally speaking, $K$-matrices in U-turn lattice models should correspond to solutions of the reflection equation coming from quantum symmetric pairs~\cite{Kolbaffine,Letzter}.
Quantum symmetric pairs are certain coideal subalgebras of quantum groups that have recently been connected to many other areas of representations theory such as canonical basis, Schur--Weyl dualities, categorification, and geometry~\cite{BalagovicKolb,BaoWangInventiones,BaoWangAsterique,EhrigStroppel,FanLaiLiLuoWang}.
Unfortunately, there is not much known about $K$-matrices and quantum symmetric pairs corresponding to quantum supergroups like $\uqsg$. 
It would be interesting to relate our lattice models, especially their possible quantum generalizations, to such quantum symmetric pairs. 
This could lead to novel relations between the representation theories of symplectic and odd orthogonal $p$-adic groups and that of quantum symmetric pairs. 

This paper is organized as follows.
In Section~\ref{sec:background}, we give the necessary background on Demazure characters and atoms.
In Section~\ref{sec:colored_atoms}, we construct our lattice model for Demazure atoms and prove its quasi-solvability to give functional equations to prove our first main theorem.
We give a partial quantum group interpretation of our Boltzmann weights. 
In Section~\ref{sec:colored_Demazure}, we give a slightly modified quasi-solvable lattice model for Demazure characters and our second main theorem.
In Section~\ref{sec:proctor}, we relate the admissible states in our lattice models to Proctor patterns and use this to give an algorithm for a (right) key for for reverse King tableaux and Sundaram tableaux.

Shortly after this paper appeared on the arXiv, independent work by Zhong~\cite{Zhong21} on stochastic type C vertex models was also posted in which the colored model is a different quantization than the $R$-matrix for $\uqsg$ we utilize and is possibly a gauge transformation of our atom model when taking $q = 0$.

\subsection*{Acknowledgments}

The authors thank Ben Brubaker and Nathan Gray for many invaluable discussions. 
We thank Huafeng Zhang for discussions about $\uqsg$ $R$-matrices.
This work benefited from computations using \textsc{SageMath}~\cite{sage,combinat}.

V.~Buciumas was supported by the Australian Research Council DP180103150 and DP17010264, NSERC Discovery RGPIN-2019-06112 and the endowment of the M.V.~Subbarao Professorship in Number Theory.

\section{Background}
\label{sec:background}

We will start with a review of the theory of Demazure operators.
Let $\Phi$ be the root system and $\Lambda$ be the weight lattice of a complex reductive Lie group $G$ with maximal torus $T$.
Let $n$ be the rank of $\Phi$. 
We identify $\Lambda$ with the group $X^{\ast}(T)$ of rational characters of $T$.
For $\zz \in T$ and $\lambda \in \Lambda$, we denote by $\zz^{\lambda}$ the application of $\lambda$ to $\zz$.
Let $\mathcal{O}(T)$ be the set of polynomial functions on $T$, that is, finite linear combinations of the functions $\zz^{\lambda}$ for $\lambda \in \Lambda$.
Let $\Phi^{+}$ (resp.~$\Phi^{-}$) be the set of positive (resp.\ negative) roots, and let $\alpha_i$ ($i\in I = \{1, 2, \dotsc, n\}$) be the simple positive roots.
Let $\alpha_i^{\vee} \in X_{\ast} (T)$ denote the corresponding simple coroots and $s_i$ the corresponding simple reflections generating the Weyl group $W$.
The Weyl group acts on the weight lattice and therefore on the space $\mathcal{O}(T)$.
We shall denote this action by $w \cdot f(\z) := f(w\z)$.
For $w \in W$, let $\ell(w)$ denote the length of $w$, the smallest number of simple reflections such that $w = s_{i_1} \dotsm s_{i_{\ell}}$, which is called a reduced word for $w$.
Let $w_0$ be the long element of the Weyl group and $\leq$ denote the (strong) Bruhat order on $W$.
For more information about properties of the Weyl group, we the refer the reader to~\cite{Humphreys90}.

\subsection{Demazure characters and atoms}

Given $s_i$ a simple reflection, we can define the associated isobaric Demazure operator acting on $f \in \mathcal{O} (T)$ as
\begin{equation} \label{isobaricddefined}
\Dop_i f (\zz) =  \frac{f(\zz) - \zz^{-\alpha_i} f (s_i \zz)}{1 - \zz^{-\alpha_i}}.
\end{equation}
The numerator is divisible by the denominator, so the resulting function is again in $\mathcal{O}(T)$.

One can check that $\Dop^2_i = \Dop_i = s_i \Dop_i$.
Given any $\mu \in \Lambda$, set  $k = \langle \mu, \alpha_i^{\vee} \rangle$ so $s_i (\mu) = \mu - k\alpha_i$.
Then the action on the monomial $\zz^{\mu}$ is given explicitly by 
\begin{equation}
  \Dop_i \zz^{\mu} = \begin{cases}
    \zz^{\mu} +\zz^{\mu - \alpha_i} + \cdots +\zz^{s_i (\mu)}
        & \text{if } k \geqslant 0, \\
    0 & \text{if } k = - 1,\\
    - (\zz^{\mu + \alpha_i} +\zz^{\mu + 2 \alpha_i} + \cdots + \zz^{s_i (\mu + \alpha_i)})
        & \text{if } k < - 1.
  \end{cases}
\end{equation}
Define $\Aop_i := \Dop_i - 1$. Explicitly, we have
\begin{equation}\label{eq:partial}
\Aop_i f(\zz) := \frac{f(\zz) - f(s_i\zz)}{\zz^{\alpha_i} - 1}.
\end{equation}

Both $\Dop_i$ and $\Aop_i$ satisfy the braid relations.
Thus, for any $w \in W$, we can choose any reduced word $w = s_{i_1} \dotsm s_{i_k}$ to define $\Dop_w = \Dop_{i_1} \cdots \Dop_{i_k}$ and $\Aop_w = \Aop_{i_1} \cdots \Aop_{i_k}$ by Matsumoto's theorem~\cite{Matsumoto64}.
For $w = 1$, we set $\Dop_1 = \Aop_1 = 1$.

For $\lambda$ a dominant weight, let $\chi_{\lambda}$ denote the character of the irreducible representation $V_{\lambda}$ with highest weight $\lambda$.
The \defn{Demazure character formula} is the identity, for $\zz\in T$:
\[
\chi_{\lambda} (\zz) = \Dop_{w_0} \zz^{\lambda} .
\]
For a proof, see~\cite[Thm.~25.3]{Bump13}.
More generally for any Weyl group element $w$, we may consider $\Dop_w\zz^\lambda$ and $\Aop_w\zz^\lambda$.
These polynomials are called \defn{Demazure characters} and \defn{Demazure atoms}, respectively.
The following relation between the two is well-known.

\begin{theorem}
  \label{thm:lskeys}
  Let $f\in\mathcal{O}(T)$. Then
  \begin{equation}
    \label{phidemaz}
    \Dop_{w} f(\zz) = \sum_{y \leqslant w} \Aop_{y} f(\zz).
  \end{equation}
\end{theorem}

As a corollary we obtain the following decomposition of characters in terms of Demazure atoms. 

\begin{corollary}
  \label{cor:character_atom_decomp}
  We have
  \[
    \ch V_{\lambda} = \Dop_{w_0} \zz^{\lambda} = \sum_{y} \Aop_{y} \zz^{\lambda}.
  \]
\end{corollary}

\subsection{Signed permutations and the Weyl group action}

For the remainder of the paper, we will only consider Cartan types $BC$.
We identify the maximal torus $T$ with the space $(\CC^*)^n$, where $n$ is the rank of the group.
The Weyl group $W$ of type $B_n$ is isomorphic to the Weyl of type $C_n$, and it is known as the hyperoctohedral group.
It is generated by the simple reflections $s_i$ for $i \in \{1, \dotsc, n\}$ subject to the relations:
\begin{equation}
\begin{aligned}
s_{n}s_{n-1}s_ns_{n-1} &= s_{n-1} s_ns_{n-1}s_n, \\
s_i s_{i+1} s_i &= s_{i+1} s_i s_{i+1},  & &\mbox{ if } i < n-1\\
s_i s_j &= s_j s_i, && \mbox{ if } |i-j| \geq 2\\
s_i^2 &= 1.
\end{aligned}
\end{equation}
The Weyl group acts on elements $\zz \in (\CC^*)^n$ as follows:
\begin{subequations}
\begin{align}
s_i (\ldots, z_i, z_{i+1}, \ldots) & =  (\ldots, z_{i+1}, z_{i},\ldots),   && \text{if } i < n,\\
s_n (\ldots, z_{n-1}, z_{n}) & =  (\ldots, z_{n-1}, z_{n}^{-1}),  && \text{if } i = n.
\end{align}
\end{subequations}

The elements of $W$ can be explicitly described using \defn{signed permutations} of $n$, which are permutations of
\[
1 < 2 < \cdots < n < \bn < \cdots < \bon
\]
such that $w(i) = \overline{w(\ibar)}$.
Here we use the convention that $\overline{\ibar}=i$.
Thus we can determine a signed permutation by the image of $1 \leq i \leq n$.
The simple transposition for $i < n$  is given by $s_i = (i \; i+1)$, and $s_n$ sends $n \leftrightarrow \bn$.
An \defn{inversion} is a pair $1 \leq i < j \leq n$ such that $w(i) > w(j)$.
The longest element $w_0$ is given by the signed permutation $[\bon, \btw, \dotsc, \bn]$.

The subgroup of $W$ generated by $s_i$ for  $i<n$ is a subgroup isomorphic to the Weyl group of type $A$.
We shall denote the subgroup by $W^A$.

\subsection{Functional equations}

We now discuss the explicit functional equations for the Demazure characters and atoms that will be used to prove the main theorems in this paper.

We first consider Demazure characters.
Equation~\eqref{isobaricddefined} can be written explicitly as
\begin{subequations}
\label{eq:demazure_explicit}
\begin{align}
\Dop_i f(\zz) & :=
\dfrac{f(\zz) - z_i^{-1} z_{i+1} f(s_i\zz)}{1 - z_i^{-1} z_{i+1}}, && \text{if } i<n, \\
\Dop_n f(\zz) & :=
\dfrac{f(\zz) -  z_n^{-1}f(s_n\zz)}{1 - z_n^{-1}},  && \text{in type B,} \\
\Dop_n f(\zz) & :=
\dfrac{f(\zz) -  z_n^{-2}f(s_n\zz)}{1 - z_n^{-2}},  && \text{in type C.}
\end{align} 
\end{subequations}
Let us denote $D_w(\z,\lambda) := \Dop_w z^{\lambda}$.
Let $s_i$ be a simple reflection and $w \in W$ such that $\ell(s_i w)>\ell(w)$.
From Equations~\eqref{eq:demazure_explicit}, we deduce the following:
\begin{subequations}
\label{eq:demauze_functional}
\begin{align}
\label{eq:demauze_functional_A}
(z_i -z_{i+1})D_{s_iw}(\zz, \lambda) & = z_{i} D_{w}(\zz, \lambda) - z_{i+1} D_{w}(s_i\zz, \lambda),  && \text{if } i < n, \\
\label{eq:demauze_functional_B}
(z_n -1)D_{s_n w}(\zz, \lambda) & = z_n D_{w}(\zz, \lambda) - D_{w}(s_n\zz, \lambda),  && \text{in type B,} \\
\label{eq:demauze_functional_C}
(z_{n}^2 -1)D_{s_n w}(\zz, \lambda) & = z_n^2 D_{w}(\zz, \lambda) - D_{w}(s_n\zz, \lambda),  && \text{in type C.}
\end{align}
\end{subequations}

Next, we consider the Demazure atoms.
In this case, Equation~\eqref{eq:partial} can be rewritten as
\begin{subequations}
\label{eq:atom_explicit}
\begin{align}
\Aop_i f(\zz) & :=
\dfrac{f(\zz)-f(s_i\zz)}{z_i z_{i+1}^{-1}-1}, && \text{if } i<n, \\
\Aop_n f(\zz) & :=
\dfrac{f(\zz)-f(s_n\zz)}{z_n-1},  && \text{in type B,} \\
\Aop_n f(\zz) & :=
\dfrac{f(\zz)-f(s_n\zz)}{z_n^2-1},  && \text{in type C.}
\end{align} 
\end{subequations}
Let us denote $A_w(\z,\lambda) := \Aop_w z^{\lambda}$.
Let $s_i$ be a simple reflection and $w \in W$ such that $\ell(s_i w)>\ell(w)$.
We rewrite the equation above as
\begin{subequations}
\label{eq:atom_functional}
\begin{align}
\label{eq:atom_functional_A}
(z_i -z_{i+1})A_{s_iw}(\zz, \lambda) & = z_{i+1} \bigl( A_{w}(\zz, \lambda) - A_{w}(s_i\zz, \lambda) \bigr),  && \text{if } i < n, \\
\label{eq:atom_functional_B}
(z_n -1)A_{s_n w}(\zz, \lambda) & = A_{w}(\zz, \lambda) - A_{w}(s_n\zz, \lambda),  && \text{in type B,} \\
\label{eq:atom_functional_C}
(z_{n}^2 -1)A_{s_n w}(\zz, \lambda) & = A_{w}(\zz, \lambda) - A_{w}(s_n\zz, \lambda),  && \text{in type C.}
\end{align}
\end{subequations}

\section{Colored lattice models and Demazure atoms}
\label{sec:colored_atoms}

We will construct colored lattice models that represent Demazure atoms in Cartan type $B$ and $C$.
These models generalize the work in~\cite{BBBG19} where type $A$ Demazure atoms have been represented as partition functions of lattice models.
The current paper and~\cite{BBBG19} produce colored models that are a refinement of the $q=0$ uncolored models in~\cite{BBF11} (representing Schur polynomials) and~\cite{Ivanov12} (representing symplectic Schur polynomials), respectively.   
Our odd orthogonal model does not refine any pre-existing model.  

\begin{remark}
\label{rem:uncolored_semidual}
Our model is in fact a refinement of a semidual version of the model in~\cite{Ivanov12} obtained by interchanging $0 \leftrightarrow 1$ on each of the horizontal components.
This choice allows us to have a more natural description of our colored lattice model and helps with visualizing admissible states in the model by using colored paths.
\end{remark}

We work with fixed set $\cc = \{c_1 < c_2 < \cdots < c_n < \overline{c_n} < \cdots <\overline{c_1}\}$ of ordered colors.
We use the conventions $\overline{\overline{c}} := c$ and $c_{\ibar} := \overline{c}_i$.
For $w \in W$ we define $w \cc = (c_{w(1)}, c_{w(2)}, \dotsc, c_{w(n)},c_{w(\bn)},\dotsc,c_{w(\bon)})$ to be the set of colors permuted by $w$.
Explicitly, $s_i$ permutes the colors $c_i \leftrightarrow c_{i+1}$ and $\overline{c_i} \leftrightarrow \overline{c_{i+1}}$ and $s_n$ permutes the colors $c_n \leftrightarrow \overline{c_n}$.
The set $\{w\cc \mid w \in W\}$ will index the left boundary conditions of our model. 


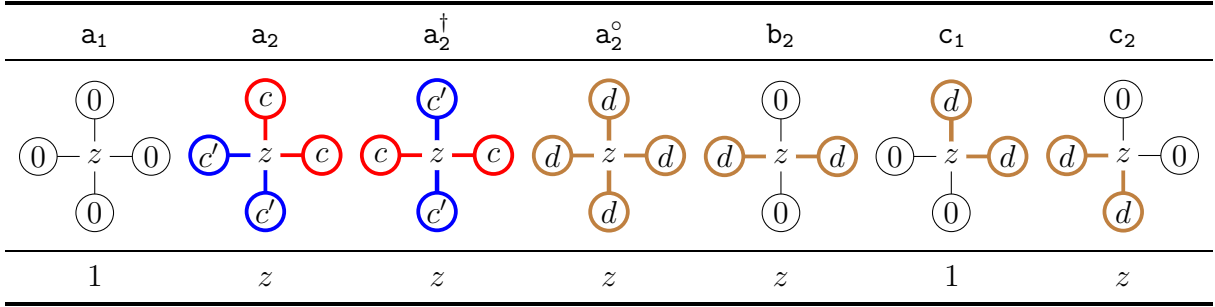
\begin{figure}
\[
\begin{array}{c@{\;\;}c@{\;\;}c@{\;\;}c@{\;\;}c@{\;\;}c@{\;\;}c}
\toprule
  \tt{a}_1&\tt{a}_2&\tt{a}^{\dagger}_2&\tt{a}^{\circ}_2&\tt{b}_2&\tt{c}_1&\tt{c}_2\\
\midrule
\begin{tikzpicture}
\coordinate (a) at (-.75, 0);
\coordinate (b) at (0, .75);
\coordinate (c) at (.75, 0);
\coordinate (d) at (0, -.75);
\coordinate (aa) at (-.75,.5);
\coordinate (cc) at (.75,.5);
\draw (a)--(0,0);
\draw (b)--(0,0);
\draw (c)--(0,0);
\draw (d)--(0,0);
\draw[fill=white] (a) circle (.25);
\draw[fill=white] (b) circle (.25);
\draw[fill=white] (c) circle (.25);
\draw[fill=white] (d) circle (.25);
\node at (0,1) { };
\node at (a) {$0$};
\node at (b) {$0$};
\node at (c) {$0$};
\node at (d) {$0$};
\path[fill=white] (0,0) circle (.2);
\node at (0,0) {$z$};
\end{tikzpicture}
& \begin{tikzpicture}
\coordinate (a) at (-.75, 0);
\coordinate (b) at (0, .75);
\coordinate (c) at (.75, 0);
\coordinate (d) at (0, -.75);
\coordinate (aa) at (-.75,.5);
\coordinate (cc) at (.75,.5);
\draw[line width=0.5mm, blue] (a)--(0,0);
\draw[line width=0.5mm, red] (b)--(0,0);
\draw[line width=0.5mm, red] (c)--(0,0);
\draw[line width=0.5mm, blue] (d)--(0,0);
\draw[line width=0.5mm, blue,fill=white] (a) circle (.25);
\draw[line width=0.5mm, red,fill=white] (b) circle (.25);
\draw[line width=0.5mm, red, fill=white] (c) circle (.25);
\draw[line width=0.5mm, blue, fill=white] (d) circle (.25);
\node at (0,1) { };
\node at (a) {$c'$};
\node at (b) {$c$};
\node at (c) {$c$};
\node at (d) {$c'$};
\path[fill=white] (0,0) circle (.2);
\node at (0,0) {$z$};
\end{tikzpicture}
& \begin{tikzpicture}
\coordinate (a) at (-.75, 0);
\coordinate (b) at (0, .75);
\coordinate (c) at (.75, 0);
\coordinate (d) at (0, -.75);
\coordinate (aa) at (-.75,.5);
\coordinate (cc) at (.75,.5);
\draw[line width=0.5mm, red] (a)--(0,0);
\draw[line width=0.6mm, blue] (b)--(0,0);
\draw[line width=0.5mm, red] (c)--(0,0);
\draw[line width=0.6mm, blue] (d)--(0,0);
\draw[line width=0.5mm, red,fill=white] (a) circle (.25);
\draw[line width=0.5mm, blue,fill=white] (b) circle (.25);
\draw[line width=0.5mm, red, fill=white] (c) circle (.25);
\draw[line width=0.5mm, blue, fill=white] (d) circle (.25);
\node at (0,1) { };
\node at (a) {$c$};
\node at (b) {$c'$};
\node at (c) {$c$};
\node at (d) {$c'$};
\path[fill=white] (0,0) circle (.2);
\node at (0,0) {$z$};
\end{tikzpicture}
& \begin{tikzpicture}
\coordinate (a) at (-.75, 0);
\coordinate (b) at (0, .75);
\coordinate (c) at (.75, 0);
\coordinate (d) at (0, -.75);
\coordinate (aa) at (-.75,.5);
\coordinate (cc) at (.75,.5);
\draw[line width=0.5mm, brown] (a)--(0,0);
\draw[line width=0.6mm, brown] (b)--(0,0);
\draw[line width=0.5mm, brown] (c)--(0,0);
\draw[line width=0.6mm, brown] (d)--(0,0);
\draw[line width=0.5mm, brown,fill=white] (a) circle (.25);
\draw[line width=0.5mm, brown,fill=white] (b) circle (.25);
\draw[line width=0.5mm, brown, fill=white] (c) circle (.25);
\draw[line width=0.5mm, brown, fill=white] (d) circle (.25);
\node at (0,1) { };
\node at (a) {$d$};
\node at (b) {$d$};
\node at (c) {$d$};
\node at (d) {$d$};
\path[fill=white] (0,0) circle (.2);
\node at (0,0) {$z$};
\end{tikzpicture}
& \begin{tikzpicture}
\coordinate (a) at (-.75, 0);
\coordinate (b) at (0, .75);
\coordinate (c) at (.75, 0);
\coordinate (d) at (0, -.75);
\coordinate (aa) at (-.75,.5);
\coordinate (cc) at (.75,.5);
\draw[line width=0.5mm, brown] (a)--(0,0);
\draw(b)--(0,0);
\draw[line width=0.5mm, brown] (c)--(0,0);
\draw (d)--(0,0);
\draw[line width=0.5mm,brown,fill=white] (a) circle (.25);
\draw[fill=white] (b) circle (.25);
\draw[line width=0.5mm,brown,fill=white] (c) circle (.25);
\draw[fill=white] (d) circle (.25);
\node at (0,1) { };
\node at (a) {$d$};
\node at (b) {$0$};
\node at (c) {$d$};
\node at (d) {$0$};
\path[fill=white] (0,0) circle (.2);
\node at (0,0) {$z$};
\end{tikzpicture}
& \begin{tikzpicture}
\coordinate (a) at (-.75, 0);
\coordinate (b) at (0, .75);
\coordinate (c) at (.75, 0);
\coordinate (d) at (0, -.75);
\coordinate (aa) at (-.75,.5);
\coordinate (cc) at (.75,.5);
\draw[line width=0.5mm, brown] (c)--(0,0);
\draw[line width=0.6mm, brown] (b)--(0,0);
\draw (a)--(0,0);
\draw (d)--(0,0);
\draw[line width=0.5mm,brown,fill=white] (c) circle (.25);
\draw[line width=0.5mm,brown,fill=white] (b) circle (.25);
\draw[fill=white] (a) circle (.25);
\draw[fill=white] (d) circle (.25);
\node at (0,1) { };
\node at (a) {$0$};
\node at (b) {$d$};
\node at (c) {$d$};
\node at (d) {$0$};
\path[fill=white] (0,0) circle (.2);
\node at (0,0) {$z$};
\end{tikzpicture}
& \begin{tikzpicture}
\coordinate (a) at (-.75, 0);
\coordinate (b) at (0, .75);
\coordinate (c) at (.75, 0);
\coordinate (d) at (0, -.75);
\coordinate (aa) at (-.75,.5);
\coordinate (cc) at (.75,.5);
\draw (c)--(0,0);
\draw (b)--(0,0);
\draw[line width=0.5mm, brown] (a)--(0,0);
\draw[line width=0.5mm, brown] (d)--(0,0);
\draw[fill=white] (c) circle (.25);
\draw[fill=white] (b) circle (.25);
\draw[line width=0.5mm,brown, fill=white] (a) circle (.25);
\draw[line width=0.5mm,brown, fill=white] (d) circle (.25);
\node at (0,1) { };
\node at (a) {$d$};
\node at (b) {$0$};
\node at (c) {$0$};
\node at (d) {$d$};
\path[fill=white] (0,0) circle (.2);
\node at (0,0) {$z$};
\end{tikzpicture}
\\ 
\midrule
1 &  z & z &z & z & 1 & z\\
\bottomrule
\end{array}
\]
\caption{The colored Boltzmann $\Gamma$-weights with ${\color{red} c} > {\color{blue} c'}$ and ${\color{brown} d}$ being any color.}
\label{fig:colored_gamma_weights}
\end{figure}

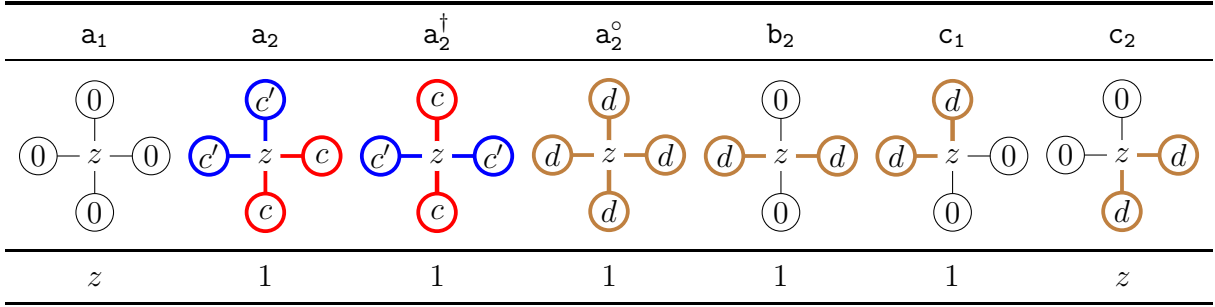
\begin{figure}
\[
\begin{array}{c@{\;\;}c@{\;\;}c@{\;\;}c@{\;\;}c@{\;\;}c@{\;\;}c}
\toprule
  \tt{a}_1&\tt{a}_2&\tt{a}^{\dagger}_2&\tt{a}^{\circ}_2&\tt{b}_2&\tt{c}_1&\tt{c}_2\\
\midrule
\begin{tikzpicture}
\coordinate (a) at (-.75, 0);
\coordinate (b) at (0, .75);
\coordinate (c) at (.75, 0);
\coordinate (d) at (0, -.75);
\coordinate (aa) at (-.75,.5);
\coordinate (cc) at (.75,.5);
\draw (a)--(0,0);
\draw (b)--(0,0);
\draw (c)--(0,0);
\draw (d)--(0,0);
\draw[fill=white] (a) circle (.25);
\draw[fill=white] (b) circle (.25);
\draw[fill=white] (c) circle (.25);
\draw[fill=white] (d) circle (.25);
\node at (0,1) { };
\node at (a) {$0$};
\node at (b) {$0$};
\node at (c) {$0$};
\node at (d) {$0$};
\path[fill=white] (0,0) circle (.2);
\node at (0,0) {$z$};
\end{tikzpicture}
& \begin{tikzpicture}
\coordinate (a) at (-.75, 0);
\coordinate (b) at (0, .75);
\coordinate (c) at (.75, 0);
\coordinate (d) at (0, -.75);
\coordinate (aa) at (-.75,.5);
\coordinate (cc) at (.75,.5);
\draw[line width=0.5mm, blue] (a)--(0,0);
\draw[line width=0.5mm, blue] (b)--(0,0);
\draw[line width=0.5mm, red] (c)--(0,0);
\draw[line width=0.5mm, red] (d)--(0,0);
\draw[line width=0.5mm, blue,fill=white] (a) circle (.25);
\draw[line width=0.5mm, blue,fill=white] (b) circle (.25);
\draw[line width=0.5mm, red, fill=white] (c) circle (.25);
\draw[line width=0.5mm, red, fill=white] (d) circle (.25);
\node at (0,1) { };
\node at (a) {$c'$};
\node at (b) {$c'$};
\node at (c) {$c$};
\node at (d) {$c$};
\path[fill=white] (0,0) circle (.2);
\node at (0,0) {$z$};
\end{tikzpicture}
& \begin{tikzpicture}
\coordinate (a) at (-.75, 0);
\coordinate (b) at (0, .75);
\coordinate (c) at (.75, 0);
\coordinate (d) at (0, -.75);
\coordinate (aa) at (-.75,.5);
\coordinate (cc) at (.75,.5);
\draw[line width=0.5mm, blue] (a)--(0,0);
\draw[line width=0.6mm, red] (b)--(0,0);
\draw[line width=0.5mm, blue] (c)--(0,0);
\draw[line width=0.6mm, red] (d)--(0,0);
\draw[line width=0.5mm, blue,fill=white] (a) circle (.25);
\draw[line width=0.5mm, red,fill=white] (b) circle (.25);
\draw[line width=0.5mm, blue, fill=white] (c) circle (.25);
\draw[line width=0.5mm, red, fill=white] (d) circle (.25);
\node at (0,1) { };
\node at (a) {$c'$};
\node at (b) {$c$};
\node at (c) {$c'$};
\node at (d) {$c$};
\path[fill=white] (0,0) circle (.2);
\node at (0,0) {$z$};
\end{tikzpicture}
& \begin{tikzpicture}
\coordinate (a) at (-.75, 0);
\coordinate (b) at (0, .75);
\coordinate (c) at (.75, 0);
\coordinate (d) at (0, -.75);
\coordinate (aa) at (-.75,.5);
\coordinate (cc) at (.75,.5);
\draw[line width=0.5mm, brown] (a)--(0,0);
\draw[line width=0.6mm, brown] (b)--(0,0);
\draw[line width=0.5mm, brown] (c)--(0,0);
\draw[line width=0.6mm, brown] (d)--(0,0);
\draw[line width=0.5mm, brown,fill=white] (a) circle (.25);
\draw[line width=0.5mm, brown,fill=white] (b) circle (.25);
\draw[line width=0.5mm, brown, fill=white] (c) circle (.25);
\draw[line width=0.5mm, brown, fill=white] (d) circle (.25);
\node at (0,1) { };
\node at (a) {$d$};
\node at (b) {$d$};
\node at (c) {$d$};
\node at (d) {$d$};
\path[fill=white] (0,0) circle (.2);
\node at (0,0) {$z$};
\end{tikzpicture}
& \begin{tikzpicture}
\coordinate (a) at (-.75, 0);
\coordinate (b) at (0, .75);
\coordinate (c) at (.75, 0);
\coordinate (d) at (0, -.75);
\coordinate (aa) at (-.75,.5);
\coordinate (cc) at (.75,.5);
\draw[line width=0.5mm, brown] (a)--(0,0);
\draw(b)--(0,0);
\draw[line width=0.5mm, brown] (c)--(0,0);
\draw (d)--(0,0);
\draw[line width=0.5mm,brown,fill=white] (a) circle (.25);
\draw[fill=white] (b) circle (.25);
\draw[line width=0.5mm,brown,fill=white] (c) circle (.25);
\draw[fill=white] (d) circle (.25);
\node at (0,1) { };
\node at (a) {$d$};
\node at (b) {$0$};
\node at (c) {$d$};
\node at (d) {$0$};
\path[fill=white] (0,0) circle (.2);
\node at (0,0) {$z$};
\end{tikzpicture}
& \begin{tikzpicture}
\coordinate (a) at (-.75, 0);
\coordinate (b) at (0, .75);
\coordinate (c) at (.75, 0);
\coordinate (d) at (0, -.75);
\coordinate (aa) at (-.75,.5);
\coordinate (cc) at (.75,.5);
\draw[line width=0.5mm, brown] (a)--(0,0);
\draw[line width=0.6mm, brown] (b)--(0,0);
\draw (c)--(0,0);
\draw (d)--(0,0);
\draw[fill=white] (c) circle (.25);
\draw[line width=0.5mm,brown,fill=white] (b) circle (.25);
\draw[line width=0.5mm,brown,fill=white] (a) circle (.25);
\draw[fill=white] (d) circle (.25);
\node at (0,1) { };
\node at (a) {$d$};
\node at (b) {$d$};
\node at (c) {$0$};
\node at (d) {$0$};
\path[fill=white] (0,0) circle (.2);
\node at (0,0) {$z$};
\end{tikzpicture}
& \begin{tikzpicture}
\coordinate (a) at (-.75, 0);
\coordinate (b) at (0, .75);
\coordinate (c) at (.75, 0);
\coordinate (d) at (0, -.75);
\coordinate (aa) at (-.75,.5);
\coordinate (cc) at (.75,.5);
\draw (a)--(0,0);
\draw (b)--(0,0);
\draw[line width=0.5mm, brown] (c)--(0,0);
\draw[line width=0.5mm, brown] (d)--(0,0);
\draw[fill=white] (a) circle (.25);
\draw[fill=white] (b) circle (.25);
\draw[line width=0.5mm,brown, fill=white] (c) circle (.25);
\draw[line width=0.5mm,brown, fill=white] (d) circle (.25);
\node at (0,1) { };
\node at (c) {$d$};
\node at (b) {$0$};
\node at (a) {$0$};
\node at (d) {$d$};
\path[fill=white] (0,0) circle (.2);
\node at (0,0) {$z$};
\end{tikzpicture}
\\ 
\midrule
z &  1 & 1 &1 & 1 & 1 & z\\
\bottomrule
\end{array}
\]
\caption{The colored Boltzmann $\Delta$-weights with ${\color{red} c} > {\color{blue} c'}$ and ${\color{brown} d}$ being any color.}
\label{fig:colored_delta_weights}
\end{figure}

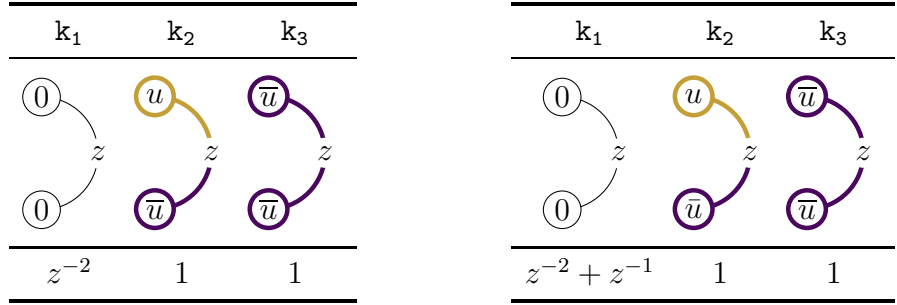
\begin{figure}
\[
\begin{array}{c@{\;\;}c@{\;\;}c}
\toprule
  \tt{k}_1&\tt{k}_2&\tt{k}_3\\ 
\midrule
\begin{tikzpicture}
\coordinate (b) at (0, .75);
\coordinate (c) at (.75, 0);
\coordinate (d) at (0, -.75);
\coordinate (aa) at (-.75,.5);
\coordinate (cc) at (.75,.5);
\draw (b) to[out=0,in=90] (c);
\draw (c) to[out=-90,in=0] (d);
\draw[fill=white] (b) circle (.25);
\path[fill=white] (c) circle (.2);
\draw[fill=white] (d) circle (.25);
\node at (0,1) { };
\node at (b) {$0$};
\node at (c) {$z$};
\node at (d) {$0$};
\end{tikzpicture}
& \begin{tikzpicture}
\coordinate (b) at (0, .75);
\coordinate (c) at (.75, 0);
\coordinate (d) at (0, -.75);
\coordinate (aa) at (-.75,.5);
\coordinate (cc) at (.75,.5);
\draw[line width=0.6mm, UQgold] (b) to[out=0,in=90] (c);
\draw[line width=0.6mm, UQpurple] (c) to[out=-90,in=0] (d);
\draw[line width=0.6mm, UQgold,fill=white] (b) circle (.25);
\path[fill=white] (c) circle (.2);
\draw[line width=0.6mm, UQpurple,fill=white] (d) circle (.25);
\node at (0,1) { };
\node at (b) {$u$};
\node at (c) {$z$};
\node at (d) {$\overline{u}$};
\end{tikzpicture}
&\begin{tikzpicture}
\coordinate (b) at (0, .75);
\coordinate (c) at (.75, 0);
\coordinate (d) at (0, -.75);
\coordinate (aa) at (-.75,.5);
\coordinate (cc) at (.75,.5);
\draw[line width=0.6mm, UQpurple] (b) to[out=0,in=90] (c);
\draw[line width=0.6mm, UQpurple] (c) to[out=-90,in=0] (d);
\draw[line width=0.6mm, UQpurple,fill=white] (b) circle (.25);
\path[fill=white] (c) circle (.2);
\draw[line width=0.6mm, UQpurple,fill=white] (d) circle (.25);
\node at (0,1) { };
\node at (b) {$\overline{u}$};
\node at (c) {$z$};
\node at (d) {$\overline{u}$};
\end{tikzpicture}
\\ 
\midrule
z^{-2} &  1 & 1 \\ 
\bottomrule
\end{array}
\quad \quad \quad \quad \quad 
\begin{array}{c@{\;\;}c@{\;\;}c}
\toprule
  \tt{k}_1&\tt{k}_2&\tt{k}_3\\ 
\midrule
\begin{tikzpicture}
\coordinate (b) at (0, .75);
\coordinate (c) at (.75, 0);
\coordinate (d) at (0, -.75);
\coordinate (aa) at (-.75,.5);
\coordinate (cc) at (.75,.5);
\draw (b) to[out=0,in=90] (c);
\draw (c) to[out=-90,in=0] (d);
\draw[fill=white] (b) circle (.25);
\path[fill=white] (c) circle (.2);
\draw[fill=white] (d) circle (.25);
\node at (0,1) { };
\node at (b) {$0$};
\node at (c) {$z$};
\node at (d) {$0$};
\end{tikzpicture}
& \begin{tikzpicture}
\coordinate (b) at (0, .75);
\coordinate (c) at (.75, 0);
\coordinate (d) at (0, -.75);
\coordinate (aa) at (-.75,.5);
\coordinate (cc) at (.75,.5);
\draw[line width=0.6mm, UQgold] (b) to[out=0,in=90] (c);
\draw[line width=0.6mm, UQpurple] (c) to[out=-90,in=0] (d);
\draw[line width=0.6mm, UQgold,fill=white] (b) circle (.25);
\path[fill=white] (c) circle (.2);
\draw[line width=0.6mm, UQpurple,fill=white] (d) circle (.25);
\node at (0,1) { };
\node at (b) {$u$};
\node at (c) {$z$};
\node at (d) {$\bar{u}$};
\end{tikzpicture}
&\begin{tikzpicture}
\coordinate (b) at (0, .75);
\coordinate (c) at (.75, 0);
\coordinate (d) at (0, -.75);
\coordinate (aa) at (-.75,.5);
\coordinate (cc) at (.75,.5);
\draw[line width=0.6mm, UQpurple] (b) to[out=0,in=90] (c);
\draw[line width=0.6mm, UQpurple] (c) to[out=-90,in=0] (d);
\draw[line width=0.6mm, UQpurple,fill=white] (b) circle (.25);
\path[fill=white] (c) circle (.2);T
\draw[line width=0.6mm, UQpurple,fill=white] (d) circle (.25);
\node at (0,1) { };
\node at (b) {$\overline{u}$};
\node at (c) {$z$};
\node at (d) {$\overline{u}$};
\end{tikzpicture}
\\ 
\midrule
z^{-2} + z^{-1} &  1 & 1 \\ 
\bottomrule
\end{array}
\]
\caption{On the left (resp.\ right) we have the colored $^{\Gamma}_{\Delta}$ (resp.~$^{\Delta}_{\Gamma}$) $K$-matrix weights for type $C$ (resp.~$B$) with ${\color{UQgold} u} > {\color{UQpurple} \overline{u}}$.}
\label{fig:colored_K_weights_BC}
\end{figure}

Let us now explain the model. 
Consider a rectangular grid with $2n$ horizontal lines which we number from top to bottom and $m$ vertical lines numbered from right to left as in Figure~\ref{fig:ground_state}.
We call the odd numbered lines $\Gamma$ and the even numbered ones $\Delta$.
The intersection of a vertical and horizontal line is called a \defn{vertex}.
On the right, we connect the $\Gamma$ line $2i-1$ to the $\Delta$ line $2i$ by a \defn{U-turn}.

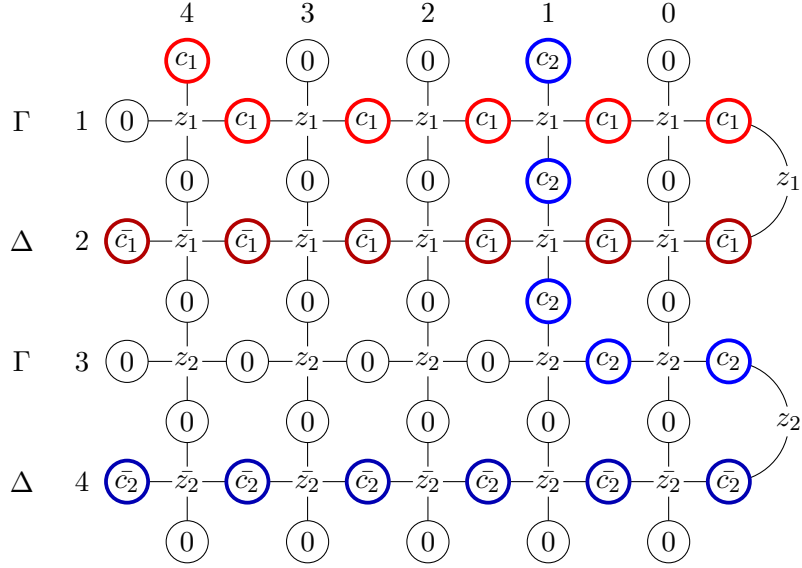
\begin{figure}
\begin{tikzpicture}[scale=0.80, font=\small]
    \foreach \y in {1,3,5,7}
        \draw (0,\y)--(10,\y);
    \foreach \x in {1,3,5,7,9}
      \draw (\x,0)--(\x,8);
    \draw (10,7) to[out=0,in=90] (11,6);
    \draw (11,6) to[out=-90,in=0] (10,5);
    \path[fill=white] (11,6) circle (.25);
    \node at (11,6) {$z_1$};
    \draw (10,3) to[out=0,in=90] (11,2);
    \draw (11,2) to[out=-90,in=0] (10,1);
    \path[fill=white] (11,2) circle (.25);
    \node at (11,2) {$z_2$};

    \foreach \x in {0,2,4,6,8,10}
    {
      \draw[fill=white] (\x,7) circle (.35);
      \draw[fill=white] (\x,5) circle (.35);
      \draw[fill=white] (\x,3) circle (.35);
      \draw[fill=white] (\x,1) circle (.35);
    }
    \foreach \x in {1,3,5,7,9}
    {
      \draw[fill=white] (\x,0) circle (.35);
      \draw[fill=white] (\x,2) circle (.35);
      \draw[fill=white] (\x,4) circle (.35);
      \draw[fill=white] (\x,6) circle (.35);
      \draw[fill=white] (\x,8) circle (.35);
      \path[fill=white] (\x,1) circle (.25);
      \node at (\x,1) {$\bar{z_2}$};
      \path[fill=white] (\x,3) circle (.25);
      \node at (\x,3) {$z_2$};
      \path[fill=white] (\x,5) circle (.25);
      \node at (\x,5) {$\bar{z_1}$};
      \path[fill=white] (\x,7) circle (.25);
      \node at (\x,7) {$z_1$};
    }

    \draw[line width=0.5mm,red,fill=white] (1,8) circle (.35); \node at (1,8) {$c_1$};
    \node at (3,8) {$0$};
    \node at (5,8) {$0$};
    \draw[line width=0.5mm,blue,fill=white] (7,8) circle (.35); \node at (7,8) {$c_2$};
    \node at (9,8) {$0$};
    \node at (1,6) {$0$};
    \node at (3,6) {$0$};
    \node at (5,6) {$0$};
    \draw[line width=0.5mm,blue,fill=white] (7,6) circle (.35); \node at (7,6) {$c_2$};
    \node at (9,6) {$0$};
    \node at (1,4) {$0$};
    \node at (3,4) {$0$};
    \node at (5,4) {$0$};
    \draw[line width=0.5mm,blue,fill=white] (7,4) circle (.35); \node at (7,4) {$c_2$};
    \node at (9,4) {$0$};
    \node at (1,2) {$0$};
    \node at (3,2) {$0$};
    \node at (5,2) {$0$};
    \node at (7,2) {$0$};
    \node at (9,2) {$0$};
    \node at (1,0) {$0$};
    \node at (3,0) {$0$};
    \node at (5,0) {$0$};
    \node at (7,0) {$0$};
    \node at (9,0) {$0$};
    \node at (0,7) {$0$};    
    \draw[line width=0.5mm,red,fill=white] (2,7) circle (.35); \node at (2,7) {$c_1$};
    \draw[line width=0.5mm,red,fill=white] (4,7) circle (.35); \node at (4,7) {$c_1$};
    \draw[line width=0.5mm,red,fill=white] (6,7) circle (.35); \node at (6,7) {$c_1$};
    \draw[line width=0.5mm,red,fill=white] (8,7) circle (.35); \node at (8,7) {$c_1$};
    \draw[line width=0.5mm,red,fill=white] (10,7) circle (.35); \node at (10,7) {$c_1$};
    \draw[line width=0.5mm,darkred,fill=white] (0,5) circle (.35); \node at (0,5) {$\bar{c_1}$};
    \draw[line width=0.5mm,darkred,fill=white] (2,5) circle (.35); \node at (2,5) {$\bar{c_1}$};
    \draw[line width=0.5mm,darkred,fill=white] (4,5) circle (.35); \node at (4,5) {$\bar{c_1}$};
    \draw[line width=0.5mm,darkred,fill=white] (6,5) circle (.35); \node at (6,5) {$\bar{c_1}$};
    \draw[line width=0.5mm,darkred,fill=white] (8,5) circle (.35); \node at (8,5) {$\bar{c_1}$};
    \draw[line width=0.5mm,darkred,fill=white] (10,5) circle (.35); \node at (10,5) {$\bar{c_1}$};
    \node at (0,3) {$0$};
    \node at (2,3) {$0$};
    \node at (4,3) {$0$};
    \node at (6,3) {$0$};
    \draw[line width=0.5mm,blue,fill=white] (8,3) circle (.35); \node at (8,3) {$c_2$};
    \draw[line width=0.5mm,blue,fill=white] (10,3) circle (.35); \node at (10,3) {$c_2$};
    \draw[line width=0.5mm,darkblue,fill=white] (0,1) circle (.35); \node at (0,1) {$\bar{c_2}$};
    \draw[line width=0.5mm,darkblue,fill=white] (2,1) circle (.35); \node at (2,1) {$\bar{c_2}$};
    \draw[line width=0.5mm,darkblue,fill=white] (4,1) circle (.35); \node at (4,1) {$\bar{c_2}$};
    \draw[line width=0.5mm,darkblue,fill=white] (6,1) circle (.35); \node at (6,1) {$\bar{c_2}$};
    \draw[line width=0.5mm,darkblue,fill=white] (8,1) circle (.35); \node at (8,1) {$\bar{c_2}$};
    \draw[line width=0.5mm,darkblue,fill=white] (10,1) circle (.35); \node at (10,1) {$\bar{c_2}$};
    \node at (1.00,8.8) {$ 4$};
    \node at (3.00,8.8) {$ 3$};
    \node at (5.00,8.8) {$ 2$};
    \node at (7.00,8.8) {$ 1$};
    \node at (9.00,8.8) {$ 0$};
    \node at (-.75,1) {$ 4$};
    \node at (-.75,3) {$ 3$};
    \node at (-.75,5) {$ 2$};
    \node at (-.75,7) {$ 1$};
    \node at (-1.75,1) {$ \Delta$};
    \node at (-1.75,3) {$ \Gamma$};
    \node at (-1.75,5) {$ \Delta$};
    \node at (-1.75,7) {$ \Gamma$};
\end{tikzpicture}
\caption{The unique admissible state in $\overline{\states}_{\lambda,w}$ for $\lambda = (3,1)$ and $w=1$. We use the convention $\bar{z}=z^{-1}$. The top boundary condition will consist of colors on columns $(4,1)=\lambda+\rho$.}
\label{fig:ground_state}
\end{figure}

Each $\Gamma$ line $2i-1$ is assigned the parameter $z_i \in \CC^*$ and $\Delta$ line $2i$ is assigned the parameter $z^{-1}_i \in \CC^*$.
We also assign the U-turn from line $2i-1$ to $2i$ the parameter $z_i$.
One can think of $\zz := (z_1,z_1^{-1},\dotsc,z_n,z_n^{-1})$ as living in the torus of $\Sp_{2n}(\CC)$ or $\SO_{2n+1}(\CC)$.
An \defn{interior edge} connects two vertices in the model, while an \defn{outer edge} (or a boundary edge) is attached to one vertex alone. 

To each edge we may assign a \defn{spin}, that is an element $c \in \cc \sqcup \{0\}$.
The \defn{Boltzmann weight} of a vertex (resp.\ U-turn) is a function that assigns a complex number to each assignment of spins to the edges of a vertex (resp.\ U-turn) that depends on the assigned parameter.
The collection of vertices (resp.\ U-turns) and their Boltzmann weights is called an \defn{$L$-matrix} (resp.\ \defn{$K$-matrix}).
The Boltzmann weights for $\Gamma$, $\Delta$, and U-turn vertices are given in Figures~\ref{fig:colored_gamma_weights},~\ref{fig:colored_delta_weights} and~\ref{fig:colored_K_weights_BC}, respectively.
Each weight that is not portrayed in the figures mentioned in this paragraph is considered to be $0$.
Both the $B$-model and the $C$-model use the same $\Gamma$ and $\Delta$ weights, while the U-turn weights are different.

Our system has fixed spins on the boundary that depend on $w \in W$ and $\lambda \in \Lambda$.
The bottom edges are labeled by $0$, the left $\Gamma$ edges are labeled by $0$.
The left $\Delta$ edges are labeled by $w w_0 \cc$ from top to bottom, and the top edges are labeled by $c_n, \dotsc, c_1$ (from right to left) in places $\lambda+\rho$, where $\rho = (n-1, n-2, \dotsc, 0)$.
The rest of the boundary edges are assigned spin $0$.
See Figure~\ref{fig:ground_state} for an example where $w=1$, $\lambda = (3,1)$ and $\lambda+\rho = (4,1)$. 
We denote such a model by $\overline{\states}^X_{\lambda,w}$, for $X \in \{B, C\}$.

An assignment of spins to the inner edges is called a \defn{state} of the system. 
The \defn{weight} of a state is the product over all vertices of the weights of each vertex.
A state is called \defn{admissible} if its weight is non-zero. 
We will often simply write $\overline{\states}_{\lambda,w} = \overline{\states}^X_{\lambda,w}$ since the states of the two models are the same and the Boltzmann weights only differ in $\tt{k}_1$.
The \defn{partition function} $ Z(\overline{\states}^X_{\lambda,w}; \zz)$ is the sum of the weights of the states over all states of the system with boundary conditions determined by $w$, $\lambda$, and the parameters $\zz$.

\begin{remark}
The $z + 1$ ratio between the $\tt{k}_1$ $K$-matrix entry in types $B$ and $C$ is exactly the ratio between the type $C_n$ and $B_n$ characters in rank $n = 1$.
\end{remark}

\begin{figure}
\[
\begin{array}{c@{\hspace{30pt}}c@{\hspace{30pt}}c@{\hspace{30pt}}c}
\toprule
\begin{tikzpicture}[scale=0.7]
\draw (0,0) to [out = 0, in = 180] (2,2);
\draw (0,2) to [out = 0, in = 180] (2,0);
\draw[fill=white] (0,0) circle (.35);
\draw[fill=white] (0,2) circle (.35);
\draw[fill=white] (2,0) circle (.35);
\draw[fill=white] (2,2) circle (.35);
\node at (0,0) {$0$};
\node at (0,2) {$0$};
\node at (2,2) {$0$};
\node at (2,0) {$0$};
\path[fill=white] (1,1) circle (.3);
\node at (1,1) {$z_i,z_j$};
\end{tikzpicture}&
\begin{tikzpicture}[scale=0.7]
\draw (0,0) to [out = 0, in = 180] (2,2);
\draw (0,2) to [out = 0, in = 180] (2,0);
\draw[fill=white] (0,0) circle (.35);
\draw[line width=0.5mm, brown, fill=white] (0,2) circle (.35);
\draw[line width=0.5mm, brown, fill=white] (2,2) circle (.35);
\draw[fill=white] (2,0) circle (.35);
\node at (0,0) {$0$};
\node at (0,2) {$d$};
\node at (2,2) {$d$};
\node at (2,0) {$0$};
\path[fill=white] (1,1) circle (.3);
\node at (1,1) {$z_i,z_j$};
\end{tikzpicture}&
\begin{tikzpicture}[scale=0.7]
\draw (0,0) to [out = 0, in = 180] (2,2);
\draw (0,2) to [out = 0, in = 180] (2,0);
\draw[line width=0.5mm, brown, fill=white] (0,0) circle (.35);
\draw[fill=white] (0,2) circle (.35);
\draw[line width=0.5mm, brown, fill=white] (2,2) circle (.35);
\draw[fill=white] (2,0) circle (.35);
\node at (0,0) {$d$};
\node at (0,2) {$0$};
\node at (2,2) {$d$};
\node at (2,0) {$0$};
\path[fill=white] (1,1) circle (.3);
\node at (1,1) {$z_i,z_j$};
\end{tikzpicture}&
\begin{tikzpicture}[scale=0.7]
\draw (0,0) to [out = 0, in = 180] (2,2);
\draw (0,2) to [out = 0, in = 180] (2,0);
\draw[line width=0.5mm, brown, fill=white] (0,0) circle (.35);
\draw[fill=white] (0,2) circle (.35);
\draw[fill=white] (2,2) circle (.35);
\draw[line width=0.5mm, brown, fill=white] (2,0) circle (.35);
\node at (0,0) {$d$};
\node at (0,2) {$0$};
\node at (2,2) {$0$};
\node at (2,0) {$d$};
\path[fill=white] (1,1) circle (.3);
\node at (1,1) {$z_i,z_j$};
\end{tikzpicture}\\
   \midrule
   z_j & z_j & z_i - z_j & z_i \\
   \midrule
\begin{tikzpicture}[scale=0.7]
\draw (0,0) to [out = 0, in = 180] (2,2);
\draw (0,2) to [out = 0, in = 180] (2,0);
\draw[line width=0.5mm, blue, fill=white] (0,0) circle (.35);
\draw[line width=0.5mm, red, fill=white] (0,2) circle (.35);
\draw[line width=0.5mm, red, fill=white] (2,2) circle (.35);
\draw[line width=0.5mm, blue, fill=white] (2,0) circle (.35);
\node at (0,0) {$c'$};
\node at (0,2) {$c$};
\node at (2,2) {$c$};
\node at (2,0) {$c'$};
\path[fill=white] (1,1) circle (.3);
\node at (1,1) {$z_i,z_j$};
\end{tikzpicture}&
\begin{tikzpicture}[scale=0.7]
\draw (0,0) to [out = 0, in = 180] (2,2);
\draw (0,2) to [out = 0, in = 180] (2,0);
\draw[line width=0.5mm, red, fill=white] (0,0) circle (.35);
\draw[line width=0.5mm, blue, fill=white] (0,2) circle (.35);
\draw[line width=0.5mm, blue, fill=white] (2,2) circle (.35);
\draw[line width=0.5mm, red, fill=white] (2,0) circle (.35);
\node at (0,0) {$c$};
\node at (0,2) {$c'$};
\node at (2,2) {$c'$};
\node at (2,0) {$c$};
\path[fill=white] (1,1) circle (.3);
\node at (1,1) {$z_i,z_j$};
\end{tikzpicture}&
\begin{tikzpicture}[scale=0.7]
\draw (0,0) to [out = 0, in = 180] (2,2);
\draw (0,2) to [out = 0, in = 180] (2,0);
\draw[line width=0.5mm, red, fill=white] (0,0) circle (.35);
\draw[line width=0.5mm, blue, fill=white] (0,2) circle (.35);
\draw[line width=0.5mm, red, fill=white] (2,2) circle (.35);
\draw[line width=0.5mm, blue, fill=white] (2,0) circle (.35);
\node at (0,0) {$c$};
\node at (0,2) {$c'$};
\node at (2,2) {$c$};
\node at (2,0) {$c'$};
\path[fill=white] (1,1) circle (.3);
\node at (1,1) {$z_i,z_j$};
\end{tikzpicture}&
\begin{tikzpicture}[scale=0.7]
\draw (0,0) to [out = 0, in = 180] (2,2);
\draw (0,2) to [out = 0, in = 180] (2,0);
\draw[line width=0.5mm, brown, fill=white] (0,0) circle (.35);
\draw[line width=0.5mm, brown, fill=white] (0,2) circle (.35);
\draw[line width=0.5mm, brown, fill=white] (2,2) circle (.35);
\draw[line width=0.5mm, brown, fill=white] (2,0) circle (.35);
\node at (0,0) {$d$};
\node at (0,2) {$d$};
\node at (2,2) {$d$};
\node at (2,0) {$d$};
\path[fill=white] (1,1) circle (.3);
\node at (1,1) {$z_i,z_j$};
\end{tikzpicture}\\
   \midrule
   z_i & z_j & z_i - z_j & z_i \\
   \bottomrule
\end{array}
\]
\caption{The colored $R^{\Gamma}_{\Gamma}$-matrix with ${\color{red} c} > {\color{blue} c'}$ and ${\color{brown} d}$ being any color.}
\label{fig:colored_R_matrix_GG}
\end{figure}


\begin{figure}
\[
\begin{array}{c@{\hspace{30pt}}c@{\hspace{30pt}}c@{\hspace{30pt}}c}
\toprule
\begin{tikzpicture}[scale=0.7]
\draw (0,0) to [out = 0, in = 180] (2,2);
\draw (0,2) to [out = 0, in = 180] (2,0);
\draw[fill=white] (0,0) circle (.35);
\draw[fill=white] (0,2) circle (.35);
\draw[fill=white] (2,0) circle (.35);
\draw[fill=white] (2,2) circle (.35);
\node at (0,0) {$0$};
\node at (0,2) {$0$};
\node at (2,2) {$0$};
\node at (2,0) {$0$};
\path[fill=white] (1,1) circle (.3);
\node at (1,1) {$z_i,z_j$};
\end{tikzpicture}&
\begin{tikzpicture}[scale=0.7]
\draw (0,0) to [out = 0, in = 180] (2,2);
\draw (0,2) to [out = 0, in = 180] (2,0);
\draw[fill=white] (0,0) circle (.35);
\draw[line width=0.5mm, brown, fill=white] (0,2) circle (.35);
\draw[line width=0.5mm, brown, fill=white] (2,2) circle (.35);
\draw[fill=white] (2,0) circle (.35);
\node at (0,0) {$0$};
\node at (0,2) {$d$};
\node at (2,2) {$d$};
\node at (2,0) {$0$};
\path[fill=white] (1,1) circle (.3);
\node at (1,1) {$z_i,z_j$};
\end{tikzpicture}&
\begin{tikzpicture}[scale=0.7]
\draw (0,0) to [out = 0, in = 180] (2,2);
\draw (0,2) to [out = 0, in = 180] (2,0);
\draw[fill=white] (0,0) circle (.35);
\draw[line width=0.5mm, brown, fill=white] (0,2) circle (.35);
\draw[fill=white] (2,2) circle (.35);
\draw[line width=0.5mm, brown, fill=white] (2,0) circle (.35);
\node at (0,0) {$0$};
\node at (0,2) {$d$};
\node at (2,2) {$0$};
\node at (2,0) {$d$};
\path[fill=white] (1,1) circle (.3);
\node at (1,1) {$z_i,z_j$};
\end{tikzpicture}&
\begin{tikzpicture}[scale=0.7]
\draw (0,0) to [out = 0, in = 180] (2,2);
\draw (0,2) to [out = 0, in = 180] (2,0);
\draw[line width=0.5mm, brown, fill=white] (0,0) circle (.35);
\draw[fill=white] (0,2) circle (.35);
\draw[fill=white] (2,2) circle (.35);
\draw[line width=0.5mm, brown, fill=white] (2,0) circle (.35);
\node at (0,0) {$d$};
\node at (0,2) {$0$};
\node at (2,2) {$0$};
\node at (2,0) {$d$};
\path[fill=white] (1,1) circle (.3);
\node at (1,1) {$z_i,z_j$};
\end{tikzpicture}\\
   \midrule
   z_j & z_i & z_i - z_j & z_j \\
   \midrule
\begin{tikzpicture}[scale=0.7]
\draw (0,0) to [out = 0, in = 180] (2,2);
\draw (0,2) to [out = 0, in = 180] (2,0);
\draw[line width=0.5mm, blue, fill=white] (0,0) circle (.35);
\draw[line width=0.5mm, red, fill=white] (0,2) circle (.35);
\draw[line width=0.5mm, red, fill=white] (2,2) circle (.35);
\draw[line width=0.5mm, blue, fill=white] (2,0) circle (.35);
\node at (0,0) {$c'$};
\node at (0,2) {$c$};
\node at (2,2) {$c$};
\node at (2,0) {$c'$};
\path[fill=white] (1,1) circle (.3);
\node at (1,1) {$z_i,z_j$};
\end{tikzpicture}&
\begin{tikzpicture}[scale=0.7]
\draw (0,0) to [out = 0, in = 180] (2,2);
\draw (0,2) to [out = 0, in = 180] (2,0);
\draw[line width=0.5mm, red, fill=white] (0,0) circle (.35);
\draw[line width=0.5mm, blue, fill=white] (0,2) circle (.35);
\draw[line width=0.5mm, blue, fill=white] (2,2) circle (.35);
\draw[line width=0.5mm, red, fill=white] (2,0) circle (.35);
\node at (0,0) {$c$};
\node at (0,2) {$c'$};
\node at (2,2) {$c'$};
\node at (2,0) {$c$};
\path[fill=white] (1,1) circle (.3);
\node at (1,1) {$z_i,z_j$};
\end{tikzpicture}&
\begin{tikzpicture}[scale=0.7]
\draw (0,0) to [out = 0, in = 180] (2,2);
\draw (0,2) to [out = 0, in = 180] (2,0);
\draw[line width=0.5mm, red, fill=white] (0,0) circle (.35);
\draw[line width=0.5mm, blue, fill=white] (0,2) circle (.35);
\draw[line width=0.5mm, red, fill=white] (2,2) circle (.35);
\draw[line width=0.5mm, blue, fill=white] (2,0) circle (.35);
\node at (0,0) {$c$};
\node at (0,2) {$c'$};
\node at (2,2) {$c$};
\node at (2,0) {$c'$};
\path[fill=white] (1,1) circle (.3);
\node at (1,1) {$z_i,z_j$};
\end{tikzpicture}&
\begin{tikzpicture}[scale=0.7]
\draw (0,0) to [out = 0, in = 180] (2,2);
\draw (0,2) to [out = 0, in = 180] (2,0);
\draw[line width=0.5mm, brown, fill=white] (0,0) circle (.35);
\draw[line width=0.5mm, brown, fill=white] (0,2) circle (.35);
\draw[line width=0.5mm, brown, fill=white] (2,2) circle (.35);
\draw[line width=0.5mm, brown, fill=white] (2,0) circle (.35);
\node at (0,0) {$d$};
\node at (0,2) {$d$};
\node at (2,2) {$d$};
\node at (2,0) {$d$};
\path[fill=white] (1,1) circle (.3);
\node at (1,1) {$z_i,z_j$};
\end{tikzpicture}\\
   \midrule
   z_j & z_i & z_i - z_j & z_i \\
   \bottomrule
\end{array}
\]
\caption{The colored $R^{\Delta}_{\Delta}$-matrix with ${\color{red} c} > {\color{blue} c'}$ and ${\color{brown} d}$ being any color.}
\label{fig:colored_R_matrix_DD}
\end{figure}


\begin{figure}
\[
\begin{array}{c@{\hspace{30pt}}c@{\hspace{30pt}}c@{\hspace{30pt}}c}
\toprule
\begin{tikzpicture}[scale=0.7]
\draw (0,0) to [out = 0, in = 180] (2,2);
\draw (0,2) to [out = 0, in = 180] (2,0);
\draw[fill=white] (0,0) circle (.35);
\draw[fill=white] (0,2) circle (.35);
\draw[fill=white] (2,0) circle (.35);
\draw[fill=white] (2,2) circle (.35);
\node at (0,0) {$0$};
\node at (0,2) {$0$};
\node at (2,2) {$0$};
\node at (2,0) {$0$};
\path[fill=white] (1,1) circle (.3);
\node at (1,1) {$z_i,z_j$};
\end{tikzpicture}&
\begin{tikzpicture}[scale=0.7]
\draw (0,0) to [out = 0, in = 180] (2,2);
\draw (0,2) to [out = 0, in = 180] (2,0);
\draw[fill=white] (0,0) circle (.35);
\draw[fill=white] (0,2) circle (.35);
\draw[line width=0.5mm, brown, fill=white] (2,2) circle (.35);
\draw[line width=0.5mm, brown, fill=white] (2,0) circle (.35);
\node at (0,0) {$0$};
\node at (0,2) {$0$};
\node at (2,2) {$d$};
\node at (2,0) {$d$};
\path[fill=white] (1,1) circle (.3);
\node at (1,1) {$z_i,z_j$};
\end{tikzpicture}&
\begin{tikzpicture}[scale=0.7]
\draw (0,0) to [out = 0, in = 180] (2,2);
\draw (0,2) to [out = 0, in = 180] (2,0);
\draw[line width=0.5mm, brown, fill=white] (0,0) circle (.35);
\draw[fill=white] (0,2) circle (.35);
\draw[line width=0.5mm, brown, fill=white] (2,2) circle (.35);
\draw[fill=white] (2,0) circle (.35);
\node at (0,0) {$d$};
\node at (0,2) {$0$};
\node at (2,2) {$d$};
\node at (2,0) {$0$};
\path[fill=white] (1,1) circle (.3);
\node at (1,1) {$z_i,z_j$};
\end{tikzpicture}&
\begin{tikzpicture}[scale=0.7]
\draw (0,0) to [out = 0, in = 180] (2,2);
\draw (0,2) to [out = 0, in = 180] (2,0);
\draw[fill=white] (0,0) circle (.35);
\draw[line width=0.5mm, brown, fill=white] (0,2) circle (.35);
\draw[fill=white] (2,2) circle (.35);
\draw[line width=0.5mm, brown, fill=white] (2,0) circle (.35);
\node at (0,0) {$0$};
\node at (0,2) {$d$};
\node at (2,2) {$0$};
\node at (2,0) {$d$};
\path[fill=white] (1,1) circle (.3);
\node at (1,1) {$z_i,z_j$};
\end{tikzpicture}\\
   \midrule
   z_i - z_j & z_i & z_j & z_j \\
   \midrule
\begin{tikzpicture}[scale=0.7]
\draw (0,0) to [out = 0, in = 180] (2,2);
\draw (0,2) to [out = 0, in = 180] (2,0);
\draw[line width=0.5mm, brown, fill=white] (0,0) circle (.35);
\draw[line width=0.5mm, brown, fill=white] (0,2) circle (.35);
\draw[fill=white] (2,2) circle (.35);
\draw[fill=white] (2,0) circle (.35);
\node at (0,0) {$d$};
\node at (0,2) {$d$};
\node at (2,2) {$0$};
\node at (2,0) {$0$};
\path[fill=white] (1,1) circle (.3);
\node at (1,1) {$z_i,z_j$};
\end{tikzpicture}&
\begin{tikzpicture}[scale=0.7]
\draw (0,0) to [out = 0, in = 180] (2,2);
\draw (0,2) to [out = 0, in = 180] (2,0);
\draw[line width=0.5mm, blue, fill=white] (0,0) circle (.35);
\draw[line width=0.5mm, blue, fill=white] (0,2) circle (.35);
\draw[line width=0.5mm, red, fill=white] (2,2) circle (.35);
\draw[line width=0.5mm, red, fill=white] (2,0) circle (.35);
\node at (0,0) {$c'$};
\node at (0,2) {$c'$};
\node at (2,2) {$c$};
\node at (2,0) {$c$};
\path[fill=white] (1,1) circle (.3);
\node at (1,1) {$z_i,z_j$};
\end{tikzpicture}&
\begin{tikzpicture}[scale=0.7]
\draw (0,0) to [out = 0, in = 180] (2,2);
\draw (0,2) to [out = 0, in = 180] (2,0);
\draw[line width=0.5mm, blue, fill=white] (0,0) circle (.35);
\draw[line width=0.5mm, red, fill=white] (0,2) circle (.35);
\draw[line width=0.5mm, blue, fill=white] (2,2) circle (.35);
\draw[line width=0.5mm, red, fill=white] (2,0) circle (.35);
\node at (0,0) {$c'$};
\node at (0,2) {$c$};
\node at (2,2) {$c'$};
\node at (2,0) {$c$};
\path[fill=white] (1,1) circle (.3);
\node at (1,1) {$z_i,z_j$};
\end{tikzpicture}&
\begin{tikzpicture}[scale=0.7]
\draw (0,0) to [out = 0, in = 180] (2,2);
\draw (0,2) to [out = 0, in = 180] (2,0);
\draw[line width=0.5mm, brown, fill=white] (0,0) circle (.35);
\draw[line width=0.5mm, brown, fill=white] (0,2) circle (.35);
\draw[line width=0.5mm, brown, fill=white] (2,2) circle (.35);
\draw[line width=0.5mm, brown, fill=white] (2,0) circle (.35);
\node at (0,0) {$d$};
\node at (0,2) {$d$};
\node at (2,2) {$d$};
\node at (2,0) {$d$};
\path[fill=white] (1,1) circle (.3);
\node at (1,1) {$z_i,z_j$};
\end{tikzpicture}\\
   \midrule
   z_j & z_j & z_j & z_j \\
   \bottomrule
\end{array}
\]
\caption{The colored $R^{\Delta}_{\Gamma}$-matrix with ${\color{red} c} > {\color{blue} c'}$ and ${\color{brown} d}$ being any color.}
\label{fig:colored_R_matrix_DG}
\end{figure}


\begin{figure}
\[
\begin{array}{c@{\hspace{30pt}}c@{\hspace{30pt}}c@{\hspace{30pt}}c}
\toprule
\begin{tikzpicture}[scale=0.7]
\draw (0,0) to [out = 0, in = 180] (2,2);
\draw (0,2) to [out = 0, in = 180] (2,0);
\draw[fill=white] (0,0) circle (.35);
\draw[fill=white] (0,2) circle (.35);
\draw[fill=white] (2,0) circle (.35);
\draw[fill=white] (2,2) circle (.35);
\node at (0,0) {$0$};
\node at (0,2) {$0$};
\node at (2,2) {$0$};
\node at (2,0) {$0$};
\path[fill=white] (1,1) circle (.3);
\node at (1,1) {$z_i,z_j$};
\end{tikzpicture}&
\begin{tikzpicture}[scale=0.7]
\draw (0,0) to [out = 0, in = 180] (2,2);
\draw (0,2) to [out = 0, in = 180] (2,0);
\draw[fill=white] (0,0) circle (.35);
\draw[fill=white] (0,2) circle (.35);
\draw[line width=0.5mm, brown, fill=white] (2,2) circle (.35);
\draw[line width=0.5mm, brown, fill=white] (2,0) circle (.35);
\node at (0,0) {$0$};
\node at (0,2) {$0$};
\node at (2,2) {$d$};
\node at (2,0) {$d$};
\path[fill=white] (1,1) circle (.3);
\node at (1,1) {$z_i,z_j$};
\end{tikzpicture}&
\begin{tikzpicture}[scale=0.7]
\draw (0,0) to [out = 0, in = 180] (2,2);
\draw (0,2) to [out = 0, in = 180] (2,0);
\draw[line width=0.5mm, brown, fill=white] (0,0) circle (.35);
\draw[fill=white] (0,2) circle (.35);
\draw[line width=0.5mm, brown, fill=white] (2,2) circle (.35);
\draw[fill=white] (2,0) circle (.35);
\node at (0,0) {$d$};
\node at (0,2) {$0$};
\node at (2,2) {$d$};
\node at (2,0) {$0$};
\path[fill=white] (1,1) circle (.3);
\node at (1,1) {$z_i,z_j$};
\end{tikzpicture}&
\begin{tikzpicture}[scale=0.7]
\draw (0,0) to [out = 0, in = 180] (2,2);
\draw (0,2) to [out = 0, in = 180] (2,0);
\draw[fill=white] (0,0) circle (.35);
\draw[line width=0.5mm, brown, fill=white] (0,2) circle (.35);
\draw[fill=white] (2,2) circle (.35);
\draw[line width=0.5mm, brown, fill=white] (2,0) circle (.35);
\node at (0,0) {$0$};
\node at (0,2) {$d$};
\node at (2,2) {$0$};
\node at (2,0) {$d$};
\path[fill=white] (1,1) circle (.3);
\node at (1,1) {$z_i,z_j$};
\end{tikzpicture}\\
   \midrule
   -z_i & z_j & z_i & z_i \\
   \midrule
\begin{tikzpicture}[scale=0.7]
\draw (0,0) to [out = 0, in = 180] (2,2);
\draw (0,2) to [out = 0, in = 180] (2,0);
\draw[line width=0.5mm, brown, fill=white] (0,0) circle (.35);
\draw[line width=0.5mm, brown, fill=white] (0,2) circle (.35);
\draw[fill=white] (2,2) circle (.35);
\draw[fill=white] (2,0) circle (.35);
\node at (0,0) {$d$};
\node at (0,2) {$d$};
\node at (2,2) {$0$};
\node at (2,0) {$0$};
\path[fill=white] (1,1) circle (.3);
\node at (1,1) {$z_i,z_j$};
\end{tikzpicture}&
\begin{tikzpicture}[scale=0.7]
\draw (0,0) to [out = 0, in = 180] (2,2);
\draw (0,2) to [out = 0, in = 180] (2,0);
\draw[line width=0.5mm, blue, fill=white] (0,0) circle (.35);
\draw[line width=0.5mm, blue, fill=white] (0,2) circle (.35);
\draw[line width=0.5mm, red, fill=white] (2,2) circle (.35);
\draw[line width=0.5mm, red, fill=white] (2,0) circle (.35);
\node at (0,0) {$c'$};
\node at (0,2) {$c'$};
\node at (2,2) {$c$};
\node at (2,0) {$c$};
\path[fill=white] (1,1) circle (.3);
\node at (1,1) {$z_i,z_j$};
\end{tikzpicture}&
\begin{tikzpicture}[scale=0.7]
\draw (0,0) to [out = 0, in = 180] (2,2);
\draw (0,2) to [out = 0, in = 180] (2,0);
\draw[line width=0.5mm, blue, fill=white] (0,0) circle (.35);
\draw[line width=0.5mm, red, fill=white] (0,2) circle (.35);
\draw[line width=0.5mm, blue, fill=white] (2,2) circle (.35);
\draw[line width=0.5mm, red, fill=white] (2,0) circle (.35);
\node at (0,0) {$c'$};
\node at (0,2) {$c$};
\node at (2,2) {$c'$};
\node at (2,0) {$c$};
\path[fill=white] (1,1) circle (.3);
\node at (1,1) {$z_i,z_j$};
\end{tikzpicture}&
\begin{tikzpicture}[scale=0.7]
\draw (0,0) to [out = 0, in = 180] (2,2);
\draw (0,2) to [out = 0, in = 180] (2,0);
\draw[line width=0.5mm, brown, fill=white] (0,0) circle (.35);
\draw[line width=0.5mm, brown, fill=white] (0,2) circle (.35);
\draw[line width=0.5mm, brown, fill=white] (2,2) circle (.35);
\draw[line width=0.5mm, brown, fill=white] (2,0) circle (.35);
\node at (0,0) {$d$};
\node at (0,2) {$d$};
\node at (2,2) {$d$};
\node at (2,0) {$d$};
\path[fill=white] (1,1) circle (.3);
\node at (1,1) {$z_i,z_j$};
\end{tikzpicture}\\
   \midrule
   z_i & z_i (*)& z_i - z_j (*) & z_i - z_j \\
   \bottomrule
\end{array}
\]
\caption{The colored $R^{\Gamma}_{\Delta}$-matrix with ${\color{red} c} > {\color{blue} c'}$ and ${\color{brown} d}$ being any color.}
\label{fig:colored_R_matrix_GD}
\end{figure}

A lattice model is called \defn{solvable} or \defn{integrable} if it there exists a full set of solutions of the Yang--Baxter equation and its generalizations that enable one to derive functional equations for the partition function that can be used to characterize it.
For example, the model in~\cite{Ivanov12} is integrable because of the existence of four $R$-matrices, called $R^{\Gamma}_{\Gamma}$, $R^{\Delta}_{\Gamma}$, $R^{\Gamma}_{\Delta}$, and $R^{\Delta}_{\Delta}$ that satisfy the appropriate Yang--Baxter equations and reflection equations. 

Our model is not integrable in this sense, but it is close. 
We produce three $R$-matrices $R^\Gamma_\Gamma$, $R^\Delta_\Delta$, and $R^\Delta_\Gamma$ that are given in Figures~\ref{fig:colored_R_matrix_GG},~\ref{fig:colored_R_matrix_DD}, and ~\ref{fig:colored_R_matrix_DG}, respectively. 
These $R$-matrices satisfy the Yang--Baxter equation with the corresponding $L$-matrices as explained in Proposition~\ref{prop:YBE}. 
However, it can be shown, computationally, that there is no solution for the Yang--Baxter equation corresponding to $^{\Gamma}_{\Delta}$.
The problem, compared to the uncolored setting discussed in~\cite{Ivanov12} where such a solution exists, is that certain colored loops can be formed inside one side of Equation~\eqref{eq:RLL_relation}.
This then ends up multiplying that side's partition function by the total number of colors, which is $2n$, whereas the other side does not depend on $n$.
Hence, the two partition functions cannot be equal.

We do however produce a fourth $R$-matrix called $R^\Gamma_\Delta$ in Figure~\ref{fig:colored_R_matrix_GD} that is partly determined.
This means that the weights marked with $(*)$ in Figure~\ref{fig:colored_R_matrix_GD} are free, so they can be changed and this does not affect our results.
Yet, we do stress that no matter how you change them, the corresponding Yang--Baxter equation will still not be satisfied, including changing the allowed colorings (such that the colors are preserved).
Given these four $R$-matrices satisfying a total of three Yang--Baxter equations, we prove in Section~\ref{sec:billiards} a functional equation for the partition function for each of the simple reflections $s_i$, for $i < n$. 
The method of proof is by a modified version of the train argument applied to U-turn lattice model; the modification is technical and needed as the fourth Yang--Baxter equation does not have a solution. 
In Section~\ref{sec:ichthyology}, we then prove certain modified fish equations which are used to show a functional equation for the last remaining simple reflection. 
Our model therefore lacks a solution for the Yang--Baxter equations, but can still be studied via modified versions of the originally tools used to study solvable lattice models.
We shall call such a model \defn{quasi-solvable}.

\begin{proposition}
\label{prop:YBE}
The $R^{\Gamma}_{\Gamma}$-matrix, $R^{\Delta}_{\Delta}$-matrix, or $R^{\Delta}_{\Gamma}$-matrix satisfy the corresponding \defn{Yang--Baxter equation}:
The partition function of the following two models are equal for any boundary conditions $a,b,c,d,e,f \in \cc \sqcup \{0\}$:
\begin{equation}
\label{eq:RLL_relation}
\begin{tikzpicture}[baseline=(current bounding box.center),scale=0.8]
  \draw (0,1) to [out = 0, in = 180] (2,3) to (4,3);
  \draw (0,3) to [out = 0, in = 180] (2,1) to (4,1);
  \draw (3,0) to (3,4);
  \draw[fill=white] (0,1) circle (.3);
  \draw[fill=white] (0,3) circle (.3);
  \draw[fill=white] (3,4) circle (.3);
  \draw[fill=white] (4,3) circle (.3);
  \draw[fill=white] (4,1) circle (.3);
  \draw[fill=white] (3,0) circle (.3);
  \node at (0,1) {$a$};
  \node at (0,3) {$b$};
  \node at (3,4) {$c$};
  \node at (4,3) {$d$};
  \node at (4,1) {$e$};
  \node at (3,0) {$f$};
\path[fill=white] (3,3) circle (.3);
\node at (3,3) {\scriptsize $z_i$};
\path[fill=white] (3,1) circle (.3);
\node at (3,1) {\scriptsize $z_j$};
\path[fill=white] (1,2) circle (.3);
\node at (1,2) {\scriptsize $z_i,z_j$};
\end{tikzpicture}
\hspace{60pt}
\begin{tikzpicture}[baseline=(current bounding box.center),scale=0.8]
  \draw (0,1) to (2,1) to [out = 0, in = 180] (4,3);
  \draw (0,3) to (2,3) to [out = 0, in = 180] (4,1);
  \draw (1,0) to (1,4);
  \draw[fill=white] (0,1) circle (.3);
  \draw[fill=white] (0,3) circle (.3);
  \draw[fill=white] (1,4) circle (.3);
  \draw[fill=white] (4,3) circle (.3);
  \draw[fill=white] (4,1) circle (.3);
  \draw[fill=white] (1,0) circle (.3);
  \node at (0,1) {$a$};
  \node at (0,3) {$b$};
  \node at (1,4) {$c$};
  \node at (4,3) {$d$};
  \node at (4,1) {$e$};
  \node at (1,0) {$f$};
\path[fill=white] (1,3) circle (.3);
\node at (1,3) {\scriptsize $z_j$};
\path[fill=white] (1,1) circle (.3);
\node at (1,1) {\scriptsize $z_i$};
\path[fill=white] (3,2) circle (.3);
\node at (3,2) {\scriptsize $z_i,z_j$};
\end{tikzpicture}
\end{equation}
where the $z_i,z_j$ weights are $R^{\Phi}_{\Theta}$-weights, the $z_i$-weights are $\Phi$-weights and the $z_j$-weights are $\Theta$-weights for $\Phi\Theta \in \{\Gamma\Gamma, \Delta\Gamma, \Delta\Gamma\}$.
\end{proposition}

\begin{proof}
Since the $R$-matrix and the $L$-matrices preserves the spins and no (colored) loops can be formed, we only have to check this statement for at most $4$ different colors.
Hence this is a finite computation that can be done by, \textit{e.g.}, \textsc{SageMath}~\cite{sage}. The \textsc{SageMath} used to perform this computation is given in the Appendix of this paper.
\end{proof}

This model also generally satisfies the \defn{reflection equation}.

\begin{proposition}\label{prop:reflection}
For any fixed boundary condition $a,b,c,d \in \cc \sqcup \{0\}$, the partition function of the model on left
\[
\begin{tikzpicture}[baseline=0, scale=0.7]
  \draw (-2,1) to [out=0, in=180] (0,3) to [out=0, in=180] (2,5) to [out=-40, in=40] (3,6) to [out=-40, in=40] (2,7) to (-2,7);
  \draw (-2,3) to [out=0, in=180] (0,1) to (2,1) to [out=-40, in=40] (3,2) to [out=-40, in=40] (2,3) to [out=180, in=0] (0,5) to (-2,5);
  \foreach \x in {-2,0,2} {
    \foreach \y in {1,3,5,7}
      \draw[fill=white] (\x,\y) circle (.3);
  }
  \node at (-3,3) {$\Gamma$};
  \node at (-3,7) {$\Gamma$};
  \node at (-3,1) {$\Delta$};
  \node at (-3,5) {$\Delta$};
  \node at (-2,7) {$a$};
  \node at (-2,5) {$b$};
  \node at (0,7) {$a$};
  \node at (0,5) {$b$};
  \node at (2,7) {$a$};
  \node at (-2,3) {$c$};
  \node at (-2,1) {$d$};
  \foreach \x/\y in {0/1, 0/3, 2/1, 2/3, 2/5}
    \node at (\x,\y) {$\ast$};
  \path[fill=white] (3,2) circle (.35);
  \node at (3,2) {\scriptsize $z_j$};
  \path[fill=white] (3,6) circle (.35);
  \node at (3,6) {\scriptsize $z_i$};
  \path[fill=white] (1,4) circle (.35);
  \node at (1,4) {\scriptsize $z_i^{-1},z_j$};
  \path[fill=white] (-1,2) circle (.35);
  \node at (-1,2) {\scriptsize $z_i^{-1},z_j^{-1}$};
\end{tikzpicture}
\hspace{60pt}
\begin{tikzpicture}[baseline=-1.4cm, scale=0.7]
  \draw (-2,-1) to (2,-1) to [out=-40, in=40] (3,0) to [out=-40, in=40] (2,1) to [out=180, in=0] (0,3) to [out=180, in=0] (-2,5);
  \draw (-2,1) to (0,1) to [out=0, in=180] (2,3) to [out=-40, in=40] (3,4) to [out=-40, in=40] (2,5) to (0,5) to [out=180, in=0] (-2,3);
  \foreach \x in {-2,0,2} {
    \foreach \y in {-1,1,3,5}
      \draw[fill=white] (\x,\y) circle (.3);
  }
  \node at (-3,1) {$\Gamma$};
  \node at (-3,5) {$\Gamma$};
  \node at (-3,-1) {$\Delta$};
  \node at (-3,3) {$\Delta$};
  \node at (-2,5) {$a$};
  \node at (-2,3) {$b$};
  \node at (-2,-1) {$d$};
  \node at (-2,1) {$c$};
  \node at (0,-1) {$d$};
  \node at (0,1) {$c$};  
  \node at (2,-1) {$d$};
  \foreach \x/\y in {0/3, 0/5, 2/1, 2/3, 2/5}
    \node at (\x,\y) {$\ast$};
  \path[fill=white] (3,0) circle (.35);
  \node at (3,0) {\scriptsize $z_i$};
  \path[fill=white] (3,4) circle (.35);
  \node at (3,4) {\scriptsize $z_j$};
  \path[fill=white] (-1,4) circle (.35);
  \node at (-1,4) {\scriptsize $z_j,z_i$};
  \path[fill=white] (1,2) circle (.35);
  \node at (1,2) {\scriptsize $z_j^{-1},z_i$};
\end{tikzpicture}
\]
equals to the partition function on the right times $\alpha = z_i^{-2}$. 
\end{proposition}

\begin{proof}
Since the $R$-matrices preserve the colors, we can restrict to the case when
\[
a,b,c,d \in \{0\} \sqcup \{ u > u' > \overline{u}' > \overline{u} \}\]
Therefore, this is also a finite computation (take $u' =1$ and $u = 2$) that can be done by, \textit{e.g.},~\textsc{SageMath}.
We can also verify this by hand as follows.
By considering the nonzero $R$-matrix and $K$-matrix entries, we can reduce it to the following cases for $(a,b,c,d)$ that result in nonzero partition functions:
\begin{align*}
& (0, 0, 0, 0)
&
& (t, 0, \overline{u}, 0),
&
& (t, 0, 0, \overline{u}),
&
& (0, t, \overline{u}, 0),
&
& (0, t, 0, \overline{u}),
\\
& (t, t, \overline{u}, \overline{u}),
&
& (\overline{u}, u, \overline{u}, \overline{u}),
&
& (t', u, \overline{u}', \overline{u}),
&
& (\overline{u}, u', \overline{u}, \overline{u}'),
&
& (\overline{u}, \overline{u}', \overline{u}, \overline{u}'),
\end{align*}
where $t = u, \overline{u}$ and $t' = u', \overline{u}'$.
In each of these cases, there is precise one state for each model, and so the claim follows by direct computation.
\end{proof}

We will use the so-called \defn{unitary equation} to describe what happens when we uncross two strands.
More precisely, we show that the partition function of the model on the left
\begin{equation}
\label{eq:unitary_model}
\begin{tikzpicture}[scale=0.7,baseline=16]
\draw (0,0) to [out = 0, in = 180] (2,2) to [out = 0, in = 180] (4,0);
\draw (0,2) to [out = 0, in = 180] (2,0) to [out = 0, in = 180] (4,2);
\draw[fill=white] (0,0) circle (.35);
\draw[fill=white] (0,2) circle (.35);
\draw[fill=white] (2,0) circle (.35);
\draw[fill=white] (2,2) circle (.35);
\draw[fill=white] (4,0) circle (.35);
\draw[fill=white] (4,2) circle (.35);
\node at (0,0) {$b$};
\node at (0,2) {$a$};
\node at (2,2) {$\ast$};
\node at (2,0) {$\ast$};
\node at (4,2) {$a$};
\node at (4,0) {$b$};
\path[fill=white] (1,1) circle (.3);
\node at (1,1) {\scriptsize $z_i,z_j$};
\path[fill=white] (3,1) circle (.3);
\node at (3,1) {\scriptsize $z_j,z_i$};
\end{tikzpicture}
\qquad\qquad
\begin{tikzpicture}[scale=0.7,baseline=16]
\draw (0,0) to (2,0);
\draw (0,2) to (2,2);
\draw[fill=white] (0,0) circle (.35);
\draw[fill=white] (0,2) circle (.35);
\draw[fill=white] (2,0) circle (.35);
\draw[fill=white] (2,2) circle (.35);
\node at (0,0) {$b$};
\node at (0,2) {$a$};
\node at (2,2) {$a$};
\node at (2,0) {$b$};
\end{tikzpicture}
\end{equation}
is simply a fixed scalar value $\beta$ independent of the boundary condition $a, b \in \{0, c_1, \dotsc, c_k\}$ times the partition function on the right, which we set to be $1$ by definition. 

\begin{proposition}
\label{prop:unitary}
The partition function of the model on the left in~\eqref{eq:unitary_model} with both of the $R$-matrices being either $R^{\Gamma}_{\Gamma}$ or $R^{\Delta}_{\Delta}$ is equal to $\beta = z_i z_j$.
\end{proposition}

\begin{proof}
Note that we can restrict this to $a,b \in \{0 < c < c'\}$ since colors are preserved by the $R$-matrices.
Thus the claim is a straightforward and follows from a computation over all possible boundary conditions.
\end{proof}

\begin{proposition}
For a state in $\overline{\states}_{\lambda, w}$, the vertices $\tt{a}_2^{\dagger}$ and $\tt{k}_2$ correspond to inversions in $w_0 w$, the number of which equals $\ell(w_0 w)$.
\end{proposition}

\begin{proof}
This can be shown by a straightforward induction on $\ell(w)$ using the boundary conditions and vertices of the lattice model.
\end{proof}

\subsection{Billiards}\label{sec:billiards}

Our first goal is to prove the following ``type A'' functional equation for the partition function.

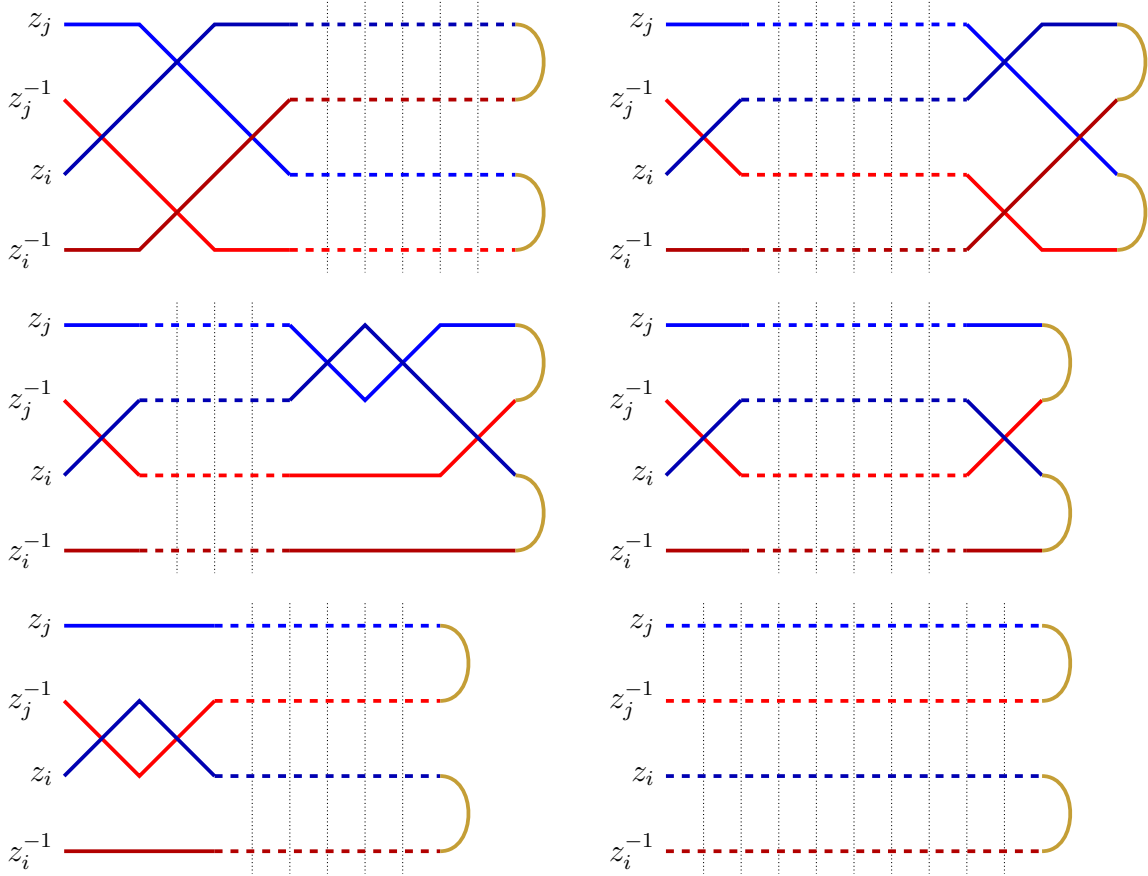
\begin{figure}
\[
\begin{tikzpicture}
  \begin{scope}
  \draw node[anchor=east] at (-2,4) {$z_j$};
  \draw node[anchor=east] at (-2,3) {$z_j^{-1}$};
  \draw node[anchor=east] at (-2,2) {$z_i$};
  \draw node[anchor=east] at (-2,1) {$z_i^{-1}$};
  \foreach \x in {1.5,2,...,3.5}
    \draw[densely dotted] (\x, 4.3) -- (\x,1-0.3);
  \draw[-, line width=0.5mm, blue] (-2,4) -- (-1,4) -- (1,2);
  \draw[-, line width=0.5mm, red] (-2,3) -- (0,1) -- (1,1);
  \draw[-, line width=0.5mm, darkblue] (-2,2) -- (0,4) -- (1,4);
  \draw[-, line width=0.5mm, darkred] (-2,1) -- (-1,1) -- (1,3);
  \draw[dashed, line width=0.5mm, blue] (1,2) -- (4,2);
  \draw[dashed, line width=0.5mm, red] (1,1) -- (4,1);
  \draw[dashed, line width=0.5mm, darkblue] (1,4) -- (4,4);
  \draw[dashed, line width=0.5mm, darkred] (1,3) -- (4,3);
  \draw[-, line width=0.5mm, UQgold] (4,4) .. controls (4.5,4) and (4.5,3) .. (4,3);
  \draw[-, line width=0.5mm, UQgold] (4,2) .. controls (4.5,2) and (4.5,1) .. (4,1);
  \end{scope}
  \begin{scope}[xshift=8cm]
  \draw node[anchor=east] at (-2,4) {$z_j$};
  \draw node[anchor=east] at (-2,3) {$z_j^{-1}$};
  \draw node[anchor=east] at (-2,2) {$z_i$};
  \draw node[anchor=east] at (-2,1) {$z_i^{-1}$};
  \foreach \x in {-0.5,0,...,1.5}
    \draw[densely dotted] (\x, 4.3) -- (\x,1-0.3);
  \draw[-, line width=0.5mm, blue] (-2,4) -- (-1,4);
  \draw[-, line width=0.5mm, red] (-2,3) -- (-1,2);
  \draw[-, line width=0.5mm, darkblue] (-2,2) -- (-1,3);
  \draw[-, line width=0.5mm, darkred] (-2,1) -- (-1,1);
  \draw[dashed, line width=0.5mm, blue] (-1,4) -- (2,4);
  \draw[dashed, line width=0.5mm, red] (-1,2) -- (2,2);
  \draw[dashed, line width=0.5mm, darkblue] (-1,3) -- (2,3);
  \draw[dashed, line width=0.5mm, darkred] (-1,1) -- (2,1);
  \draw[-, line width=0.5mm, blue] (2,4) -- (4,2);
  \draw[-, line width=0.5mm, red] (2,2) -- (3,1) -- (4,1);
  \draw[-, line width=0.5mm, darkblue] (2,3) -- (3,4) -- (4,4);
  \draw[-, line width=0.5mm, darkred] (2,1) -- (4,3);
  \draw[-, line width=0.5mm, UQgold] (4,4) .. controls (4.5,4) and (4.5,3) .. (4,3);
  \draw[-, line width=0.5mm, UQgold] (4,2) .. controls (4.5,2) and (4.5,1) .. (4,1);
  \end{scope}
  \begin{scope}[xshift=0cm, yshift=-4cm]
  \draw node[anchor=east] at (-2,4) {$z_j$};
  \draw node[anchor=east] at (-2,3) {$z_j^{-1}$};
  \draw node[anchor=east] at (-2,2) {$z_i$};
  \draw node[anchor=east] at (-2,1) {$z_i^{-1}$};
  \foreach \x in {-0.5,0,0.5}
    \draw[densely dotted] (\x, 4.3) -- (\x,1-0.3);
  \draw[-, line width=0.5mm, blue] (-2,4) -- (-1,4);
  \draw[-, line width=0.5mm, red] (-2,3) -- (-1,2);
  \draw[-, line width=0.5mm, darkblue] (-2,2) -- (-1,3);
  \draw[-, line width=0.5mm, darkred] (-2,1) -- (-1,1);
  \draw[dashed, line width=0.5mm, blue] (-1,4) -- (1,4);
  \draw[dashed, line width=0.5mm, red] (-1,2) -- (1,2);
  \draw[dashed, line width=0.5mm, darkblue] (-1,3) -- (1,3);
  \draw[dashed, line width=0.5mm, darkred] (-1,1) -- (1,1);
  \draw[-, line width=0.5mm, blue] (1,4) -- (2,3) -- (3,4) -- (4,4);
  \draw[-, line width=0.5mm, red] (1,2) -- (3,2) -- (4,3);
  \draw[-, line width=0.5mm, darkblue] (1,3) -- (2,4) -- (4,2);
  \draw[-, line width=0.5mm, darkred] (1,1) -- (4,1);
  \draw[-, line width=0.5mm, UQgold] (4,4) .. controls (4.5,4) and (4.5,3) .. (4,3);
  \draw[-, line width=0.5mm, UQgold] (4,2) .. controls (4.5,2) and (4.5,1) .. (4,1);
  \end{scope}
  \begin{scope}[xshift=8cm, yshift=-4cm]
  \draw node[anchor=east] at (-2,4) {$z_j$};
  \draw node[anchor=east] at (-2,3) {$z_j^{-1}$};
  \draw node[anchor=east] at (-2,2) {$z_i$};
  \draw node[anchor=east] at (-2,1) {$z_i^{-1}$};
  \foreach \x in {-0.5,0,...,1.5}
    \draw[densely dotted] (\x, 4.3) -- (\x,1-0.3);
  \draw[-, line width=0.5mm, blue] (-2,4) -- (-1,4);
  \draw[-, line width=0.5mm, red] (-2,3) -- (-1,2);
  \draw[-, line width=0.5mm, darkblue] (-2,2) -- (-1,3);
  \draw[-, line width=0.5mm, darkred] (-2,1) -- (-1,1);
  \draw[dashed, line width=0.5mm, blue] (-1,4) -- (2,4);
  \draw[dashed, line width=0.5mm, red] (-1,2) -- (2,2);
  \draw[dashed, line width=0.5mm, darkblue] (-1,3) -- (2,3);
  \draw[dashed, line width=0.5mm, darkred] (-1,1) -- (2,1);
  \draw[-, line width=0.5mm, blue] (2,4) -- (3,4);
  \draw[-, line width=0.5mm, red] (2,2) -- (3,3);
  \draw[-, line width=0.5mm, darkblue] (2,3) -- (3,2);
  \draw[-, line width=0.5mm, darkred] (2,1) -- (3,1);
  \draw[-, line width=0.5mm, UQgold] (3,4) .. controls (3.5,4) and (3.5,3) .. (3,3);
  \draw[-, line width=0.5mm, UQgold] (3,2) .. controls (3.5,2) and (3.5,1) .. (3,1);
  \end{scope}
  \begin{scope}[xshift=0cm, yshift=-8cm]
  \draw node[anchor=east] at (-2,4) {$z_j$};
  \draw node[anchor=east] at (-2,3) {$z_j^{-1}$};
  \draw node[anchor=east] at (-2,2) {$z_i$};
  \draw node[anchor=east] at (-2,1) {$z_i^{-1}$};
  \foreach \x in {0.5,1,...,2.5}
    \draw[densely dotted] (\x, 4.3) -- (\x,1-0.3);
  \draw[-, line width=0.5mm, blue] (-2,4) -- (0,4);
  \draw[-, line width=0.5mm, red] (-2,3) -- (-1,2) -- (0,3);
  \draw[-, line width=0.5mm, darkblue] (-2,2) -- (-1,3) -- (0,2);
  \draw[-, line width=0.5mm, darkred] (-2,1) -- (0,1);
  \draw[dashed, line width=0.5mm, blue] (0,4) -- (3,4);
  \draw[dashed, line width=0.5mm, darkblue] (0,2) -- (3,2);
  \draw[dashed, line width=0.5mm, red] (0,3) -- (3,3);
  \draw[dashed, line width=0.5mm, darkred] (0,1) -- (3,1);
  \draw[-, line width=0.5mm, UQgold] (3,4) .. controls (3.5,4) and (3.5,3) .. (3,3);
  \draw[-, line width=0.5mm, UQgold] (3,2) .. controls (3.5,2) and (3.5,1) .. (3,1);
  \end{scope}
  \begin{scope}[xshift=8cm, yshift=-8cm]
  \draw node[anchor=east] at (-2,4) {$z_j$};
  \draw node[anchor=east] at (-2,3) {$z_j^{-1}$};
  \draw node[anchor=east] at (-2,2) {$z_i$};
  \draw node[anchor=east] at (-2,1) {$z_i^{-1}$};
  \foreach \x in {-1.5,-1,...,2.5}
    \draw[densely dotted] (\x, 4.3) -- (\x,1-0.3);
  \draw[dashed, line width=0.5mm, blue] (-2,4) -- (3,4);
  \draw[dashed, line width=0.5mm, darkblue] (-2,2) -- (3,2);
  \draw[dashed, line width=0.5mm, red] (-2,3) -- (3,3);
  \draw[dashed, line width=0.5mm, darkred] (-2,1) -- (3,1);
  \draw[-, line width=0.5mm, UQgold] (3,4) .. controls (3.5,4) and (3.5,3) .. (3,3);
  \draw[-, line width=0.5mm, UQgold] (3,2) .. controls (3.5,2) and (3.5,1) .. (3,1);
  \end{scope}
\end{tikzpicture}
\]
\caption{Pictorial description of the sequence of steps to compute the action of the $i$-th atom operator. Note that all the $\Gamma$ rows are colored in blue and all the $\Delta$ rows are colored in red. The R-matrices are determined by the color on the row. For example the left most $R$-matrix in the first step is the $R^\Gamma_\Delta$ $R$-matrix.}
\label{fig:ping_pong}
\end{figure}

\begin{lemma}
\label{lemma:type_A_relation}
Choose $i < n$, $j=i+1$ and $w \in W$ such that $\ell(s_i w) = \ell(w) + 1$.
Then we have
\begin{equation}\label{eq:typeAfe}
(z_i - z_j) Z(\overline{\states}_{\lambda, s_i w}; \zz) =  z_j \bigl( Z(\overline{\states}_{\lambda, w}; \zz) - z_i z_j^{-1} Z(\overline{\states}_{\lambda, w}; s_i \zz) \bigr).
\end{equation}
\end{lemma}

\begin{proof}
We prove the functional equation by following the sequence of steps pictured in Figure~\ref{fig:ping_pong}).
In each step we exhibit a model, and models in two consecutive steps will have the same partition function, possibly up to some factor.
Finally, by comparing the partition functions of the first and the last models, we prove the result. 
\begin{enumerate}
\item In the first step of Figure~\ref{fig:ping_pong} we add the ``double $R$-matrix'' to the left of the model $\overline{\states}_{\lambda, s_iw}$. 
	With the imposed left boundary condition, the double $R$-matrix must be in one of the following two admissible configurations:
	
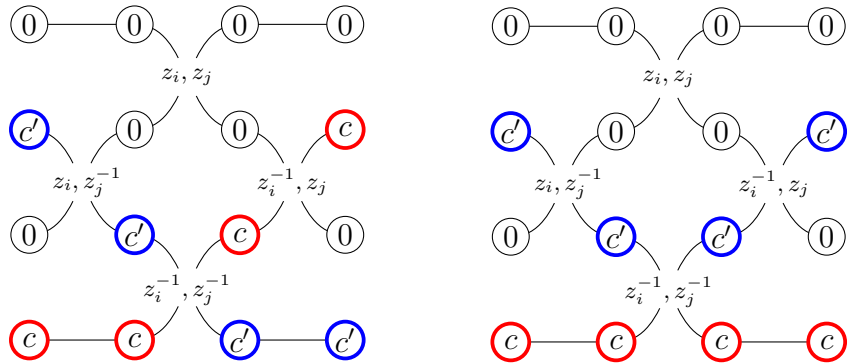
\[
	\begin{tikzpicture}[scale=0.7]
	\draw (0,-2) -- (2, -2);
	\draw (0,4) -- (2, 4);
	\draw (4,-2) -- (6, -2);
	\draw (4,4) -- (6, 4);
	\draw (0,0) to [out = 0, in = 180] (2,2);
	\draw (0,2) to [out = 0, in = 180] (2,0);
	\draw (2,2) to [out = 0, in = 180] (4,4);
	\draw (2,4) to [out = 0, in = 180] (4,2);
	\draw (2,-2) to [out = 0, in = 180] (4,0);
	\draw (2,0) to [out = 0, in = 180] (4,-2);
	\draw (4,0) to [out = 0, in = 180] (6,2);
	\draw (4,2) to [out = 0, in = 180] (6,0);
	\foreach \x in {0,2,4,6} {
	  \foreach \y in {-2,0,2,4} {
	  \draw[fill=white] (\x,\y) circle (.35);
	  }
	}
	\foreach \x/\y in {0/-2, 2/-2, 4/0, 6/2} {
	  \draw[line width=0.5mm, red, fill=white] (\x,\y) circle (.35);
	  \node at (\x,\y) {$c$};
	}
	\foreach \x/\y in {0/2, 2/0, 4/-2, 6/-2} {
	  \draw[line width=0.5mm, blue, fill=white] (\x,\y) circle (.35);
	  \node at (\x,\y) {$c'$};
	}
	\node at (0,0) {$0$};
	\node at (0,4) {$0$};
	\node at (2,2) {$0$};
	\node at (2,4) {$0$};
	\node at (4,2) {$0$};
	\node at (4,4) {$0$};
	\node at (6,0) {$0$};
	\node at (6,4) {$0$};
	\path[fill=white] (1,1) circle (.4);
	\node at (1.1,1) {\scriptsize $z_i,z_j^{-1}$};
	\path[fill=white] (3,3) circle (.4);
	\node at (3,3) {\scriptsize $z_i,z_j$};
	\path[fill=white] (3,-1) circle (.4);
	\node at (3,-1.05) {\scriptsize $z_i^{-1},z_j^{-1}$};
	\path[fill=white] (5,1) circle (.4);
	\node at (5,1) {\scriptsize $z_i^{-1},z_j$};
	\end{tikzpicture}
	\qquad\qquad
	\begin{tikzpicture}[scale=0.7]
	\draw (0,-2) -- (2, -2);
	\draw (0,4) -- (2, 4);
	\draw (4,-2) -- (6, -2);
	\draw (4,4) -- (6, 4);
	\draw (0,0) to [out = 0, in = 180] (2,2);
	\draw (0,2) to [out = 0, in = 180] (2,0);
	\draw (2,2) to [out = 0, in = 180] (4,4);
	\draw (2,4) to [out = 0, in = 180] (4,2);
	\draw (2,-2) to [out = 0, in = 180] (4,0);
	\draw (2,0) to [out = 0, in = 180] (4,-2);
	\draw (4,0) to [out = 0, in = 180] (6,2);
	\draw (4,2) to [out = 0, in = 180] (6,0);
	\foreach \x in {0,2,4,6} {
	  \foreach \y in {-2,0,2,4} {
	  \draw[fill=white] (\x,\y) circle (.35);
	  }
	}
	\foreach \x/\y in {0/-2, 2/-2, 4/-2, 6/-2} {
	  \draw[line width=0.5mm, red, fill=white] (\x,\y) circle (.35);
	  \node at (\x,\y) {$c$};
	}
	\foreach \x/\y in {0/2, 2/0, 4/0, 6/2} {
	  \draw[line width=0.5mm, blue, fill=white] (\x,\y) circle (.35);
	  \node at (\x,\y) {$c'$};
	}
	\node at (0,0) {$0$};
	\node at (0,4) {$0$};
	\node at (2,2) {$0$};
	\node at (2,4) {$0$};
	\node at (4,2) {$0$};
	\node at (4,4) {$0$};
	\node at (6,0) {$0$};
	\node at (6,4) {$0$};
	\path[fill=white] (1,1) circle (.4);
	\node at (1.1,1) {\scriptsize $z_i,z_j^{-1}$};
	\path[fill=white] (3,3) circle (.4);
	\node at (3,3) {\scriptsize $z_i,z_j$};
	\path[fill=white] (3,-1) circle (.4);
	\node at (3,-1.05) {\scriptsize $z_i^{-1},z_j^{-1}$};
	\path[fill=white] (5,1) circle (.4);
	\node at (5,1) {\scriptsize $z_i^{-1},z_j$};
	\end{tikzpicture}
	\]
	Here the colors satisfy ${\color{red} c} > {\color{blue} c'}$ because $\ell(s_i w) = \ell(w) + 1$.
	We conclude that the partition function of the model in the first step is
	\[
	z_i (z_i^{-1} - z_j^{-1}) z_j z_j Z(\overline{\states}_{\lambda, s_i w}; \zz) + z_i z_i^{-1} z_j z_j Z(\overline{\states}_{\lambda, w}; \zz). 
	\]
\item To obtain the second model, we pass the three $R$-matrices to the right by using the Yang--Baxter equation in Proposition~\ref{prop:YBE}, leaving behind the $R^{\Gamma}_{\Delta}$-matrix on the left.
\item We apply the reflection equation (Proposition~\ref{prop:reflection}) to the $R^{\Delta}_{\Delta}$-matrix and $R^{\Delta}_{\Gamma}$-matrix.
	This contributes a factor of $\alpha=z_i^{-2}$.
\item We apply the unitary equation for the square of the $R^{\Gamma}_{\Gamma}$-matrix.
	This contributes a factor of $\beta = z_i z_j$ to the partition function (Proposition~\ref{prop:unitary}).
\item We use the Yang--Baxter equation to pass the remaining $R^{\Delta}_{\Gamma}$-matrix back.
\item We use a weak version of the unitary equation for the $R^{\Delta}_{\Gamma}$-matrix and $R^{\Gamma}_{\Delta}$-matrix with a boundary condition ${\color{blue} c'}$ on top left and top right and $0$ on bottom left and bottom right.
	In terms of Equation~\eqref{eq:unitary_model}, we consider $a = {\color{blue} c'}$ and $b = 0$.
	We only need to consider this boundary equation here because of the boundary conditions for $\states_{\lambda, s_i w}$.
	This contributes a factor of $\gamma = z_i^2$ times the partition function of $Z(\states_{\lambda, s_i w}; s_i \zz)$.
\end{enumerate}
Comparing the initial and final models up to the contributions highlighted above, we obtain the equations
\[
z_i (z_i^{-1} - z_j^{-1}) z_j z_j Z(\overline{\states}_{\lambda, s_i w}; \zz) + z_i z_i^{-1} z_j z_j Z(\overline{\states}_{\lambda, w}; \zz)  =  \alpha \beta \gamma Z(\overline{\states}_{\lambda, s_i w}; s_i \zz).
\]
Using the fact that $\alpha=z_i^{-2}, \beta = z_i z_j$, and $\gamma = z_i^2$ and some basic algebra we can manipulate the equation above to obtain the desired result~\eqref{eq:typeAfe}.
%
\end{proof}

\subsection{Ichthyology}\label{sec:ichthyology}


We now study a version of the fish equation that we use to show a functional equation corresponding to the last simple reflection $s_n$.
We do not actually prove the usual fish equation, but instead dissect it to its component pieces to obtain the desired functional equation.

\begin{lemma}
\label{lemma:type_BC_relation}
Choose $w \in W$ such that $\ell(s_n w) = \ell(w) + 1$.
Then we have
\begin{align*}
(z_n^2 - 1) Z(\overline{\states}_{\lambda, s_n w}; \zz) & = Z(\overline{\states}_{\lambda, w}; \zz) - Z(\overline{\states}_{\lambda, w}; s_n \zz) && (\text{type C}), \\
(z_n - 1) Z(\overline{\states}_{\lambda, s_n w}; \zz) & = Z(\overline{\states}_{\lambda, w}; \zz) - Z(\overline{\states}_{\lambda, w}; s_n \zz) && (\text{type B}).
\end{align*}
\end{lemma}

\begin{proof}
The proof is the same for both cases outside of one computation where the $K$-matrix appears (recall that the L-matrix weights are the same for type $B$ and type $C$). 
We proceed with one proof and be precise where the difference between types occur.
Subsequently, we denote the Boltzmann weight of $\tt{k}_1$ by $K$.
 
In this proof, we work with models consisting of two rows connected by a U-turn on the right. 
These can be thought of as the last two rows in our previous model; therefore we can use the fact that at most one unbarred and one barred colors will appear in this model (and if both do, they will be a pair, \textit{i.e.}, the barred color will be the bar of the unbarred color).
We first follow the idea in~\cite{Ivanov12} to modify the model from a $^{\Gamma}_{\Delta}$ model to a $^{\Gamma}_{\Gamma}$ model.
To do this, for each admissible state in the row, we interchange the color and non-color in the last row $d \leftrightarrow 0$. 
Thus we are now using the $\Gamma$-weights in Figure~\ref{fig:colored_gamma_weights} for the last row.
We will now show that the partition function does not change under this process.

Note that the weights of every entry of the $\Delta$ $L$-matrix with the bottom value being $0$ equals the weights of the corresponding $\Gamma$ $L$-matrix with the interchanging $d \leftrightarrow 0$.
From the boundary conditions, this induces a bijection on the states of the model and preserves the Boltzmann weight contributions from the $L$-matrices.
Therefore the bottom two rows of the model get interchanged as follows:
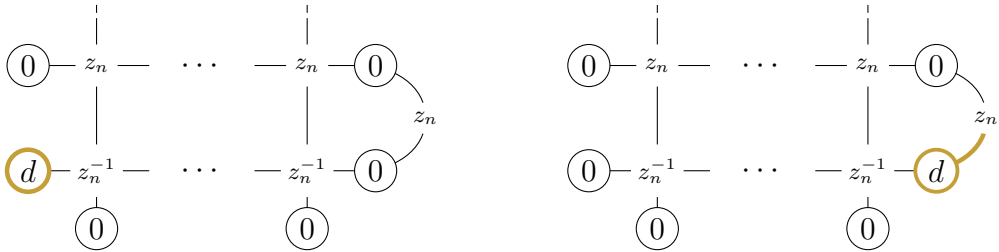
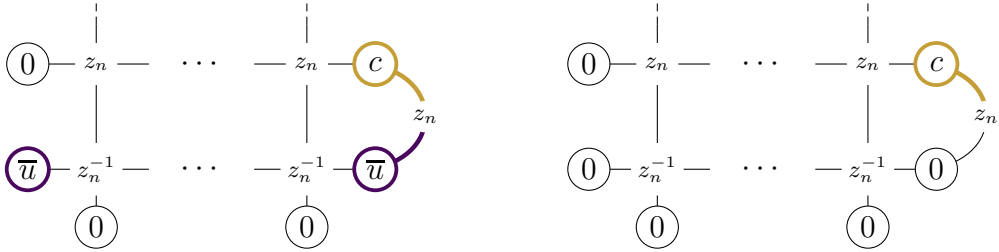
\begin{align*}
&
\begin{tikzpicture}[scale=0.7]
  \draw (2,1) to (4,1);
  \draw (2,3) to (4,3);
  \draw (3,0.25) to (3,3.75);
  \draw (7,0.25) to (7,3.75);
  \draw (6,1) to (8,1);
  \draw (6,3) to (8,3);
  \draw (8,3) to[out=0,in=90] (9.25,2);
  \draw (9.25,2) to[out=-90,in=0] (8,1);
  \fill[white] (9.25,2) circle (.3);
  \node at (9.25,2) {\scriptsize $z_n$};
  \draw[fill=white] (2-0.3,3) circle (.4);
  \draw[fill=white] (2-0.3,1) circle (.4);
  \draw[line width=0.6mm, UQgold] (2-0.3,1) circle (.4);
  \draw[fill=white] (8.3,3) circle (.4);
  \draw[fill=white] (8.3,1) circle (.4);
  \node at (2-0.3,1) {$d$};
  \node at (2-0.3,3) {$0$};
  \node at (5,3) {$\cdots$};
  \node at (5,1) {$\cdots$};
  \draw[densely dashed] (3,3.75) to (3,4.25);
  \draw[fill=white] (3,-0.1) circle (.4) node{$0$};
  \draw[densely dashed] (7,3.75) to (7,4.25);
  \draw[fill=white] (7,-0.1) circle (.4) node{$0$};
  \node at (8.3,3) {$0$};
  \node at (8.3,1) {$0$};
  \path[fill=white] (3,3) circle (.4);
  \node at (3,3) {\scriptsize$z_n$};
  \path[fill=white] (3,1) circle (.5);
  \node at (3,1) {\scriptsize$z_n^{-1}$};
  \path[fill=white] (7,3) circle (.4);
  \node at (7,3) {\scriptsize$z_n$};
  \path[fill=white] (7,1) circle (.5);
  \node at (7,1) {\scriptsize$z_n^{-1}$};
\end{tikzpicture}
&&
\begin{tikzpicture}[scale=0.7]
  \draw (2,1) to (4,1);
  \draw (2,3) to (4,3);
  \draw (3,0.25) to (3,3.75);
  \draw (7,0.25) to (7,3.75);
  \draw (6,1) to (8,1);
  \draw (6,3) to (8,3);
  \draw (8,3) to[out=0,in=90] (9.25,2);
  \draw[line width=0.6mm, UQgold] (9.25,2) to[out=-90,in=0] (8,1);
  \fill[white] (9.25,2) circle (.3);
  \node at (9.25,2) {\scriptsize $z_n$};
  \draw[fill=white] (2-0.3,3) circle (.4);
  \draw[fill=white] (2-0.3,1) circle (.4);
  \draw[fill=white] (8.3,3) circle (.4);
  \draw[line width=0.5mm,UQgold,fill=white] (8.3,1) circle (.4);
  \node at (2-0.3,1) {$0$};
  \node at (2-0.3,3) {$0$};
  \node at (5,3) {$\cdots$};
  \node at (5,1) {$\cdots$};
  \draw[densely dashed] (3,3.75) to (3,4.25);
  \draw[fill=white] (3,-0.1) circle (.4) node{$0$};
  \draw[densely dashed] (7,3.75) to (7,4.25);
  \draw[fill=white] (7,-0.1) circle (.4) node{$0$};
  \node at (8.3,3) {$0$};
  \node at (8.3,1) {$d$};
  \path[fill=white] (3,3) circle (.4);
  \node at (3,3) {\scriptsize$z_n$};
  \path[fill=white] (3,1) circle (.5);
  \node at (3,1) {\scriptsize$z_n^{-1}$};
  \path[fill=white] (7,3) circle (.4);
  \node at (7,3) {\scriptsize$z_n$};
  \path[fill=white] (7,1) circle (.5);
  \node at (7,1) {\scriptsize$z_n^{-1}$};
\end{tikzpicture}
\allowdisplaybreaks
\\ &
\begin{tikzpicture}[scale=0.7]
  \draw (2,1) to (4,1);
  \draw (2,3) to (4,3);
  \draw (3,0.25) to (3,3.75);
  \draw (7,0.25) to (7,3.75);
  \draw (6,1) to (8,1);
  \draw (6,3) to (8,3);
  \draw[line width=0.6mm, UQgold] (8,3) to[out=0,in=90] (9.25,2);
  \draw[line width=0.6mm, UQpurple] (9.25,2) to[out=-90,in=0] (8,1);
  \fill[white] (9.25,2) circle (.3);
  \node at (9.25,2) {\scriptsize $z_n$};
  \draw[fill=white] (2-0.3,3) circle (.4);
  \draw[line width=0.5mm,UQpurple,fill=white] (2-0.3,1) circle (.4);
  \draw[line width=0.5mm,UQgold,fill=white] (8.3,3) circle (.4);
  \draw[line width=0.5mm,UQpurple,fill=white] (8.3,1) circle (.4);
  \node at (2-0.3,1) {$\overline{u}$};
  \node at (2-0.3,3) {$0$};
  \node at (5,3) {$\cdots$};
  \node at (5,1) {$\cdots$};
  \draw[densely dashed] (3,3.75) to (3,4.25);
  \draw[fill=white] (3,-0.1) circle (.4) node{$0$};
  \draw[densely dashed] (7,3.75) to (7,4.25);
  \draw[fill=white] (7,-0.1) circle (.4) node{$0$};
  \node at (8.3,3) {$c$};
  \node at (8.3,1) {$\overline{u}$};
  \path[fill=white] (3,3) circle (.4);
  \node at (3,3) {\scriptsize$z_n$};
  \path[fill=white] (3,1) circle (.5);
  \node at (3,1) {\scriptsize$z_n^{-1}$};
  \path[fill=white] (7,3) circle (.4);
  \node at (7,3) {\scriptsize$z_n$};
  \path[fill=white] (7,1) circle (.5);
  \node at (7,1) {\scriptsize$z_n^{-1}$};
\end{tikzpicture}
&&
\begin{tikzpicture}[scale=0.7]
  \draw (2,1) to (4,1);
  \draw (2,3) to (4,3);
  \draw (3,0.25) to (3,3.75);
  \draw (7,0.25) to (7,3.75);
  \draw (6,1) to (8,1);
  \draw (6,3) to (8,3);
  \draw[line width=0.6mm, UQgold] (8,3) to[out=0,in=90] (9.25,2);
  \draw (9.25,2) to[out=-90,in=0] (8,1);
  \fill[white] (9.25,2) circle (.3);
  \node at (9.25,2) {\scriptsize $z_n$};
  \draw[fill=white] (2-0.3,3) circle (.4);
  \draw[fill=white] (2-0.3,1) circle (.4);
  \draw[line width=0.5mm,UQgold,fill=white] (8.3,3) circle (.4);
  \draw[fill=white] (8.3,1) circle (.4);
  \node at (2-0.3,1) {$0$};
  \node at (2-0.3,3) {$0$};
  \node at (5,3) {$\cdots$};
  \node at (5,1) {$\cdots$};
  \draw[densely dashed] (3,3.75) to (3,4.25);
  \draw[fill=white] (3,-0.1) circle (.4) node{$0$};
  \draw[densely dashed] (7,3.75) to (7,4.25);
  \draw[fill=white] (7,-0.1) circle (.4) node{$0$};
  \node at (8.3,3) {$c$};
  \node at (8.3,1) {$0$};
  \path[fill=white] (3,3) circle (.4);
  \node at (3,3) {\scriptsize$z_n$};
  \path[fill=white] (3,1) circle (.5);
  \node at (3,1) {\scriptsize$z_n^{-1}$};
  \path[fill=white] (7,3) circle (.4);
  \node at (7,3) {\scriptsize$z_n$};
  \path[fill=white] (7,1) circle (.5);
  \node at (7,1) {\scriptsize$z_n^{-1}$};
\end{tikzpicture}
\end{align*}
where we have drawn the $^{\Gamma}_{\Delta}$ model on the left and $^{\Gamma}_{\Gamma}$ model on the right, with $u$ being an unbarred color such that $c \in \{u, \overline{u}\}$.

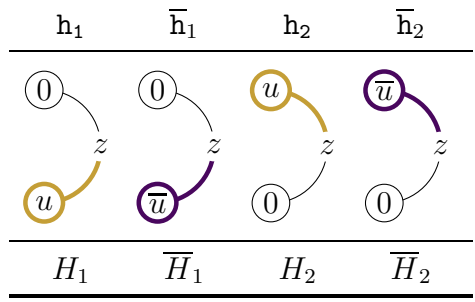
\begin{figure}
\[
\begin{array}{c@{\;\;}c@{\;\;}c@{\;\;}c}
\toprule
  \tt{h}_1 & \overline{\tt{h}}_1 & \tt{h}_2 & \overline{\tt{h}}_2 \\
\midrule
\begin{tikzpicture}
\coordinate (b) at (0, .75);
\coordinate (c) at (.75, 0);
\coordinate (d) at (0, -.75);
\coordinate (aa) at (-.75,.5);
\coordinate (cc) at (.75,.5);
\draw (b) to[out=0,in=90] (c);
\draw[line width=0.6mm, UQgold] (c) to[out=-90,in=0] (d);
\draw[fill=white] (b) circle (.25);
\path[fill=white] (c) circle (.2);
\draw[line width=0.6mm, UQgold, fill=white] (d) circle (.25);
\node at (0,1) { };
\node at (b) {$0$};
\node at (c) {$z$};
\node at (d) {$u$};
\end{tikzpicture}
& \begin{tikzpicture}
\coordinate (b) at (0, .75);
\coordinate (c) at (.75, 0);
\coordinate (d) at (0, -.75);
\coordinate (aa) at (-.75,.5);
\coordinate (cc) at (.75,.5);
\draw (b) to[out=0,in=90] (c);
\draw[line width=0.6mm, UQpurple] (c) to[out=-90,in=0] (d);
\draw[fill=white] (b) circle (.25);
\path[fill=white] (c) circle (.2);
\draw[line width=0.6mm, UQpurple, fill=white] (d) circle (.25);
\node at (0,1) { };
\node at (b) {$0$};
\node at (c) {$z$};
\node at (d) {$\overline{u}$};
\end{tikzpicture}
& \begin{tikzpicture}
\coordinate (b) at (0, .75);
\coordinate (c) at (.75, 0);
\coordinate (d) at (0, -.75);
\coordinate (aa) at (-.75,.5);
\coordinate (cc) at (.75,.5);
\draw[line width=0.6mm, UQgold] (b) to[out=0,in=90] (c);
\draw (c) to[out=-90,in=0] (d);
\draw[line width=0.6mm, UQgold, fill=white] (b) circle (.25);
\path[fill=white] (c) circle (.2);
\draw[fill=white] (d) circle (.25);
\node at (0,1) { };
\node at (b) {$u$};
\node at (c) {$z$};
\node at (d) {$0$};
\end{tikzpicture}
& \begin{tikzpicture}
\coordinate (b) at (0, .75);
\coordinate (c) at (.75, 0);
\coordinate (d) at (0, -.75);
\coordinate (aa) at (-.75,.5);
\coordinate (cc) at (.75,.5);
\draw[line width=0.6mm, UQpurple] (b) to[out=0,in=90] (c);
\draw (c) to[out=-90,in=0] (d);
\draw[line width=0.6mm, UQpurple, fill=white] (b) circle (.25);
\path[fill=white] (c) circle (.2);
\draw[fill=white] (d) circle (.25);
\node at (0,1) { };
\node at (b) {$\overline{u}$};
\node at (c) {$z$};
\node at (d) {$0$};
\end{tikzpicture}
\\ 
\midrule
H_1 & \overline{H}_1 & H_2 & \overline{H}_2 \\
\bottomrule
\end{array}
\]
\caption{The colored $K^{\Gamma}_{\Gamma}$-matrix weights with ${\color{UQgold} u}$ being any \emph{unbarred} color.}
\label{fig:colored_K_weight_GG}
\end{figure}

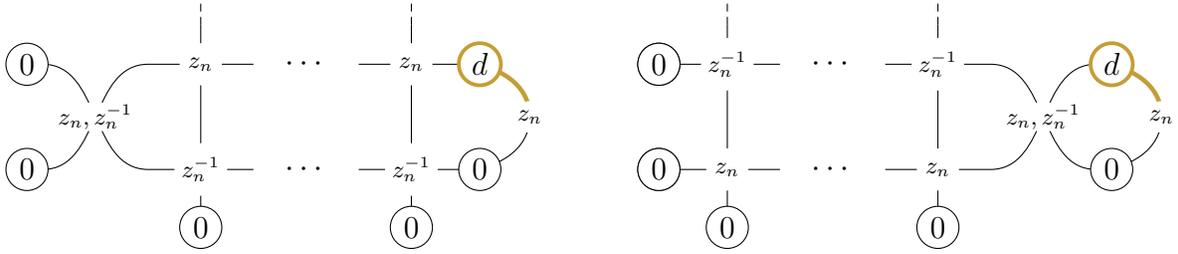
\begin{figure}
\begin{tikzpicture}[scale=0.7]
\begin{scope}[shift={(-6,0)}]
  \draw (0,1) to [out = 0, in = 180] (2,3) to (4,3);
  \draw (0,3) to [out = 0, in = 180] (2,1) to (4,1);
  \draw (3,0.25) to (3,3.75);
  \draw (7,0.25) to (7,3.75);
  \draw (6,1) to (8,1);
  \draw (6,3) to (8,3);
  \draw[line width=0.6mm, UQgold] (8,3) to[out=0,in=90] (9.25,2);
  \draw (9.25,2) to[out=-90,in=0] (8,1);
  \fill[white] (9.25,2) circle (.3);
  \node at (9.25,2) {\scriptsize $z_n$};
  \draw[fill=white] (-0.3,1) circle (.4);
  \draw[fill=white] (-0.3,3) circle (.4);
  \draw[line width=0.5mm,UQgold,fill=white] (8.3,3) circle (.4);
  \draw[fill=white] (8.3,1) circle (.4);
  \node at (-0.3,1) {$0$};
  \node at (-0.3,3) {$0$};
  \node at (5,3) {$\cdots$};
  \node at (5,1) {$\cdots$};
  \draw[densely dashed] (3,3.75) to (3,4.25);
  \draw[fill=white] (3,-0.1) circle (.4) node{$0$};
  \draw[densely dashed] (7,3.75) to (7,4.25);
  \draw[fill=white] (7,-0.1) circle (.4) node{$0$};
  \node at (8.3,3) {$d$};
  \node at (8.3,1) {$0$};
  \path[fill=white] (3,3) circle (.4);
  \node at (3,3) {\scriptsize$z_n$};
  \path[fill=white] (3,1) circle (.5);
  \node at (3,1) {\scriptsize$z_n^{-1}$};
  \path[fill=white] (7,3) circle (.4);
  \node at (7,3) {\scriptsize$z_n$};
  \path[fill=white] (7,1) circle (.5);
  \node at (7,1) {\scriptsize$z_n^{-1}$};
  \path[fill=white] (1,2) circle (.3);
  \node at (1,2) {\scriptsize$z_n,z_n^{-1}$};
\end{scope}
\begin{scope}[shift={(6,0)}]
  \draw (4,1) to (6,1) to [out = 0, in = 180] (8,3);
  \draw (4,3) to (6,3) to [out = 0, in = 180] (8,1);
  \draw[line width=0.6mm, UQgold] (8,3) to[out=0,in=90] (9.25,2);
  \draw (9.25,2) to[out=-90,in=0] (8,1);
  \fill[white] (9.25,2) circle (.3);
  \node at (9.25,2) {\scriptsize $z_n$};
  \draw[fill=white] (-0.3,1) circle (.4);
  \draw[fill=white] (-0.3,3) circle (.4);
  \draw[line width=0.5mm,UQgold,fill=white] (8.3,3) circle (.4);
  \draw[fill=white] (8.3,1) circle (.4);
  \draw (0,1) to (2,1);
  \draw (0,3) to (2,3);
  \draw (5,0.25) to (5,3.75);
  \draw (1,0.25) to (1,3.75);
  \draw[fill=white] (-0.3,1) circle (.4);
  \draw[fill=white] (-0.3,3) circle (.4);
  \node at (3,1) {$\cdots$};
  \node at (3,3) {$\cdots$};
  \draw[densely dashed] (1,3.75) to (1,4.25);
  \draw[fill=white] (1,-0.1) circle (.4) node{$0$};
  \draw[densely dashed] (5,3.75) to (5,4.25);
  \draw[fill=white] (5,-0.1) circle (.4) node{$0$};
  \path[fill=white] (1,3) circle (.5);
  \node at (1,3) {\scriptsize$z_n^{-1}$};
  \path[fill=white] (1,1) circle (.4);
  \node at (1,1) {\scriptsize$z_n$};
  \path[fill=white] (5,3) circle (.5);
  \node (a) at (5,3) {\scriptsize$z_n^{-1}$};
  \path[fill=white] (5,1) circle (.4);
  \node at (5,1) {\scriptsize$z_n$};
  \path[fill=white] (7,2) circle (.3);
  \node at (7,2) {\scriptsize$z_n,z_n^{-1}$};
  \node at (8.3,1) {$0$};
  \node at (8.3,3) {$d$};
  \node at (-0.3,1) {$0$};
  \node at (-0.3,3) {$0$};
\end{scope}
\end{tikzpicture}
\caption{Left: The model $\widetilde{\states}_{\tt{h}_2}$ with an $R$-matrix attached on the left.
    Right: The model after using the Yang--Baxter equation and the standard train argument.}
\label{fig:fishy_argument}
\end{figure}

The $K$-matrix on the right will be called the $K^{\Gamma}_{\Gamma}$-matrix and its Boltzmann weights are presented in Figure~\ref{fig:colored_K_weight_GG}.
For our purposes we only require $H_1$, $\overline{H}_1$, $H_2$, and $\overline{H}_2$ to be any non-zero complex numbers, but otherwise there will be no restrictions.

Let $\widetilde{\states}_{\tt{v}}$ denote the states of the $^{\Gamma}_{\Gamma}$ model for $w$ with a fixed entry $\tt{v}$ from the $K^{\Gamma}_{\Gamma}$-matrix.
Then, we have
\begin{align*}
Z(\overline{\states}_{\lambda, w}; \zz) 
& = K \overline{H}_1^{-1} Z(\widetilde{\states}_{\overline{\tt{h}}_1}; \zz) + H_2^{-1} Z(\widetilde{\states}_{\tt{h}_2}; \zz) + \overline{H}_2^{-1} Z(\widetilde{\states}_{\overline{\tt{h}}_2}; \zz).
\end{align*}
Next attach on the left a $R^{\Gamma}_{\Gamma}$-matrix, and then apply the standard train argument to pass the $R^{\Gamma}_{\Gamma}$-matrix to the right side as in Figure~\ref{fig:fishy_argument}.
Let $\widetilde{\states}^d_{d'}$ denote the $^{\Gamma}_{\Gamma}$ model with the $K$-matrix removed and the bottom two right boundary conditions are $d$ above and $d'$ below (either of which could be $0$).
Therefore, from the $R^{\Gamma}_{\Gamma}$-matrix and the Yang--Baxter equation, we have
\begin{align*}
z_n^{-1} \overline{H}_1^{-1} Z(\widetilde{\states}_{\overline{\tt{h}}_1}; \zz) & = z_n Z(\widetilde{\states}_{\overline{u}}^0; s_n \zz),
\\ z_n^{-1} H_2^{-1} Z(\widetilde{\states}_{\tt{h}_2}; \zz) & = (z_n - z_n^{-1}) Z(\widetilde{\states}_u^0; s_n \zz) + z_n^{-1} Z(\widetilde{\states}_0^u; s_n \zz),
\\ z_n^{-1} \overline{H}_2^{-1} Z(\widetilde{\states}_{\overline{\tt{h}}_2}; \zz) & = (z_n - z_n^{-1}) Z(\widetilde{\states}_{\overline{u}}^0; s_n \zz) + z_n^{-1} Z(\widetilde{\states}_0^{\overline{u}}; s_n \zz),
\end{align*}
where $u$ is an unbarred color.
In particular, we have the following possible \defn{fish}, local configurations of an $R^{\Gamma}_{\Gamma}$-matrix and $K^{\Gamma}_{\Gamma}$-matrix:
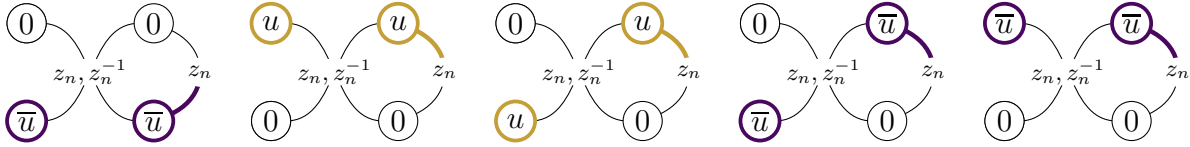
\[
\begin{tikzpicture}[scale=0.65]
\begin{scope}[shift={(0,0)}]
  \draw (0,1) to [out = 0, in = 180] (2,3);
  \draw (0,3) to [out = 0, in = 180] (2,1);
  \draw[-] (2,3) to[out=0,in=90] (3.25,2);
  \draw[line width=0.6mm, UQpurple] (3.25,2) to[out=-90,in=0] (2,1);
  \fill[white] (3.25,2) circle (.3);
  \node at (3.25,2) {\scriptsize $z_n$};
  \draw[fill=white] (-0.3,1) circle (.4);
  \draw[fill=white] (-0.3,3) circle (.4);
  \draw[fill=white] (2.3,3) circle (.4);
  \draw[line width=0.5mm,UQpurple,fill=white] (2.3,1) circle (.4);
  \draw[line width=0.5mm,UQpurple,fill=white] (-0.3,1) circle (.4);
  \draw[fill=white] (-0.3,3) circle (.4);
  \path[fill=white] (1,2) circle (.3);
  \node at (1,2) {\scriptsize$z_n,z_n^{-1}$};
  \node at (2.3,1) {$\overline{u}$};
  \node at (2.3,3) {$0$};
  \node at (-0.3,1) {$\overline{u}$};
  \node at (-0.3,3) {$0$};
\end{scope}
\begin{scope}[shift={(10,0)}]
  \draw (0,1) to [out = 0, in = 180] (2,3);
  \draw (0,3) to [out = 0, in = 180] (2,1);
  \draw[line width=0.6mm, UQgold] (2,3) to[out=0,in=90] (3.25,2);
  \draw[-] (3.25,2) to[out=-90,in=0] (2,1);
  \fill[white] (3.25,2) circle (.3);
  \node at (3.25,2) {\scriptsize $z_n$};
  \draw[fill=white] (-0.3,1) circle (.4);
  \draw[fill=white] (-0.3,3) circle (.4);
  \draw[line width=0.5mm,UQgold,fill=white] (2.3,3) circle (.4);
  \draw[fill=white] (2.3,1) circle (.4);
  \draw[line width=0.5mm,UQgold,fill=white] (-0.3,1) circle (.4);
  \draw[fill=white] (-0.3,3) circle (.4);
  \path[fill=white] (1,2) circle (.3);
  \node at (1,2) {\scriptsize$z_n,z_n^{-1}$};
  \node at (2.3,1) {$0$};
  \node at (2.3,3) {$u$};
  \node at (-0.3,1) {$u$};
  \node at (-0.3,3) {$0$};
\end{scope}
\begin{scope}[shift={(5,0)}]
  \draw (0,1) to [out = 0, in = 180] (2,3);
  \draw (0,3) to [out = 0, in = 180] (2,1);
  \draw[line width=0.6mm, UQgold] (2,3) to[out=0,in=90] (3.25,2);
  \draw[-] (3.25,2) to[out=-90,in=0] (2,1);
  \fill[white] (3.25,2) circle (.3);
  \node at (3.25,2) {\scriptsize $z_n$};
  \draw[fill=white] (-0.3,1) circle (.4);
  \draw[fill=white] (-0.3,3) circle (.4);
  \draw[line width=0.5mm,UQgold,fill=white] (2.3,3) circle (.4);
  \draw[fill=white] (2.3,1) circle (.4);
  \draw[fill=white] (-0.3,1) circle (.4);
  \draw[line width=0.5mm,UQgold,fill=white] (-0.3,3) circle (.4);
  \path[fill=white] (1,2) circle (.3);
  \node at (1,2) {\scriptsize$z_n,z_n^{-1}$};
  \node at (2.3,1) {$0$};
  \node at (2.3,3) {$u$};
  \node at (-0.3,1) {$0$};
  \node at (-0.3,3) {$u$};
\end{scope}
\begin{scope}[shift={(15,0)}]
  \draw (0,1) to [out = 0, in = 180] (2,3);
  \draw (0,3) to [out = 0, in = 180] (2,1);
  \draw[line width=0.6mm, UQpurple] (2,3) to[out=0,in=90] (3.25,2);
  \draw[-] (3.25,2) to[out=-90,in=0] (2,1);
  \fill[white] (3.25,2) circle (.3);
  \node at (3.25,2) {\scriptsize $z_n$};
  \draw[fill=white] (-0.3,1) circle (.4);
  \draw[fill=white] (-0.3,3) circle (.4);
  \draw[line width=0.5mm,UQpurple,fill=white] (2.3,3) circle (.4);
  \draw[fill=white] (2.3,1) circle (.4);
  \draw[line width=0.5mm,UQpurple,fill=white] (-0.3,1) circle (.4);
  \draw[fill=white] (-0.3,3) circle (.4);
  \path[fill=white] (1,2) circle (.3);
  \node at (1,2) {\scriptsize$z_n,z_n^{-1}$};
  \node at (2.3,1) {$0$};
  \node at (2.3,3) {$\overline{u}$};
  \node at (-0.3,1) {$\overline{u}$};
  \node at (-0.3,3) {$0$};
\end{scope}
\begin{scope}[shift={(20,0)}]
  \draw (0,1) to [out = 0, in = 180] (2,3);
  \draw (0,3) to [out = 0, in = 180] (2,1);
  \draw[line width=0.6mm, UQpurple] (2,3) to[out=0,in=90] (3.25,2);
  \draw[-] (3.25,2) to[out=-90,in=0] (2,1);
  \fill[white] (3.25,2) circle (.3);
  \node at (3.25,2) {\scriptsize $z_n$};
  \draw[fill=white] (-0.3,1) circle (.4);
  \draw[fill=white] (-0.3,3) circle (.4);
  \draw[line width=0.5mm,UQpurple,fill=white] (2.3,3) circle (.4);
  \draw[fill=white] (2.3,1) circle (.4);
  \draw[fill=white] (-0.3,1) circle (.4);
  \draw[line width=0.5mm,UQpurple,fill=white] (-0.3,3) circle (.4);
  \path[fill=white] (1,2) circle (.3);
  \node at (1,2) {\scriptsize$z_n,z_n^{-1}$};
  \node at (2.3,1) {$0$};
  \node at (2.3,3) {$\overline{u}$};
  \node at (-0.3,1) {$0$};
  \node at (-0.3,3) {$\overline{u}$};
\end{scope}
\end{tikzpicture}
\]
Therefore, we have
\begin{align*}
Z(\overline{\states}_{\lambda, w}; \zz) & =  K z_n^2 Z(\widetilde{\states}_{\overline{u}}^0; s_n \zz)
+ z_n\bigl( (z_n - z_n^{-1}) Z(\widetilde{\states}_u^0; s_n \zz) + z_n^{-1} Z(\widetilde{\states}_0^u; s_n \zz) \bigr)
  \\ & \hspace{20pt}  + z_n \bigl( (z_n - z_n^{-1}) Z(\widetilde{\states}_{\overline{u}}^0; s_n \zz) + z_n^{-1} Z(\widetilde{\states}_0^{\overline{u}}; s_n \zz) \bigr)
\allowdisplaybreaks
\\  & = (K z_n^2 + z_n^2 - 1) \overline{K}^{-1} Z(\overline{\states}_0^0(w); s_n \zz) + Z(\overline{\states}_{\overline{u}}^u(w); s_n \zz) + Z(\overline{\states}_{\overline{u}}^{\overline{u}}(w); s_n \zz)
  \\ & \hspace{20pt} + (z_n^2 - 1) \overline{K}^{-1} Z(\overline{\states}_0^0(s_n w); s_n \zz),
\end{align*}
where $\overline{\states}_{d'}^d(w)$ denotes the model $\overline{\states}_{\lambda, w}$ with a fixed the $K$-matrix $^{d}_{d'}$ on the bottom two rows and $\overline{K}$ denotes the parameter $K$ but with $z_n \mapsto z_n^{-1}$.

Note that $Z(\overline{\states}_0^0(s_n w); s_n \zz) = Z(\overline{\states}_{\lambda, s_n w}; s_n \zz)$.
To obtain our desired functional equations, we need to group the first three terms together into $Z(\overline{\states}_{\lambda, w}; s_n \zz)$.
Hence, we require
\begin{equation}
\label{eq:K_relation}
(K z_n^2 + z_n^2 - 1) \overline{K}^{-1} = 1.
\end{equation}
The $K$-matrix for type $C$, where $K = z_n^{-2}$, satisfies Equation~\eqref{eq:K_relation}.
Therefore, we have
\begin{align*}
Z(\overline{\states}^C_{\lambda, w}; \zz) & = Z(\overline{\states}^C_{\lambda, w}; s_n \zz) + (1- z_n^{-2}) Z(\overline{\states}^C_{\lambda, s_n w}; s_n \zz),
\\ (z_n^2 - 1) Z(\overline{\states}^C_{\lambda, s_n w}; \zz) & = Z(\overline{\states}^C_{\lambda, w}; \zz) - Z(\overline{\states}^C_{\lambda, w}; s_n \zz).
\end{align*}
We can also take $K = z_n^{-2} \cdot (z_n + 1) = z_n^{-1} + z_n^{-2}$ using the type $B$ $K$-matrix, which also is a solution to Equation~\eqref{eq:K_relation}.
In this case, we obtain
\begin{align*}
(z_n - 1) Z(\overline{\states}^B_{\lambda, s_n w}; \zz) & = Z(\overline{\states}^B_{\lambda, w}; \zz) - Z(\overline{\states}^B_{\lambda, w}; s_n \zz).
\end{align*}
This is the functional equation for type $B$ as desired.
\end{proof}

\begin{remark}
We can construct another quasi-solvable lattice model $\overline{\states}^R_{\lambda, w}$ by using the solution $K = -1$ to Equation~\eqref{eq:K_relation}.
In this case, our functional equation becomes
\begin{align*}
(z_n^2 - 1) Z(\overline{\states}^R_{\lambda, s_n w}; s_n \zz) & = Z(\overline{\states}^R_{\lambda, w}; s_n \zz) - Z(\overline{\states}^R_{\lambda, w}; \zz).
\end{align*}
which is a flipped version of the type $C$ functional equation.
\end{remark}


\subsection{The first main theorem}

Using the functional equations we have shown, we can now prove our first main result.

\begin{theorem}
\label{thm:partition_atom_BC}
For Cartan type $X \in \{B, C\}$, we have
\[
Z(\overline{\states}^X_{\lambda, w}; \zz) = \zz^{\rho} A_w(\zz, \lambda).
\]
\end{theorem}

\begin{proof}
The case $w = 1$ is clear as $\zz^{\rho} A_w(\zz, \lambda) = \z^{\lambda+\rho}$ by definition and there is a unique admissible state in $\overline{\states}^X_{\lambda,1}$ whose weight can be easily seen to be $\z^{\lambda+\rho}$ both in types $B$ and $C$.
The remainder of the proof proceeds by induction on the length of $w$ and uses the functional equations proved in the previous sections.

Let $i < n$ with $j = i + 1$ and let $w$ be such that $\ell(s_iw) = \ell(w)+1$. From our induction assumption, we have that $Z(\overline{\states}_{\lambda, w}; \zz) = \zz^{\rho} A_w(\zz, \lambda)$ and that
\[
Z(\overline{\states}_{\lambda, w}; s_i \zz) = z_i^{-1} z_j \cdot \zz^{\rho} s_i A_w(\zz, \lambda) = z_i^{-1} z_j \cdot \zz^{\rho} A_w(s_i \zz, \lambda).
\]
Using the functional equation in Lemma~\ref{lemma:type_A_relation}, we obtain 
\[
(z_i - z_j) Z(\overline{\states}_{\lambda, s_i w}; \zz) =  z_j \cdot \zz^\rho \bigl( A_w(\zz, \lambda) -  A_w(s_i \zz, \lambda) \bigr),
\]
which combined with Equation~\eqref{eq:atom_functional_A} produces $Z(\overline{\states}_{\lambda, s_iw}; \zz) = \zz^{\rho} A_{s_iw}(\zz, \lambda)$ as desired.

For $i = n$, we note that $s_n \zz^{\rho} = \zz^{\rho}$.
It is then easy to see that by using the induction assumption together with Lemma~\ref{lemma:type_BC_relation} and Equation~\eqref{eq:atom_functional_B} in type $B$ or Equation~\eqref{eq:atom_functional_C} in type $C$, we obtain that $Z(\overline{\states}_{\lambda, s_n w}; \zz) = \zz^{\rho} A_{s_n w}(\zz, \lambda)$.
\end{proof}

\begin{example}
\label{ex:atom_21}
Let $n = 2$ and $\lambda = (2,1)$.
Then we have the following states in $\overline{\states}_w$:
\ifexamples
\newcommand{\cs}{.35} 
\newcommand{\lw}{0.6mm} 
\newcommand{\thescale}{0.35} 
\newcommand{\base}{
  \foreach \y in {1,5} {
    \draw[-] (8,\y) to [out=0, in=0] (8,\y+2);
  }
  \foreach \x in {1,3,5,7}{
    \draw[-] (\x,0) -- (\x,8);
    \foreach \y in {0,2,4,6,8}
      \draw[fill=white] (\x,\y) circle (\cs);
  }
  \foreach \y in {1,3,5,7} {
    \draw[-] (0,\y) -- (8,\y);
    \foreach \x in {0,2,4,6,8}
      \draw[fill=white] (\x,\y) circle (\cs);
  }
}
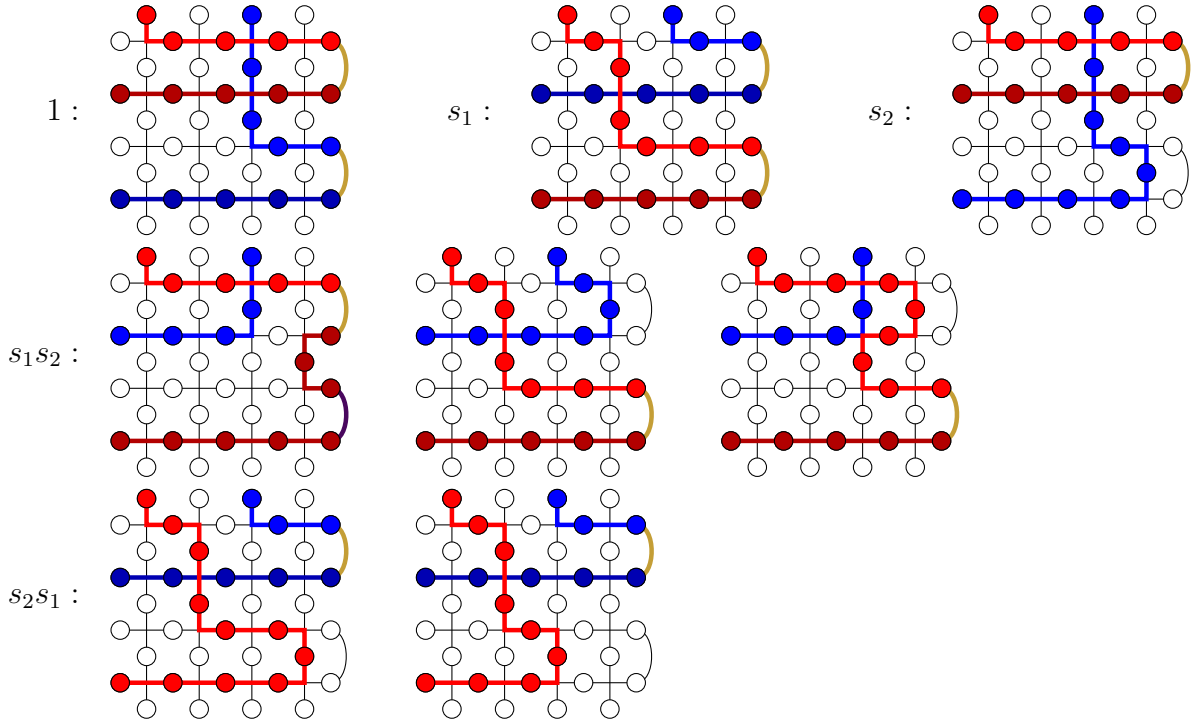
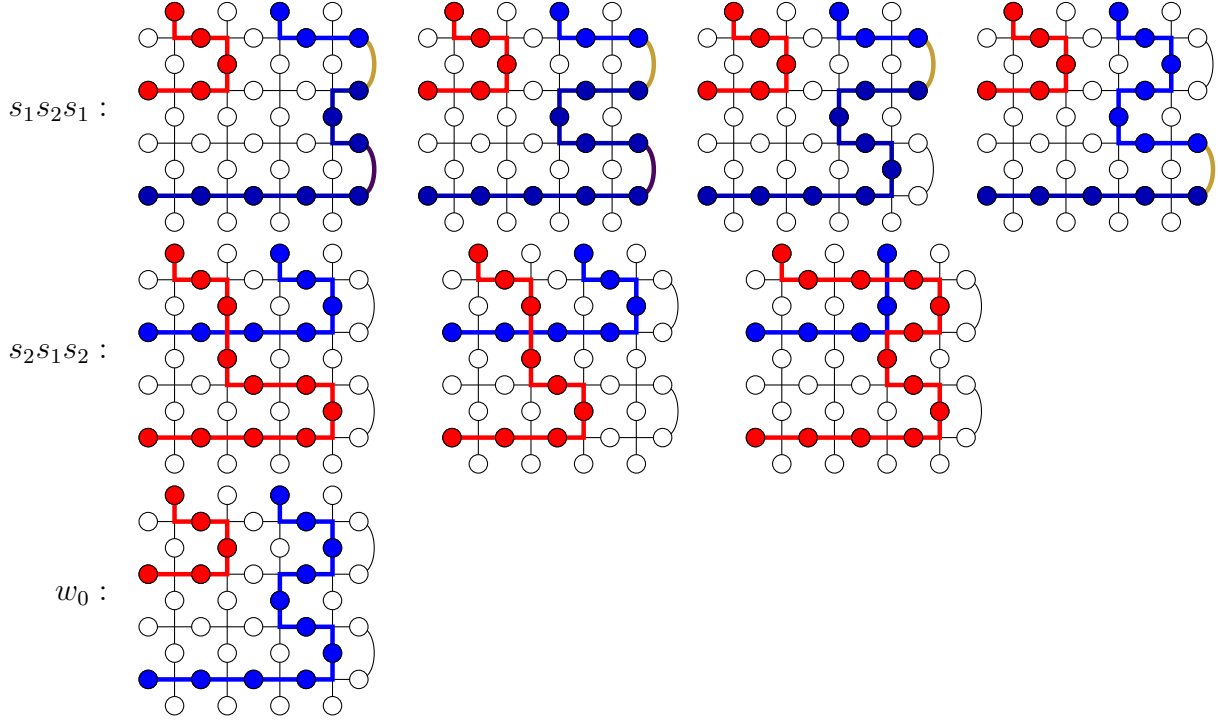
\begin{align*}
1: & \quad
\begin{tikzpicture}[scale=\thescale,baseline=40]
  \base
  \draw[-, UQgold, line width=\lw] (8,7) to[out=0, in=0] (8,5);
  \draw[-, UQgold, line width=\lw] (8,3) to[out=0, in=0] (8,1);
  \draw[-,blue,line width=\lw] (5,8) -- (5,3) -- (8,3);
  \draw[-,darkblue, line width=\lw] (8,1) -- (0,1);
  \draw[-,red,line width=\lw] (1,8) -- (1,7) -- (8,7);
  \draw[-,darkred, line width=\lw] (8,5) -- (0,5);
  \draw[fill=red] (1,8) circle (\cs);
  \foreach \x in {2,4,6,8}
    \draw[fill=red] (\x,7) circle (\cs);
  \foreach \x in {0,2,4,6,8}
  \draw[fill=darkred] (\x,5) circle (\cs);
  \foreach \y in {8,6,4}
    \draw[fill=blue] (5,\y) circle (\cs);
  \foreach \x in {6,8}
    \draw[fill=blue] (\x,3) circle (\cs);
  \foreach \x in {0,2,4,6,8}
    \draw[fill=darkblue] (\x,1) circle (\cs);
\end{tikzpicture}
\qquad\quad
s_1: \quad
\begin{tikzpicture}[scale=\thescale,baseline=40]
  \base
  \draw[-, UQgold, line width=\lw] (8,7) to[out=0, in=0] (8,5);
  \draw[-, UQgold, line width=\lw] (8,3) to[out=0, in=0] (8,1);
  \draw[-,blue,line width=\lw] (5,8) -- (5,7) -- (8,7);
  \draw[-,darkblue, line width=\lw] (8,5) -- (0,5);
  \draw[-,red,line width=\lw] (1,8) -- (1,7) -- (3,7) -- (3,3) -- (8,3);
  \draw[-,darkred, line width=\lw] (8,1) -- (0,1);
  \draw[fill=red] (1,8) circle (\cs);
  \draw[fill=red] (2,7) circle (\cs);
  \foreach \y in {4,6}
    \draw[fill=red] (3,\y) circle (\cs);
  \foreach \x in {4,6,8}
    \draw[fill=red] (\x,3) circle (\cs);
  \foreach \x in {0,2,4,6,8}
    \draw[fill=darkred] (\x,1) circle (\cs);
  \draw[fill=blue] (5,8) circle (\cs);
  \foreach \x in {6,8}
    \draw[fill=blue] (\x,7) circle (\cs);
  \foreach \x in {0,2,4,6,8}
    \draw[fill=darkblue] (\x,5) circle (\cs);
\end{tikzpicture}
\qquad\quad
s_2: \quad
\begin{tikzpicture}[scale=\thescale,baseline=40]
  \base
  \draw[-, UQgold, line width=\lw] (8,7) to[out=0, in=0] (8,5);
  \draw[-,blue,line width=\lw] (5,8) -- (5,3) -- (7,3) -- (7,1) -- (0,1);
  \draw[-,red,line width=\lw] (1,8) -- (1,7) -- (8,7);
  \draw[-,darkred, line width=\lw] (8,5) -- (0,5);
  \draw[fill=red] (1,8) circle (\cs);
  \foreach \x in {2,4,6,8}
    \draw[fill=red] (\x,7) circle (\cs);
  \foreach \x in {0,2,4,6,8}
  \draw[fill=darkred] (\x,5) circle (\cs);
  \foreach \y in {8,6,4}
    \draw[fill=blue] (5,\y) circle (\cs);
  \draw[fill=blue] (6,3) circle (\cs);
  \draw[fill=blue] (7,2) circle (\cs);
  \foreach \x in {0,2,4,6}
    \draw[fill=blue] (\x,1) circle (\cs);
\end{tikzpicture}
\allowdisplaybreaks
\\
s_1 s_2: & \quad
\begin{tikzpicture}[scale=\thescale,baseline=40]
  \base
  \draw[-, UQgold, line width=\lw] (8,7) to[out=0, in=0] (8,5);
  \draw[-, UQpurple, line width=\lw] (8,1) to[out=0, in=0] (8,3);
  \draw[-,blue,line width=\lw] (5,8) -- (5,5) -- (0,5);
  \draw[-,red,line width=\lw] (1,8) -- (1,7) -- (8,7);
  \draw[-,darkred, line width=\lw] (8,5) -- (7,5) -- (7,3) -- (8,3);
  \draw[-,darkred, line width=\lw] (8,1) -- (0,1);
  \draw[fill=red] (1,8) circle (\cs);
  \foreach \x in {2,4,6,8}
    \draw[fill=red] (\x,7) circle (\cs);
  \draw[fill=darkred] (8,5) circle (\cs);
  \draw[fill=darkred] (7,4) circle (\cs);
  \draw[fill=darkred] (8,3) circle (\cs);
  \foreach \x in {0,2,4,6,8}
    \draw[fill=darkred] (\x,1) circle (\cs);
  \draw[fill=blue] (5,8) circle (\cs);
  \draw[fill=blue] (5,6) circle (\cs);
  \foreach \x in {0,2,4}
    \draw[fill=blue] (\x,5) circle (\cs);
\end{tikzpicture}
\qquad
\begin{tikzpicture}[scale=\thescale,baseline=40]
  \base
  \draw[-, UQgold, line width=\lw] (8,3) to[out=0, in=0] (8,1);
  \draw[-,blue,line width=\lw] (5,8) -- (5,7) -- (7,7) -- (7,5) -- (0,5);
  \draw[-,red,line width=\lw] (1,8) -- (1,7) -- (3,7) -- (3,3) -- (8,3);
  \draw[-,darkred, line width=\lw] (8,1) -- (0,1);
  \draw[fill=red] (1,8) circle (\cs);
  \draw[fill=red] (2,7) circle (\cs);
  \foreach \y in {4,6}
    \draw[fill=red] (3,\y) circle (\cs);
  \foreach \x in {4,6,8}
    \draw[fill=red] (\x,3) circle (\cs);
  \foreach \x in {0,2,4,6,8}
    \draw[fill=darkred] (\x,1) circle (\cs);
  \draw[fill=blue] (5,8) circle (\cs);
  \draw[fill=blue] (6,7) circle (\cs);
  \draw[fill=blue] (7,6) circle (\cs);
  \foreach \x in {0,2,4,6}
  \draw[fill=blue] (\x,5) circle (\cs);
\end{tikzpicture}
\qquad
\begin{tikzpicture}[scale=\thescale,baseline=40]
  \base
  \draw[-, UQgold, line width=\lw] (8,3) to[out=0, in=0] (8,1);
  \draw[-,blue,line width=\lw] (5,8) -- (5,5) -- (0,5);
  \draw[-,red,line width=\lw] (1,8) -- (1,7) -- (7,7) -- (7,5) -- (5,5) -- (5,3) -- (8,3);
  \draw[-,darkred, line width=\lw] (8,1) -- (0,1);
  \draw[fill=red] (1,8) circle (\cs);
  \foreach \x in {2,4,6}
    \draw[fill=red] (\x,7) circle (\cs);
  \draw[fill=red] (7,6) circle (\cs);
  \draw[fill=red] (6,5) circle (\cs);
  \draw[fill=red] (5,4) circle (\cs);
  \foreach \x in {6,8}
    \draw[fill=red] (\x,3) circle (\cs);
  \foreach \x in {0,2,4,6,8}
    \draw[fill=darkred] (\x,1) circle (\cs);
  \draw[fill=blue] (5,8) circle (\cs);
  \draw[fill=blue] (5,6) circle (\cs);
  \foreach \x in {0,2,4}
    \draw[fill=blue] (\x,5) circle (\cs);
\end{tikzpicture}
\allowdisplaybreaks
\\
s_2 s_1: & \quad
\begin{tikzpicture}[scale=\thescale,baseline=40]
  \base
  \draw[-, UQgold, line width=\lw] (8,7) to[out=0, in=0] (8,5);
  \draw[-,blue,line width=\lw] (5,8) -- (5,7) -- (8,7);
  \draw[-,darkblue, line width=\lw] (8,5) -- (0,5);
  \draw[-,red,line width=\lw] (1,8) -- (1,7) -- (3,7) -- (3,3) -- (7,3) -- (7,1) -- (0,1);
  \draw[fill=red] (1,8) circle (\cs);
  \draw[fill=red] (2,7) circle (\cs);
  \draw[fill=red] (7,2) circle (\cs);
  \foreach \y in {4,6}
    \draw[fill=red] (3,\y) circle (\cs);
  \foreach \x in {4,6}
    \draw[fill=red] (\x,3) circle (\cs);
  \foreach \x in {0,2,4,6}
    \draw[fill=red] (\x,1) circle (\cs);
  \draw[fill=blue] (5,8) circle (\cs);
  \foreach \x in {6,8}
    \draw[fill=blue] (\x,7) circle (\cs);
  \foreach \x in {0,2,4,6,8}
    \draw[fill=darkblue] (\x,5) circle (\cs);
\end{tikzpicture}
\qquad
\begin{tikzpicture}[scale=\thescale,baseline=40]
  \base
  \draw[-, UQgold, line width=\lw] (8,7) to[out=0, in=0] (8,5);
  \draw[-,blue,line width=\lw] (5,8) -- (5,7) -- (8,7);
  \draw[-,darkblue, line width=\lw] (8,5) -- (0,5);
  \draw[-,red,line width=\lw] (1,8) -- (1,7) -- (3,7) -- (3,3) -- (5,3) -- (5,1) -- (0,1);
  \draw[fill=red] (1,8) circle (\cs);
  \draw[fill=red] (2,7) circle (\cs);
  \draw[fill=red] (5,2) circle (\cs);
  \foreach \y in {4,6}
    \draw[fill=red] (3,\y) circle (\cs);
  \draw[fill=red] (4,3) circle (\cs);
  \foreach \x in {0,2,4}
    \draw[fill=red] (\x,1) circle (\cs);
  \draw[fill=blue] (5,8) circle (\cs);
  \foreach \x in {6,8}
    \draw[fill=blue] (\x,7) circle (\cs);
  \foreach \x in {0,2,4,6,8}
    \draw[fill=darkblue] (\x,5) circle (\cs);
\end{tikzpicture}
\allowdisplaybreaks
\\
s_1 s_2 s_1: & \quad
\begin{tikzpicture}[scale=\thescale,baseline=40]
  \base
  \draw[-, UQgold, line width=\lw] (8,7) to[out=0, in=0] (8,5);
  \draw[-, UQpurple, line width=\lw] (8,1) to[out=0, in=0] (8,3);
  \draw[-,blue,line width=\lw] (5,8) -- (5,7) -- (8,7);
  \draw[-,darkblue, line width=\lw] (8,5) -- (7,5) -- (7,3) -- (8,3);
  \draw[-,darkblue, line width=\lw] (8,1) -- (0,1);
  \draw[-,red,line width=\lw] (1,8) -- (1,7) -- (3,7) -- (3,5) -- (0,5);
  \draw[fill=red] (1,8) circle (\cs);
  \draw[fill=red] (2,7) circle (\cs);
  \draw[fill=red] (3,6) circle (\cs);
  \foreach \x in {0,2}
    \draw[fill=red] (\x,5) circle (\cs);
  \draw[fill=blue] (5,8) circle (\cs);
  \foreach \x in {6,8}
    \draw[fill=blue] (\x,7) circle (\cs);
  \draw[fill=darkblue] (8,5) circle (\cs);
  \draw[fill=darkblue] (7,4) circle (\cs);
  \draw[fill=darkblue] (8,3) circle (\cs);
  \foreach \x in {0,2,4,6,8}
    \draw[fill=darkblue] (\x,1) circle (\cs);
\end{tikzpicture}
\quad\,
\begin{tikzpicture}[scale=\thescale,baseline=40]
  \base
  \draw[-, UQgold, line width=\lw] (8,7) to[out=0, in=0] (8,5);
  \draw[-, UQpurple, line width=\lw] (8,1) to[out=0, in=0] (8,3);
  \draw[-,blue,line width=\lw] (5,8) -- (5,7) -- (8,7);
  \draw[-,darkblue, line width=\lw] (8,5) -- (5,5) -- (5,3) -- (8,3);
  \draw[-,darkblue, line width=\lw] (8,1) -- (0,1);
  \draw[-,red,line width=\lw] (1,8) -- (1,7) -- (3,7) -- (3,5) -- (0,5);
  \draw[fill=red] (1,8) circle (\cs);
  \draw[fill=red] (2,7) circle (\cs);
  \draw[fill=red] (3,6) circle (\cs);
  \foreach \x in {0,2}
    \draw[fill=red] (\x,5) circle (\cs);
  \draw[fill=blue] (5,8) circle (\cs);
  \foreach \x in {6,8} {
    \draw[fill=blue] (\x,7) circle (\cs);
    \draw[fill=darkblue] (\x,5) circle (\cs);
    \draw[fill=darkblue] (\x,3) circle (\cs);
  }
  \draw[fill=darkblue] (5,4) circle (\cs);
  \foreach \x in {0,2,4,6,8}
    \draw[fill=darkblue] (\x,1) circle (\cs);
\end{tikzpicture}
\quad\,
\begin{tikzpicture}[scale=\thescale,baseline=40]
  \base
  \draw[-, UQgold, line width=\lw] (8,7) to[out=0, in=0] (8,5);
  \draw[-,blue,line width=\lw] (5,8) -- (5,7) -- (8,7);
  \draw[-,darkblue, line width=\lw] (8,5) -- (5,5) -- (5,3) -- (7,3) -- (7,1) -- (0,1);
  \draw[-,red,line width=\lw] (1,8) -- (1,7) -- (3,7) -- (3,5) -- (0,5);
  \draw[fill=red] (1,8) circle (\cs);
  \draw[fill=red] (2,7) circle (\cs);
  \draw[fill=red] (3,6) circle (\cs);
  \foreach \x in {0,2}
    \draw[fill=red] (\x,5) circle (\cs);
  \draw[fill=blue] (5,8) circle (\cs);
  \foreach \x in {6,8} {
    \draw[fill=blue] (\x,7) circle (\cs);
    \draw[fill=darkblue] (\x,5) circle (\cs);
  }
  \draw[fill=darkblue] (5,4) circle (\cs);
  \draw[fill=darkblue] (6,3) circle (\cs);
  \draw[fill=darkblue] (7,2) circle (\cs);
  \foreach \x in {0,2,4,6}
    \draw[fill=darkblue] (\x,1) circle (\cs);
\end{tikzpicture}
\quad\,
\begin{tikzpicture}[scale=\thescale,baseline=40]
  \base
  \draw[-, UQgold, line width=\lw] (8,3) to[out=0, in=0] (8,1);
  \draw[-,blue,line width=\lw] (5,8) -- (5,7) -- (7,7) -- (7,5) -- (5,5) -- (5,3) -- (8,3);
  \draw[-,darkblue, line width=\lw] (8,1) -- (0,1);
  \draw[-,red,line width=\lw] (1,8) -- (1,7) -- (3,7) -- (3,5) -- (0,5);
  \draw[fill=red] (1,8) circle (\cs);
  \draw[fill=red] (2,7) circle (\cs);
  \draw[fill=red] (3,6) circle (\cs);
  \foreach \x in {0,2}
    \draw[fill=red] (\x,5) circle (\cs);
  \draw[fill=blue] (5,8) circle (\cs);
  \draw[fill=blue] (6,7) circle (\cs);
  \draw[fill=blue] (7,6) circle (\cs);
  \draw[fill=blue] (6,5) circle (\cs);
  \draw[fill=blue] (5,4) circle (\cs);
  \foreach \x in {6,8}
  \draw[fill=blue] (\x,3) circle (\cs);
  \foreach \x in {0,2,4,6,8}
    \draw[fill=darkblue] (\x,1) circle (\cs);
\end{tikzpicture}
\allowdisplaybreaks
\\
s_2 s_1 s_2: & \quad
\begin{tikzpicture}[scale=\thescale,baseline=40]
  \base
  \draw[-,blue,line width=\lw] (5,8) -- (5,7) -- (7,7) -- (7,5) -- (0,5);
  \draw[-,red,line width=\lw] (1,8) -- (1,7) -- (3,7) -- (3,3) -- (7,3) -- (7,1) -- (0,1);
  \draw[fill=red] (1,8) circle (\cs);
  \draw[fill=red] (2,7) circle (\cs);
  \draw[fill=red] (7,2) circle (\cs);
  \foreach \y in {4,6}
    \draw[fill=red] (3,\y) circle (\cs);
  \foreach \x in {4,6}
    \draw[fill=red] (\x,3) circle (\cs);
  \foreach \x in {0,2,4,6}
    \draw[fill=red] (\x,1) circle (\cs);
  \draw[fill=blue] (5,8) circle (\cs);
  \draw[fill=blue] (6,7) circle (\cs);
  \draw[fill=blue] (7,6) circle (\cs);
  \foreach \x in {0,2,4,6}
    \draw[fill=blue] (\x,5) circle (\cs);
\end{tikzpicture}
\qquad
\begin{tikzpicture}[scale=\thescale,baseline=40]
  \base
  \draw[-,blue,line width=\lw] (5,8) -- (5,7) -- (7,7) -- (7,5) -- (0,5);
  \draw[-,red,line width=\lw] (1,8) -- (1,7) -- (3,7) -- (3,3) -- (5,3) -- (5,1) -- (0,1);
  \draw[fill=red] (1,8) circle (\cs);
  \draw[fill=red] (2,7) circle (\cs);
  \draw[fill=red] (5,2) circle (\cs);
  \foreach \y in {4,6}
    \draw[fill=red] (3,\y) circle (\cs);
  \draw[fill=red] (4,3) circle (\cs);
  \foreach \x in {0,2,4}
    \draw[fill=red] (\x,1) circle (\cs);
  \draw[fill=blue] (5,8) circle (\cs);
  \draw[fill=blue] (6,7) circle (\cs);
  \draw[fill=blue] (7,6) circle (\cs);
  \foreach \x in {0,2,4,6}
    \draw[fill=blue] (\x,5) circle (\cs);
\end{tikzpicture}
\qquad
\begin{tikzpicture}[scale=\thescale,baseline=40]
  \base
  \draw[-,blue,line width=\lw] (5,8) -- (5,5) -- (0,5);
  \draw[-,red,line width=\lw] (1,8) -- (1,7) -- (7,7) -- (7,5) -- (5,5) -- (5,3) -- (7,3) -- (7,1) -- (0,1);
  \draw[fill=red] (1,8) circle (\cs);
  \foreach \x in {2,4,6}
    \draw[fill=red] (\x,7) circle (\cs);
  \draw[fill=red] (7,6) circle (\cs);
  \draw[fill=red] (6,5) circle (\cs);
  \draw[fill=red] (5,4) circle (\cs);
  \draw[fill=red] (6,3) circle (\cs);
  \draw[fill=red] (7,2) circle (\cs);
  \foreach \x in {0,2,4,6}
    \draw[fill=red] (\x,1) circle (\cs);
  \foreach \y in {8,6}
    \draw[fill=blue] (5,\y) circle (\cs);
  \foreach \x in {0,2,4}
    \draw[fill=blue] (\x,5) circle (\cs);
\end{tikzpicture}
\allowdisplaybreaks
\\
w_0: &  \quad
\begin{tikzpicture}[scale=\thescale,baseline=40]
  \base
  \draw[-,blue,line width=\lw] (5,8) -- (5,7) -- (7,7) -- (7,5) -- (5,5) -- (5,3) -- (7,3) -- (7,1) -- (0,1);
  \draw[-,red,line width=\lw] (1,8) -- (1,7) -- (3,7) -- (3,5) -- (0,5);
  \draw[fill=red] (1,8) circle (\cs);
  \draw[fill=red] (2,7) circle (\cs);
  \draw[fill=red] (3,6) circle (\cs);
  \foreach \x in {0,2}
    \draw[fill=red] (\x,5) circle (\cs);
  \draw[fill=blue] (5,8) circle (\cs);
  \draw[fill=blue] (6,7) circle (\cs);
  \draw[fill=blue] (7,6) circle (\cs);
  \draw[fill=blue] (6,5) circle (\cs);
  \draw[fill=blue] (5,4) circle (\cs);
  \draw[fill=blue] (6,3) circle (\cs);
  \draw[fill=blue] (7,2) circle (\cs);
  \foreach \x in {0,2,4,6}
    \draw[fill=blue] (\x,1) circle (\cs);
\end{tikzpicture}
\end{align*}
\fi
Next, we compute the partition functions for each model in type $C$ for $\overline{Z}_w := Z(\overline{\states}^C_{\lambda, w}; \zz)$:
\begin{gather*}
\overline{Z}_1 = z_1^3 z_2,
\qquad\qquad
\overline{Z}_{s_1} = z_1^2 z_2^2,
\qquad\qquad
\overline{Z}_{s_2} = z_1^3 z_2^{-1},
\\
\overline{Z}_{s_1 s_2} = z_1^2 + z_2^2 + z_1 z_2,
\qquad\qquad
\overline{Z}_{s_2 s_1} = z_1^2 + z_1^2 z_2^{-2},
\\
\overline{Z}_{s_1 s_2 s_1} = 1 + z_1 z_2 + z_1 z_2^{-1} + z_1^{-1} z_2,
\qquad\qquad
\overline{Z}_{s_2 s_1 s_2} = 1 + z_2^{-2} + z_1 z_2^{-1},
\\
\overline{Z}_{w_0} = z_1^{-1} z_2^{-1}.
\end{gather*}
We can see that these differ by $\zz^{\rho} = z_1$ from the atoms $A_w := A_w(\zz, \lambda)$:
\begin{gather*}
A_1 = z_1^2 z_2,
\qquad\qquad
A_{s_1} = z_1 z_2^2,
\qquad\qquad
A_{s_2} = z_1^2 z_2^{-1},
\\
A_{s_1 s_2} = z_1 + z_1^{-1} z_2^2 + z_2,
\qquad\qquad
A_{s_2 s_1} = z_1 + z_1 z_2^{-2},
\\
A_{s_1 s_2 s_1} = z_1^{-1} + z_2 + z_2^{-1} + z_1^{-2} z_2,
\qquad\qquad
A_{s_2 s_1 s_2} = z_1^{-1} + z_1^{-1} z_2^{-2} + z_2^{-1},
\\
A_{w_0} = z_1^{-2} z_2^{-1}.
\end{gather*}
\end{example}

\subsection{Lattice models and quantum supergroups}
\label{sec:quantum_groups}

Quantum (super)groups and solvable lattice models are naturally related by identifying the $L$-matrix or $R$-matrix of Boltzmann weights with $R$-matrices coming from representations of affine quantum (super)groups.
In Proposition~\ref{prop:quantummatchingDDGG}, we give a partial quantum supergroup interpretation of the solutions to the Yang--Baxter equation for $R^\Gamma_\Gamma$ and $R^\Delta_\Delta$ in terms of a $q \to 0$ limit of certain $R$-matrices related to representations of the quantum supergroup $\uqsg$.
However, after Proposition~\ref{prop:quantummatchingDDGG}, we discuss why a complete quantum group interpretation is impossible in our setting.  
This is not surprising due to the lack of uniqueness of the $R$-matrix in Figure~\ref{fig:colored_R_matrix_GD}.
Matching lattice models $R$-matrices to quantum group $R$-matrices has possible implication for future work relating lattice models with Iwahori Whittaker functions or Hall--Littlewood polynomials for the symplectic or odd orthogonal groups. For example, in Cartan type A, a quantum group interpretation of the lattice model R-matrix can be used to relate quantum $R$-matrices with $p$-adic intertwining integrals~\cite{BBBGIwahori}.
The $q=0$ results in this section (Proposition~\ref{prop:quantummatchingDDGG}) are a first step towards developing a similar theory in Cartan type B and C.  

Let $R_q := R_q(z_j/z_i)$ be the $\uqsg$ $R$-matrix defined in~\cite[Def.~2.1]{Kojima13} acting on a tensor product of evaluation representations $V_{z_j} \otimes V_{z_i}$, where $V_z$ is a $2n+1$ dimensional super vector space.
We identify the basis of $V_z$ with the spins in our lattice model by the $2n$ colors corresponding to the even subspace and the $0$ for the odd subspace.
We may write Boltzmann weights corresponding to the $R_q(z)$ matrix as in Figure~\ref{fig:R_q_weights}.
Denote by $R_q^{**}$ the $\uqsg$ $R$-matrix acting on the dual of the evaluation representations $V_{q^2 z_j}^* \otimes V_{q^2 z_i}^*$.

\begin{figure}
\[
\begin{array}{c@{\hspace{30pt}}c@{\hspace{30pt}}c@{\hspace{30pt}}c@{\hspace{30pt}}c}
\toprule
  \tt{a}_1&\tt{b}_1&\tt{b}_2&\tt{c}_1&\tt{c}_2\\
\midrule
\begin{tikzpicture}[scale=0.7]
\draw (0,0) to [out = 0, in = 180] (2,2);
\draw (0,2) to [out = 0, in = 180] (2,0);
\draw[fill=white] (0,0) circle (.35);
\draw[fill=white] (0,2) circle (.35);
\draw[fill=white] (2,0) circle (.35);
\draw[fill=white] (2,2) circle (.35);
\node at (0,0) {$0$};
\node at (0,2) {$0$};
\node at (2,2) {$0$};
\node at (2,0) {$0$};
\path[fill=white] (1,1) circle (.3);
\node at (1,1) {$z_i,z_j$};
\end{tikzpicture}&
\begin{tikzpicture}[scale=0.7]
\draw (0,0) to [out = 0, in = 180] (2,2);
\draw (0,2) to [out = 0, in = 180] (2,0);
\draw[fill=white] (0,0) circle (.35);
\draw[line width=0.5mm, brown, fill=white] (0,2) circle (.35);
\draw[line width=0.5mm, brown, fill=white] (2,0) circle (.35);
\draw[fill=white] (2,2) circle (.35);
\node at (0,0) {$0$};
\node at (0,2) {$d$};
\node at (2,2) {$0$};
\node at (2,0) {$d$};
\path[fill=white] (1,1) circle (.3);
\node at (1,1) {$z_i,z_j$};
\end{tikzpicture}&
\begin{tikzpicture}[scale=0.7]
\draw (0,0) to [out = 0, in = 180] (2,2);
\draw (0,2) to [out = 0, in = 180] (2,0);
\draw[line width=0.5mm, brown, fill=white] (0,0) circle (.35);
\draw[fill=white] (0,2) circle (.35);
\draw[line width=0.5mm, brown, fill=white] (2,2) circle (.35);
\draw[fill=white] (2,0) circle (.35);
\node at (0,0) {$d$};
\node at (0,2) {$0$};
\node at (2,2) {$d$};
\node at (2,0) {$0$};
\path[fill=white] (1,1) circle (.3);
\node at (1,1) {$z_i,z_j$};
\end{tikzpicture}&
\begin{tikzpicture}[scale=0.7]
\draw (0,0) to [out = 0, in = 180] (2,2);
\draw (0,2) to [out = 0, in = 180] (2,0);
\draw[line width=0.5mm, brown, fill=white] (0,0) circle (.35);
\draw[fill=white] (0,2) circle (.35);
\draw[fill=white] (2,2) circle (.35);
\draw[line width=0.5mm, brown, fill=white] (2,0) circle (.35);
\node at (0,0) {$d$};
\node at (0,2) {$0$};
\node at (2,2) {$0$};
\node at (2,0) {$d$};
\path[fill=white] (1,1) circle (.3);
\node at (1,1) {$z_i,z_j$};
\end{tikzpicture}&
\begin{tikzpicture}[scale=0.7]
\draw (0,0) to [out = 0, in = 180] (2,2);
\draw (0,2) to [out = 0, in = 180] (2,0);
\draw[line width=0.5mm, brown, fill=white] (0,2) circle (.35);
\draw[fill=white] (0,0) circle (.35);
\draw[fill=white] (2,0) circle (.35);
\draw[line width=0.5mm, brown, fill=white] (2,2) circle (.35);
\node at (0,0) {$0$};
\node at (0,2) {$d$};
\node at (2,2) {$d$};
\node at (2,0) {$0$};
\path[fill=white] (1,1) circle (.3);
\node at (1,1) {$z_i,z_j$};
\end{tikzpicture}\\
   \midrule
   z-q^2 & (1-z)q & (1-z)q & (1-q^2)z & 1-q^2 \\
   \midrule
  \tt{c}'_1&\tt{c}'_2&\tt{b}'_2&\tt{b}'_1&\tt{a}_2\\
   \midrule
\begin{tikzpicture}[scale=0.7]
\draw (0,0) to [out = 0, in = 180] (2,2);
\draw (0,2) to [out = 0, in = 180] (2,0);
\draw[line width=0.5mm, blue, fill=white] (0,0) circle (.35);
\draw[line width=0.5mm, red, fill=white] (0,2) circle (.35);
\draw[line width=0.5mm, red, fill=white] (2,2) circle (.35);
\draw[line width=0.5mm, blue, fill=white] (2,0) circle (.35);
\node at (0,0) {$c'$};
\node at (0,2) {$c$};
\node at (2,2) {$c$};
\node at (2,0) {$c'$};
\path[fill=white] (1,1) circle (.3);
\node at (1,1) {$z_i,z_j$};
\end{tikzpicture}&
\begin{tikzpicture}[scale=0.7]
\draw (0,0) to [out = 0, in = 180] (2,2);
\draw (0,2) to [out = 0, in = 180] (2,0);
\draw[line width=0.5mm, red, fill=white] (0,0) circle (.35);
\draw[line width=0.5mm, blue, fill=white] (0,2) circle (.35);
\draw[line width=0.5mm, blue, fill=white] (2,2) circle (.35);
\draw[line width=0.5mm, red, fill=white] (2,0) circle (.35);
\node at (0,0) {$c$};
\node at (0,2) {$c'$};
\node at (2,2) {$c'$};
\node at (2,0) {$c$};
\path[fill=white] (1,1) circle (.3);
\node at (1,1) {$z_i,z_j$};
\end{tikzpicture}&
\begin{tikzpicture}[scale=0.7]
\draw (0,0) to [out = 0, in = 180] (2,2);
\draw (0,2) to [out = 0, in = 180] (2,0);
\draw[line width=0.5mm, red, fill=white] (0,0) circle (.35);
\draw[line width=0.5mm, blue, fill=white] (0,2) circle (.35);
\draw[line width=0.5mm, red, fill=white] (2,2) circle (.35);
\draw[line width=0.5mm, blue, fill=white] (2,0) circle (.35);
\node at (0,0) {$c$};
\node at (0,2) {$c'$};
\node at (2,2) {$c$};
\node at (2,0) {$c'$};
\path[fill=white] (1,1) circle (.3);
\node at (1,1) {$z_i,z_j$};
\end{tikzpicture}&
\begin{tikzpicture}[scale=0.7]
\draw (0,0) to [out = 0, in = 180] (2,2);
\draw (0,2) to [out = 0, in = 180] (2,0);
\draw[line width=0.5mm, blue, fill=white] (0,0) circle (.35);
\draw[line width=0.5mm, red, fill=white] (0,2) circle (.35);
\draw[line width=0.5mm, blue, fill=white] (2,2) circle (.35);
\draw[line width=0.5mm, red, fill=white] (2,0) circle (.35);
\node at (0,0) {$c'$};
\node at (0,2) {$c$};
\node at (2,2) {$c'$};
\node at (2,0) {$c$};
\path[fill=white] (1,1) circle (.3);
\node at (1,1) {$z_i,z_j$};
\end{tikzpicture}&
\begin{tikzpicture}[scale=0.7]
\draw (0,0) to [out = 0, in = 180] (2,2);
\draw (0,2) to [out = 0, in = 180] (2,0);
\draw[line width=0.5mm, brown, fill=white] (0,0) circle (.35);
\draw[line width=0.5mm, brown, fill=white] (0,2) circle (.35);
\draw[line width=0.5mm, brown, fill=white] (2,2) circle (.35);
\draw[line width=0.5mm, brown, fill=white] (2,0) circle (.35);
\node at (0,0) {$d$};
\node at (0,2) {$d$};
\node at (2,2) {$d$};
\node at (2,0) {$d$};
\path[fill=white] (1,1) circle (.3);
\node at (1,1) {$z_i,z_j$};
\end{tikzpicture}\\
   \midrule
   -(1-q^2)z & -(1-q^2) & (1-z)q & (1-z)q & q^2z-1 \\
   \bottomrule
\end{array}
\]
\caption{The $R_q$-matrix in~\cite{Kojima13} with $z = z_j/z_i$, with ${\color{red} c} > {\color{blue} c'}$, and with ${\color{brown} d}$ being any color.}
\label{fig:R_q_weights}
\end{figure}

\begin{proposition}\label{prop:quantummatchingDDGG}
Under a certain Drinfeld twist of $\uqsg$, the $R$-matrix $R^\Gamma_\Gamma$ is the $q\to0$ limit of the $R_q$ and $R^\Delta_\Delta$ is the $q \to 0$ limit of $R_q^{**}$.
\end{proposition} 

\begin{proof}
To obtain $R^\Gamma_\Gamma$ in Figure~\ref{fig:colored_R_matrix_GG} as the $q \to 0$ limit of $R_q$ in Figure~\ref{fig:R_q_weights}, perform the following manipulations on $R_q$, where $z = z_j / z_i$:
\begin{enumerate}
\item Multiply all fully colored states by $-1$ (this corresponds to passing from a graded solution of the Yang--Baxter equation to an ungraded solution of the Yang--Baxter equation).
\item Multiply weights $\tt{c}_1$ and $\tt{c}'_1$ by $z^{-1}$ and weights $\tt{c}_2$ and $\tt{c}'_2$ by $z$ (this corresponds to a change of basis in $V_z$ and does not affect the quantum group structure). 
\item Multiply weights $\tt{b}_1$ by $q$ and weight $\tt{b}_2$ by $q^{-1}$; multiply weight $\tt{b}'_1$ by $-q$ and weight $\tt{b}'_2$ by $-q^{-1}$ (this corresponds to a Drinfeld twist which affects the quantum group structure, namely its comultiplication, and universal $R$-matrix~\cite{Drinfeldtwist,ReshetikhinDrinfeldtwist}).
\item Take the limit $q \to 0$.
\end{enumerate}

To compute the $q \to 0$ limit of $R_q^{**}$, we use two standard facts from the theory of quantum groups.
The first fact is that given a quantum group (more precisely a quasitriangular Hopf algebra) representation $V$ with $v \in V$ and its dual $V^*$ with $v^* \in V^*$, then for any element $h$ of the quantum group we have that $h \cdot v^*(v) = v^* (S(h) \cdot v)$, where $S$ is the antipode.
The second fact we will need is the property that the universal $R$-matrix $\univR$ satisfies the relation
\[
(S \otimes S) \univR = \univR,
\] 
which is proved, for example, in~\cite[Prop.~4.2.7]{CPbook}.
These immediately imply that if $R^{\Gamma}_{\Gamma}$ is the $q \to 0$ limit of the $R_q$, then $(R^{\Gamma}_{\Gamma})^t$ is the $q \to 0$ limit of the $R$-matrix $R_q^{**}$.

Note that the $R$-matrix corresponding to $V_{z_j} \otimes V_{z_i}$ only depends on $z_j / z_i$, therefore the $R$-matrices corresponding to $V^*_{z_j} \otimes V^*_{z_i}$ and $V^*_{q^2 z_j} \otimes V^*_{q^2 z_i}$ will be equal.
What remains to show then is the relation $(R^{\Gamma}_{\Gamma})^t = R^{\Delta}_{\Delta}$.
This can be seen by comparing Figure~\ref{fig:colored_R_matrix_GG} and Figure~\ref{fig:colored_R_matrix_DD} and using the fact that taking transpose of an $R$-matrix modifies the states as follows:
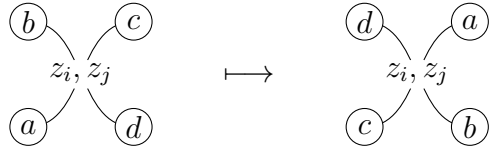
\[
\begin{tikzpicture}[scale=0.7,baseline=17]
\draw (0,0) to [out = 0, in = 180] (2,2);
\draw (0,2) to [out = 0, in = 180] (2,0);
\draw[fill=white] (0,0) circle (.35);
\draw[fill=white] (0,2) circle (.35);
\draw[fill=white] (2,0) circle (.35);
\draw[fill=white] (2,2) circle (.35);
\node at (0,0) {$a$};
\node at (0,2) {$b$};
\node at (2,2) {$c$};
\node at (2,0) {$d$};
\path[fill=white] (1,1) circle (.3);
\node at (1,1) {$z_i,z_j$};
\end{tikzpicture}
\qquad \longmapsto \qquad
\begin{tikzpicture}[scale=0.7, baseline=17]
\draw (0,0) to [out = 0, in = 180] (2,2);
\draw (0,2) to [out = 0, in = 180] (2,0);
\draw[fill=white] (0,0) circle (.35);
\draw[fill=white] (0,2) circle (.35);
\draw[fill=white] (2,0) circle (.35);
\draw[fill=white] (2,2) circle (.35);
\node at (0,0) {$c$};
\node at (0,2) {$d$};
\node at (2,2) {$a$};
\node at (2,0) {$b$};
\path[fill=white] (1,1) circle (.3);
\node at (1,1) {$z_i,z_j$};
\end{tikzpicture}
\]
\end{proof}

In order to argue that we can interpret the horizontal edges in terms of $\uqsg$-representations, we also need to examine $R^\Delta_\Gamma$ and $R^\Gamma_\Delta$.
The $R$-matrix $R^\Gamma_\Delta$ would correspond to the $R$-matrix $R_q^*$ associated to $V^*_{q^2z_1} \otimes V_{z_2}$. 
One may compute this $R$-matrix by using
\[
(S \otimes \operatorname{id}) \univR = \univR^{-1}.
\]
We also note that the $R$-matrix can be explicitly computed from~\cite[Eq.~(2.9)]{Zhang17}.
The $\tt{a}_1$ and $\tt{a}_2$ entries of this $R$-matrix will be (following Zhang~\cite{Zhang17}) $bq-aq^{-1}$ and $bq^{-1}-aq$. 
In trying to match these two entries with the corresponding entries in the $R^\Gamma_\Delta$ matrix, we can rescale both entries by a factor and set the parameters $a$ and $b$. 
Taking the limit $q \to 0$ or $q\to \infty$ we will not be able to match them with the factors $z_i -z_j$ and $z_j$ in a way that no factor will blow up or $\tt{b}_1 \neq 0$.

Finally, we may consider the $R$-matrix $R^{\Delta}_{\Gamma}$, which corresponds to the $R$-matrix associated to $V_{z_1} \otimes V^*_{q^2z_2}$.
We compute this quantum $R$-matrix explicitly for $n = 1,2,3$ by directly inverting the $R$-matrix given by~\cite[Eq.~(2.9)]{Zhang17} (this uses the unitarity of the affine $R$-matrix).
In particular for $n = 2$, the $\tt{a}_1$ and $\tt{a}_2$ entries of the $R$-matrix for $V_b \otimes V^*_a$ are $q^7 a - q^{-1} b$ and $q^5 a - q b$, respectively.
By a similar argument to the previous paragraph, we are also unable to obtain the desired $R^\Delta_\Gamma$-matrix (see Figure~\ref{fig:colored_R_matrix_DG}).

The failure to simultaneously match the $\Gamma$ and $\Delta$ horizontal rows with representations of the quantum affine group $U_q\bigl( \widehat{\mathfrak{gl}}(2n|1) \bigr)$ is in accordance with the non-uniqueness of the $R^{\Gamma}_{\Delta}$ matrix in Figure~\ref{fig:colored_R_matrix_GD}. 
If we had a match we would expect to compute a unique $R^{\Gamma}_{\Delta}$-matrix. 
A similar phenomena appears in the work of Zhong~\cite{Zhong21}.

One may also try to relate our $K$-matrix with affine $K$-matrices corresponding to a version of quantum symmetric pairs.
Unfortunately, there is not much research on super versions of quantum affine pairs, so there is not much we can say on this subject. 

\section{Colored lattice models and Demazure characters}
\label{sec:colored_Demazure}

In this section we construct a colored lattice model for a Demazure character by modifying our previous lattice model.
Consider our previous model, but replace the $\tt{a}_2$ vertices in both the $\Gamma$ $L$-matrix (Figure~\ref{fig:colored_gamma_weights}) and the $\Delta$ $L$-matrix (Figure~\ref{fig:colored_delta_weights}) with
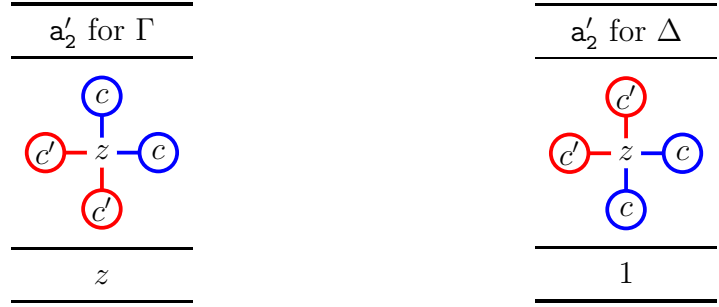
\[
\begin{array}{c}
\toprule
{\tt a'_2} \text{ for } \Gamma \\
\midrule
\begin{tikzpicture}
\coordinate (a) at (-.75, 0);
\coordinate (b) at (0, .75);
\coordinate (c) at (.75, 0);
\coordinate (d) at (0, -.75);
\coordinate (aa) at (-.75,.5);
\coordinate (cc) at (.75,.5);
\draw[line width=0.5mm, red] (a)--(0,0);
\draw[line width=0.5mm, blue] (b)--(0,0);
\draw[line width=0.5mm, blue] (c)--(0,0);
\draw[line width=0.5mm, red] (d)--(0,0);
\draw[line width=0.5mm, red,fill=white] (a) circle (.25);
\draw[line width=0.5mm, blue,fill=white] (b) circle (.25);
\draw[line width=0.5mm, blue, fill=white] (c) circle (.25);
\draw[line width=0.5mm, red, fill=white] (d) circle (.25);
\node at (0,1) { };
\node at (a) {$c'$};
\node at (b) {$c$};
\node at (c) {$c$};
\node at (d) {$c'$};
\path[fill=white] (0,0) circle (.2);
\node at (0,0) {$z$};
\end{tikzpicture}
\\\midrule
z
\\\bottomrule
\end{array}
\hspace{130pt}
\begin{array}{c}
\toprule
{\tt a'_2} \text{ for } \Delta \\
\midrule
\begin{tikzpicture}
\coordinate (a) at (-.75, 0);
\coordinate (b) at (0, .75);
\coordinate (c) at (.75, 0);
\coordinate (d) at (0, -.75);
\coordinate (aa) at (-.75,.5);
\coordinate (cc) at (.75,.5);
\draw[line width=0.5mm, red] (a)--(0,0);
\draw[line width=0.5mm, red] (b)--(0,0);
\draw[line width=0.5mm, blue] (c)--(0,0);
\draw[line width=0.5mm, blue] (d)--(0,0);
\draw[line width=0.5mm, red,fill=white] (a) circle (.25);
\draw[line width=0.5mm, red,fill=white] (b) circle (.25);
\draw[line width=0.5mm, blue, fill=white] (c) circle (.25);
\draw[line width=0.5mm, blue, fill=white] (d) circle (.25);
\node at (0,1) { };
\node at (a) {$c'$};
\node at (b) {$c'$};
\node at (c) {$c$};
\node at (d) {$c$};
\path[fill=white] (0,0) circle (.2);
\node at (0,0) {$z$};
\end{tikzpicture}
\\\midrule
1
\\\bottomrule
\end{array}
\]
We also change our $K$-matrix in both types $B$ and $C$ by replacing $\tt{k}_3$ with
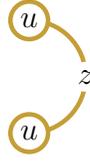
\[
\begin{tikzpicture}
\coordinate (b) at (0, .75);
\coordinate (c) at (.75, 0);
\coordinate (d) at (0, -.75);
\coordinate (aa) at (-.75,.5);
\coordinate (cc) at (.75,.5);
\draw[line width=0.6mm, UQgold] (b) to[out=0,in=90] (c);
\draw[line width=0.6mm, UQgold] (c) to[out=-90,in=0] (d);
\draw[line width=0.6mm, UQgold,fill=white] (b) circle (.25);
\path[fill=white] (c) circle (.2);
\draw[line width=0.6mm, UQgold,fill=white] (d) circle (.25);
\node at (0,1) { };
\node at (b) {$u$};
\node at (c) {$z$};
\node at (d) {$u$};
\end{tikzpicture}
\]
and keeping the Boltzmann weight as $1$.
We denote this new $K$-matrix configuration as $\tt{k}_3'$.

Let $\states^X_{\lambda, w}$ denote the new model using these new $L$-matrices and $K$-matrix.
Analogous to $\overline{\states}^X_{\lambda, w}$, we will often write this simply as $\states_{\lambda, w}$.
This causes the lower left two values in the $R^{\Gamma}_{\Gamma}$-matrix and $R^{\Delta}_{\Delta}$-matrix to swap values.
Therefore, the R-matrices that satisfy the Yang--Baxter equation in the new model will be
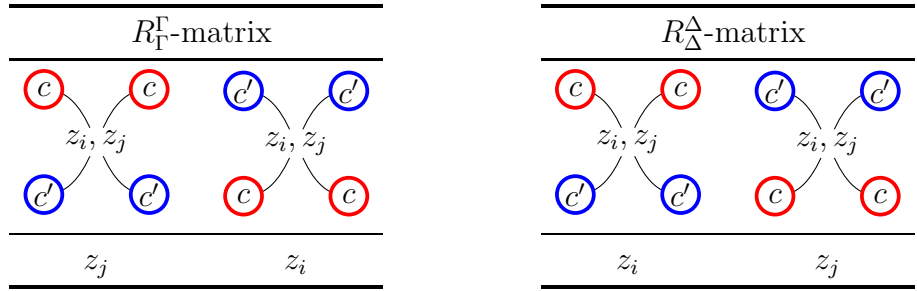
\[
\begin{array}{c@{\hspace{20pt}}c}
\toprule
\multicolumn{2}{c}{\text{ $R^{\Gamma}_{\Gamma}$-matrix} } \\
\midrule
\begin{tikzpicture}[scale=0.7]
\draw (0,0) to [out = 0, in = 180] (2,2);
\draw (0,2) to [out = 0, in = 180] (2,0);
\draw[line width=0.5mm, blue, fill=white] (0,0) circle (.35);
\draw[line width=0.5mm, red, fill=white] (0,2) circle (.35);
\draw[line width=0.5mm, red, fill=white] (2,2) circle (.35);
\draw[line width=0.5mm, blue, fill=white] (2,0) circle (.35);
\node at (0,0) {$c'$};
\node at (0,2) {$c$};
\node at (2,2) {$c$};
\node at (2,0) {$c'$};
\path[fill=white] (1,1) circle (.3);
\node at (1,1) {$z_i,z_j$};
\end{tikzpicture}&
\begin{tikzpicture}[scale=0.7]
\draw (0,0) to [out = 0, in = 180] (2,2);
\draw (0,2) to [out = 0, in = 180] (2,0);
\draw[line width=0.5mm, red, fill=white] (0,0) circle (.35);
\draw[line width=0.5mm, blue, fill=white] (0,2) circle (.35);
\draw[line width=0.5mm, blue, fill=white] (2,2) circle (.35);
\draw[line width=0.5mm, red, fill=white] (2,0) circle (.35);
\node at (0,0) {$c$};
\node at (0,2) {$c'$};
\node at (2,2) {$c'$};
\node at (2,0) {$c$};
\path[fill=white] (1,1) circle (.3);
\node at (1,1) {$z_i,z_j$};
\end{tikzpicture}
\\\midrule
z_j & z_i
\\\bottomrule
\end{array}
\hspace{60pt}
\begin{array}{c@{\hspace{20pt}}c}
\toprule
\multicolumn{2}{c}{\text{ $R^{\Delta}_{\Delta}$-matrix} } \\
\midrule
\begin{tikzpicture}[scale=0.7]
\draw (0,0) to [out = 0, in = 180] (2,2);
\draw (0,2) to [out = 0, in = 180] (2,0);
\draw[line width=0.5mm, blue, fill=white] (0,0) circle (.35);
\draw[line width=0.5mm, red, fill=white] (0,2) circle (.35);
\draw[line width=0.5mm, red, fill=white] (2,2) circle (.35);
\draw[line width=0.5mm, blue, fill=white] (2,0) circle (.35);
\node at (0,0) {$c'$};
\node at (0,2) {$c$};
\node at (2,2) {$c$};
\node at (2,0) {$c'$};
\path[fill=white] (1,1) circle (.3);
\node at (1,1) {$z_i,z_j$};
\end{tikzpicture}&
\begin{tikzpicture}[scale=0.7]
\draw (0,0) to [out = 0, in = 180] (2,2);
\draw (0,2) to [out = 0, in = 180] (2,0);
\draw[line width=0.5mm, red, fill=white] (0,0) circle (.35);
\draw[line width=0.5mm, blue, fill=white] (0,2) circle (.35);
\draw[line width=0.5mm, blue, fill=white] (2,2) circle (.35);
\draw[line width=0.5mm, red, fill=white] (2,0) circle (.35);
\node at (0,0) {$c$};
\node at (0,2) {$c'$};
\node at (2,2) {$c'$};
\node at (2,0) {$c$};
\path[fill=white] (1,1) circle (.3);
\node at (1,1) {$z_i,z_j$};
\end{tikzpicture}
\\\midrule
 z_i & z_j
\\\bottomrule
\end{array}
\]
with all weights that are not listed above remaining the same as in Figures~\ref{fig:colored_R_matrix_GG},~\ref{fig:colored_R_matrix_DD},~\ref{fig:colored_R_matrix_DG} and~\ref{fig:colored_R_matrix_GD}.
A direct check shows that these modified $R$-matrices and $K$-matrix still satisfy the corresponding reflection and unitary equations.

By the same argument as in Lemma~\ref{lemma:type_A_relation}, we can prove the following functional equation for the Demazure character model.
\begin{lemma}
\label{lemma:type_A_relation_character}
Choose $i < n$, $j=i+1$ and $w \in W$ such that $\ell(s_i w) = \ell(w) + 1$.
Then we have
\begin{equation}\label{eq:typeAfe_character}
(z_i - z_j) Z(\states_{\lambda, s_i w}; \zz) =  z_j \bigl( Z(\states_{\lambda, w}; \zz) - z_i z_j^{-1} Z(\states_{\lambda, w}; s_i \zz) \bigr).
\end{equation}
\end{lemma}


Now we look at the corresponding version of the fish equation.

\begin{lemma}
\label{lemma:type_BC_relation_character}
Choose $w \in W$ such that $\ell(s_n w) = \ell(w) + 1$.
Then we have
\begin{align*}
(z_n^2 - 1) Z(\states_{\lambda, s_n w}; \zz) & = z_n^2 Z(\states_{\lambda, w}; \zz) - Z(\states_{\lambda, w}; s_n \zz) && (\text{type C}), \\
(z_n - 1) Z(\states_{\lambda, s_n w}; \zz) & = z_n Z(\states_{\lambda, w}; \zz) - Z(\states_{\lambda, w}; s_n \zz) && (\text{type B}).
\end{align*}
\end{lemma}

\begin{proof}
We follow the same procedure as in the proof of Lemma~\ref{lemma:type_BC_relation} and use the same notation.
In this case, we instead have
\[
Z(\states_{\lambda, w}; \zz) = K \overline{H}_1^{-1} Z(\widetilde{\states}_{\overline{\tt{a}}_1}; \zz) + H_2^{-1} Z(\widetilde{\states}_{\tt{a}_2}; \zz).
\]
By applying the standard train argument (see Figure~\ref{fig:fishy_argument}), we see that there are now only three possible fish:
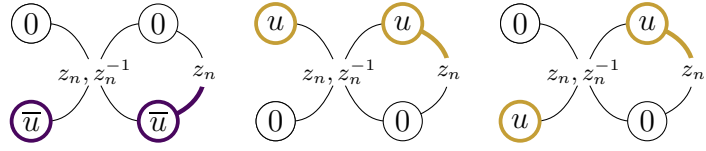
\[
\begin{tikzpicture}[scale=0.65]
\begin{scope}[shift={(0,0)}]
  \draw (0,1) to [out = 0, in = 180] (2,3);
  \draw (0,3) to [out = 0, in = 180] (2,1);
  \draw[-] (2,3) to[out=0,in=90] (3.25,2);
  \draw[line width=0.6mm, UQpurple] (3.25,2) to[out=-90,in=0] (2,1);
  \fill[white] (3.25,2) circle (.3);
  \node at (3.25,2) {\scriptsize $z_n$};
  \draw[fill=white] (-0.3,1) circle (.4);
  \draw[fill=white] (-0.3,3) circle (.4);
  \draw[fill=white] (2.3,3) circle (.4);
  \draw[line width=0.5mm,UQpurple,fill=white] (2.3,1) circle (.4);
  \draw[line width=0.5mm,UQpurple,fill=white] (-0.3,1) circle (.4);
  \draw[fill=white] (-0.3,3) circle (.4);
  \path[fill=white] (1,2) circle (.3);
  \node at (1,2) {\scriptsize$z_n,z_n^{-1}$};
  \node at (2.3,1) {$\overline{u}$};
  \node at (2.3,3) {$0$};
  \node at (-0.3,1) {$\overline{u}$};
  \node at (-0.3,3) {$0$};
\end{scope}
\begin{scope}[shift={(10,0)}]
  \draw (0,1) to [out = 0, in = 180] (2,3);
  \draw (0,3) to [out = 0, in = 180] (2,1);
  \draw[line width=0.6mm, UQgold] (2,3) to[out=0,in=90] (3.25,2);
  \draw[-] (3.25,2) to[out=-90,in=0] (2,1);
  \fill[white] (3.25,2) circle (.3);
  \node at (3.25,2) {\scriptsize $z_n$};
  \draw[fill=white] (-0.3,1) circle (.4);
  \draw[fill=white] (-0.3,3) circle (.4);
  \draw[line width=0.5mm,UQgold,fill=white] (2.3,3) circle (.4);
  \draw[fill=white] (2.3,1) circle (.4);
  \draw[line width=0.5mm,UQgold,fill=white] (-0.3,1) circle (.4);
  \draw[fill=white] (-0.3,3) circle (.4);
  \path[fill=white] (1,2) circle (.3);
  \node at (1,2) {\scriptsize$z_n,z_n^{-1}$};
  \node at (2.3,1) {$0$};
  \node at (2.3,3) {$u$};
  \node at (-0.3,1) {$u$};
  \node at (-0.3,3) {$0$};
\end{scope}
\begin{scope}[shift={(5,0)}]
  \draw (0,1) to [out = 0, in = 180] (2,3);
  \draw (0,3) to [out = 0, in = 180] (2,1);
  \draw[line width=0.6mm, UQgold] (2,3) to[out=0,in=90] (3.25,2);
  \draw[-] (3.25,2) to[out=-90,in=0] (2,1);
  \fill[white] (3.25,2) circle (.3);
  \node at (3.25,2) {\scriptsize $z_n$};
  \draw[fill=white] (-0.3,1) circle (.4);
  \draw[fill=white] (-0.3,3) circle (.4);
  \draw[line width=0.5mm,UQgold,fill=white] (2.3,3) circle (.4);
  \draw[fill=white] (2.3,1) circle (.4);
  \draw[fill=white] (-0.3,1) circle (.4);
  \draw[line width=0.5mm,UQgold,fill=white] (-0.3,3) circle (.4);
  \path[fill=white] (1,2) circle (.3);
  \node at (1,2) {\scriptsize$z_n,z_n^{-1}$};
  \node at (2.3,1) {$0$};
  \node at (2.3,3) {$u$};
  \node at (-0.3,1) {$0$};
  \node at (-0.3,3) {$u$};
\end{scope}
\end{tikzpicture}
\]
where $u$ is an unbarred color.
Hence, by the same computation as in the atom case, we have
\begin{align*}
Z(\states_{\lambda, w}; \zz) & =  K z_n^2 Z(\widetilde{\states}_{\overline{u}}^0; s_n \zz) + z_n\bigl( (z_n - z_n^{-1}) Z(\widetilde{\states}_u^0; s_n \zz) + z_n^{-1} Z(\widetilde{\states}_0^u; s_n \zz) \bigr)
\\ & = K z_n^2 Z(\widetilde{\states}_{\overline{u}}^0; s_n \zz) + (z_n^2 - 1) Z(\widetilde{\states}_u^0; s_n \zz) + Z(\widetilde{\states}_0^u; s_n \zz).
\end{align*}
Next, we note there are two choices when converting the $^{\Gamma}_{\Gamma}$ version $\widetilde{\states}_0^u$ back to the $^{\Gamma}_{\Delta}$ model, and thus we have
\[
Z(\widetilde{\states}_0^u(w); s_n \zz) = \zeta Z(\states_{\overline{u}}^u(w); s_n \zz) + (1 - \zeta) Z(\states_u^u(s_n w); s_n \zz)
\]
for some parameter $\zeta$ as the two partition functions on the right hand side are equal.
Using this, we have
\begin{align*}
Z(\states_{\lambda, w}; \zz) & = K z_n^2 \overline{K}^{-1} Z(\states_0^0(w); s_n \zz) + \zeta Z(\states_{\overline{u}}^u(w); s_n \zz)
  \\ & \hspace{20pt} + (1 - \zeta) Z(\states_u^u(s_n w); s_n \zz) + (z_n^2 - 1) \overline{K}^{-1} Z(\states_0^0(s_n w); s_n \zz).
\end{align*}
In order to obtain the desired partition functions, we require that
\begin{subequations}
\label{eq:K_alpha_relations}
\begin{align}
1 - \zeta & = (z_n^2 - 1) \overline{K}^{-1},
\\
\zeta & = K z_n^2 \overline{K}^{-1},
\end{align}
\end{subequations}
which hold if and only if $K$ satisfies
\begin{equation}
\label{eq:K_matrix_check}
1 = \bigl( (K+1) z_n^2 - 1 \bigr) \overline{K}^{-1}.
\end{equation}
Indeed, using the relations~\eqref{eq:K_alpha_relations}, we see that
\begin{align*}
Z(\states_{\lambda, w}; \zz) & = \zeta Z(\states_{\lambda, w}; s_n \zz) + (1- \zeta) Z(\states_{\lambda, s_n w}; s_n \zz),
\\ (1 - \zeta) Z(\states_{\lambda, s_n w}; s_n \zz) & = Z(\states_{\lambda, w}; \zz) - \zeta Z(\states_{\lambda, w}; s_n \zz),
\\ (\overline{\zeta} - 1) Z(\states_{\lambda, s_n w}; \zz) & = \overline{\zeta} Z(\states_{\lambda, w}; \zz) - Z(\states_{\lambda, w}; s_n \zz),
\end{align*}
where $\overline{\zeta}(z_n) = \zeta(z_n^{-1})$ and we applied $z_n \leftrightarrow z_n^{-1}$ to the second equation.
We see that $K = z_n^{-2}$ (resp.~$K = z_n^{-1} + z_n^{-2}$) satisfies Equation~\eqref{eq:K_matrix_check} with $\zeta = z_n^{-2}$ (resp.~$\zeta = z_n^{-1})$ for type $C$ (resp.~$B$).
Hence, we obtain the desired functional equations for types $C$ and $B$.
\end{proof}

Thus, we can prove our second main result.

\begin{theorem}
\label{thm:partition_demazure_BC}
For Cartan type $X \in \{B, C\}$, we have
\[
Z(\states^X_{\lambda, w}; \zz) = \zz^{\rho} D_w(\zz, \lambda).
\]
\end{theorem}

\begin{proof}
The proof uses induction and Lemmas~\ref{lemma:type_A_relation_character} and~\ref{lemma:type_BC_relation_character}.
As it is similar to the proof of Theorem~\ref{thm:partition_atom_BC} and a straightforward computation, we omit the details.
\end{proof}

\begin{example}
Let $\lambda = (2, 1)$.
The following are all possible states for $\states_{\lambda, s_2 s_1 s_2}$:
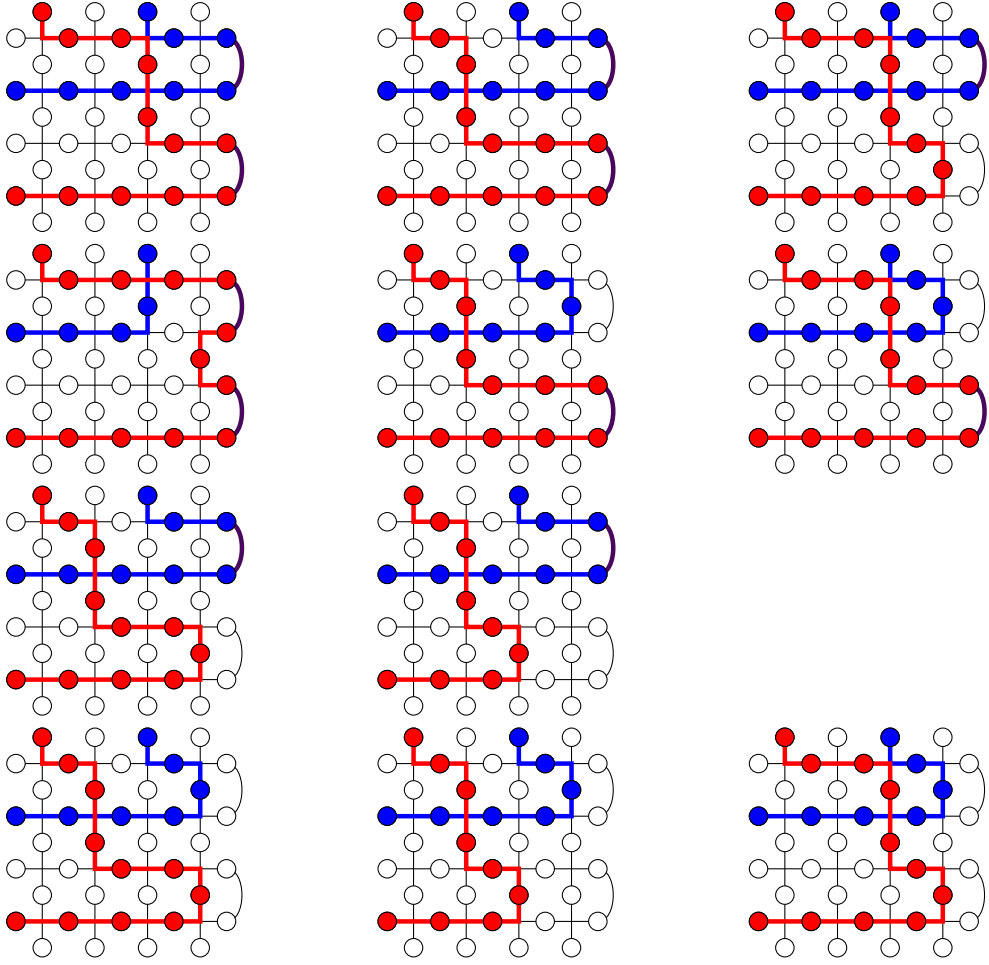
\ifexamples
\newcommand{\cs}{.35} 
\newcommand{\lw}{0.6mm} 
\newcommand{\thescale}{0.35} 
\newcommand{\base}{
  \foreach \y in {1,5} {
    \draw[-] (8,\y) to [out=0, in=0] (8,\y+2);
  }
  \foreach \x in {1,3,5,7}{
    \draw[-] (\x,0) -- (\x,8);
    \foreach \y in {0,2,4,6,8}
      \draw[fill=white] (\x,\y) circle (\cs);
  }
  \foreach \y in {1,3,5,7} {
    \draw[-] (0,\y) -- (8,\y);
    \foreach \x in {0,2,4,6,8}
      \draw[fill=white] (\x,\y) circle (\cs);
  }
}
\begin{align*}
&
\begin{tikzpicture}[scale=\thescale,baseline=40]
  \base
  \draw[-, UQpurple, line width=\lw] (8,7) to[out=0, in=0] (8,5);
  \draw[-, UQpurple, line width=\lw] (8,3) to[out=0, in=0] (8,1);
  \draw[-,blue,line width=\lw] (5,8) -- (5,7) -- (8,7);
  \draw[-,blue,line width=\lw] (8,5) -- (0,5);
  \draw[-,red,line width=\lw] (1,8) -- (1,7) -- (5,7) -- (5,3) -- (8,3);
  \draw[-,red,line width=\lw] (8,1) -- (0,1);
  \draw[fill=red] (1,8) circle (\cs);
  \foreach \x in {2,4}
    \draw[fill=red] (\x,7) circle (\cs);
  \draw[fill=red] (5,6) circle (\cs);
  \draw[fill=red] (5,4) circle (\cs);
  \draw[fill=red] (6,3) circle (\cs);
  \draw[fill=red] (8,3) circle (\cs);
  \foreach \x in {0,2,4,6,8}
    \draw[fill=red] (\x,1) circle (\cs);
  \draw[fill=blue] (5,8) circle (\cs);
  \foreach \x in {6,8}
    \draw[fill=blue] (\x,7) circle (\cs);
  \foreach \x in {0,2,4,6,8}
    \draw[fill=blue] (\x,5) circle (\cs);
\end{tikzpicture}
&&
\begin{tikzpicture}[scale=\thescale,baseline=40]
  \base
  \draw[-, UQpurple, line width=\lw] (8,7) to[out=0, in=0] (8,5);
  \draw[-, UQpurple, line width=\lw] (8,3) to[out=0, in=0] (8,1);
  \draw[-,blue,line width=\lw] (5,8) -- (5,7) -- (8,7);
  \draw[-,blue, line width=\lw] (8,5) -- (0,5);
  \draw[-,red,line width=\lw] (1,8) -- (1,7) -- (3,7) -- (3,3) -- (8,3);
  \draw[-,red, line width=\lw] (8,1) -- (0,1);
  \draw[fill=red] (1,8) circle (\cs);
  \draw[fill=red] (2,7) circle (\cs);
  \foreach \y in {4,6}
    \draw[fill=red] (3,\y) circle (\cs);
  \foreach \x in {4,6,8}
    \draw[fill=red] (\x,3) circle (\cs);
  \foreach \x in {0,2,4,6,8}
    \draw[fill=red] (\x,1) circle (\cs);
  \draw[fill=blue] (5,8) circle (\cs);
  \foreach \x in {6,8}
    \draw[fill=blue] (\x,7) circle (\cs);
  \foreach \x in {0,2,4,6,8}
    \draw[fill=blue] (\x,5) circle (\cs);
\end{tikzpicture}
&&
\begin{tikzpicture}[scale=\thescale,baseline=40]
  \base
  \draw[-, UQpurple, line width=\lw] (8,7) to[out=0, in=0] (8,5);
  \draw[-,blue,line width=\lw] (5,8) -- (5,7) -- (8,7);
  \draw[-,blue,line width=\lw] (8,5) -- (0,5);
  \draw[-,red,line width=\lw] (1,8) -- (1,7) -- (5,7) -- (5,3) -- (7,3) -- (7,1) -- (0,1);
  \draw[fill=red] (1,8) circle (\cs);
  \foreach \x in {2,4}
    \draw[fill=red] (\x,7) circle (\cs);
  \draw[fill=red] (5,6) circle (\cs);
  \draw[fill=red] (5,4) circle (\cs);
  \draw[fill=red] (6,3) circle (\cs);
  \draw[fill=red] (7,2) circle (\cs);
  \foreach \x in {0,2,4,6}
    \draw[fill=red] (\x,1) circle (\cs);
  \draw[fill=blue] (5,8) circle (\cs);
  \foreach \x in {6,8}
    \draw[fill=blue] (\x,7) circle (\cs);
  \foreach \x in {0,2,4,6,8}
    \draw[fill=blue] (\x,5) circle (\cs);
\end{tikzpicture}
\allowdisplaybreaks
\\ &
\begin{tikzpicture}[scale=\thescale,baseline=40]
  \base
  \draw[-, UQpurple, line width=\lw] (8,7) to[out=0, in=0] (8,5);
  \draw[-, UQpurple, line width=\lw] (8,1) to[out=0, in=0] (8,3);
  \draw[-,blue,line width=\lw] (5,8) -- (5,5) -- (0,5);
  \draw[-,red,line width=\lw] (1,8) -- (1,7) -- (8,7);
  \draw[-,red, line width=\lw] (8,5) -- (7,5) -- (7,3) -- (8,3);
  \draw[-,red, line width=\lw] (8,1) -- (0,1);
  \draw[fill=red] (1,8) circle (\cs);
  \foreach \x in {2,4,6,8}
    \draw[fill=red] (\x,7) circle (\cs);
  \draw[fill=red] (8,5) circle (\cs);
  \draw[fill=red] (7,4) circle (\cs);
  \draw[fill=red] (8,3) circle (\cs);
  \foreach \x in {0,2,4,6,8}
    \draw[fill=red] (\x,1) circle (\cs);
  \draw[fill=blue] (5,8) circle (\cs);
  \draw[fill=blue] (5,6) circle (\cs);
  \foreach \x in {0,2,4}
    \draw[fill=blue] (\x,5) circle (\cs);
\end{tikzpicture}
&&
\begin{tikzpicture}[scale=\thescale,baseline=40]
  \base
  \draw[-, UQpurple, line width=\lw] (8,3) to[out=0, in=0] (8,1);
  \draw[-,blue,line width=\lw] (5,8) -- (5,7) -- (7,7) -- (7,5) -- (0,5);
  \draw[-,red,line width=\lw] (1,8) -- (1,7) -- (3,7) -- (3,3) -- (8,3);
  \draw[-,red, line width=\lw] (8,1) -- (0,1);
  \draw[fill=red] (1,8) circle (\cs);
  \draw[fill=red] (2,7) circle (\cs);
  \foreach \y in {4,6}
    \draw[fill=red] (3,\y) circle (\cs);
  \foreach \x in {4,6,8}
    \draw[fill=red] (\x,3) circle (\cs);
  \foreach \x in {0,2,4,6,8}
    \draw[fill=red] (\x,1) circle (\cs);
  \draw[fill=blue] (5,8) circle (\cs);
  \draw[fill=blue] (6,7) circle (\cs);
  \draw[fill=blue] (7,6) circle (\cs);
  \foreach \x in {0,2,4,6}
  \draw[fill=blue] (\x,5) circle (\cs);
\end{tikzpicture}
&&
\begin{tikzpicture}[scale=\thescale,baseline=40]
  \base
  \draw[-, UQpurple, line width=\lw] (8,3) to[out=0, in=0] (8,1);
  \draw[-,blue,line width=\lw] (5,8) -- (5,7) -- (7,7) -- (7,5) -- (0,5);
  \draw[-,red,line width=\lw] (1,8) -- (1,7) -- (5,7) -- (5,3) -- (8,3);
  \draw[-,red, line width=\lw] (8,1) -- (0,1);
  \draw[fill=red] (1,8) circle (\cs);
  \foreach \x in {2,4}
    \draw[fill=red] (\x,7) circle (\cs);
  \foreach \y in {6,4}
    \draw[fill=red] (5,\y) circle (\cs);
  \foreach \x in {6,8}
    \draw[fill=red] (\x,3) circle (\cs);
  \foreach \x in {0,2,4,6,8}
    \draw[fill=red] (\x,1) circle (\cs);
  \draw[fill=blue] (5,8) circle (\cs);
  \draw[fill=blue] (6,7) circle (\cs);
  \draw[fill=blue] (7,6) circle (\cs);
  \foreach \x in {0,2,4,6}
    \draw[fill=blue] (\x,5) circle (\cs);
\end{tikzpicture}
\allowdisplaybreaks
\\ &
\begin{tikzpicture}[scale=\thescale,baseline=40]
  \base
  \draw[-, UQpurple, line width=\lw] (8,7) to[out=0, in=0] (8,5);
  \draw[-,blue,line width=\lw] (5,8) -- (5,7) -- (8,7);
  \draw[-,blue, line width=\lw] (8,5) -- (0,5);
  \draw[-,red,line width=\lw] (1,8) -- (1,7) -- (3,7) -- (3,3) -- (7,3) -- (7,1) -- (0,1);
  \draw[fill=red] (1,8) circle (\cs);
  \draw[fill=red] (2,7) circle (\cs);
  \draw[fill=red] (7,2) circle (\cs);
  \foreach \y in {4,6}
    \draw[fill=red] (3,\y) circle (\cs);
  \foreach \x in {4,6}
    \draw[fill=red] (\x,3) circle (\cs);
  \foreach \x in {0,2,4,6}
    \draw[fill=red] (\x,1) circle (\cs);
  \draw[fill=blue] (5,8) circle (\cs);
  \foreach \x in {6,8}
    \draw[fill=blue] (\x,7) circle (\cs);
  \foreach \x in {0,2,4,6,8}
    \draw[fill=blue] (\x,5) circle (\cs);
\end{tikzpicture}
&&
\begin{tikzpicture}[scale=\thescale,baseline=40]
  \base
  \draw[-, UQpurple, line width=\lw] (8,7) to[out=0, in=0] (8,5);
  \draw[-,blue,line width=\lw] (5,8) -- (5,7) -- (8,7);
  \draw[-,blue, line width=\lw] (8,5) -- (0,5);
  \draw[-,red,line width=\lw] (1,8) -- (1,7) -- (3,7) -- (3,3) -- (5,3) -- (5,1) -- (0,1);
  \draw[fill=red] (1,8) circle (\cs);
  \draw[fill=red] (2,7) circle (\cs);
  \draw[fill=red] (5,2) circle (\cs);
  \foreach \y in {4,6}
    \draw[fill=red] (3,\y) circle (\cs);
  \draw[fill=red] (4,3) circle (\cs);
  \foreach \x in {0,2,4}
    \draw[fill=red] (\x,1) circle (\cs);
  \draw[fill=blue] (5,8) circle (\cs);
  \foreach \x in {6,8}
    \draw[fill=blue] (\x,7) circle (\cs);
  \foreach \x in {0,2,4,6,8}
    \draw[fill=blue] (\x,5) circle (\cs);
\end{tikzpicture}
\allowdisplaybreaks
\\ &
\begin{tikzpicture}[scale=\thescale,baseline=40]
  \base
  \draw[-,blue,line width=\lw] (5,8) -- (5,7) -- (7,7) -- (7,5) -- (0,5);
  \draw[-,red,line width=\lw] (1,8) -- (1,7) -- (3,7) -- (3,3) -- (7,3) -- (7,1) -- (0,1);
  \draw[fill=red] (1,8) circle (\cs);
  \draw[fill=red] (2,7) circle (\cs);
  \draw[fill=red] (7,2) circle (\cs);
  \foreach \y in {4,6}
    \draw[fill=red] (3,\y) circle (\cs);
  \foreach \x in {4,6}
    \draw[fill=red] (\x,3) circle (\cs);
  \foreach \x in {0,2,4,6}
    \draw[fill=red] (\x,1) circle (\cs);
  \draw[fill=blue] (5,8) circle (\cs);
  \draw[fill=blue] (6,7) circle (\cs);
  \draw[fill=blue] (7,6) circle (\cs);
  \foreach \x in {0,2,4,6}
    \draw[fill=blue] (\x,5) circle (\cs);
\end{tikzpicture}
&&
\begin{tikzpicture}[scale=\thescale,baseline=40]
  \base
  \draw[-,blue,line width=\lw] (5,8) -- (5,7) -- (7,7) -- (7,5) -- (0,5);
  \draw[-,red,line width=\lw] (1,8) -- (1,7) -- (3,7) -- (3,3) -- (5,3) -- (5,1) -- (0,1);
  \draw[fill=red] (1,8) circle (\cs);
  \draw[fill=red] (2,7) circle (\cs);
  \draw[fill=red] (5,2) circle (\cs);
  \foreach \y in {4,6}
    \draw[fill=red] (3,\y) circle (\cs);
  \draw[fill=red] (4,3) circle (\cs);
  \foreach \x in {0,2,4}
    \draw[fill=red] (\x,1) circle (\cs);
  \draw[fill=blue] (5,8) circle (\cs);
  \draw[fill=blue] (6,7) circle (\cs);
  \draw[fill=blue] (7,6) circle (\cs);
  \foreach \x in {0,2,4,6}
    \draw[fill=blue] (\x,5) circle (\cs);
\end{tikzpicture}
&&
\begin{tikzpicture}[scale=\thescale,baseline=40]
  \base
  \draw[-,blue,line width=\lw] (5,8) -- (5,7) -- (7,7) -- (7,5) -- (0,5);
  \draw[-,red,line width=\lw] (1,8) -- (1,7) -- (5,7) -- (5,3) -- (7,3) -- (7,1) -- (0,1);
  \draw[fill=red] (1,8) circle (\cs);
  \foreach \x in {2,4}
    \draw[fill=red] (\x,7) circle (\cs);
  \draw[fill=red] (5,6) circle (\cs);
  \draw[fill=red] (5,4) circle (\cs);
  \draw[fill=red] (6,3) circle (\cs);
  \draw[fill=red] (7,2) circle (\cs);
  \foreach \x in {0,2,4,6}
    \draw[fill=red] (\x,1) circle (\cs);
  \draw[fill=blue] (5,8) circle (\cs);
  \draw[fill=blue] (6,7) circle (\cs);
  \draw[fill=blue] (7,6) circle (\cs);
  \foreach \x in {0,2,4,6}
    \draw[fill=blue] (\x,5) circle (\cs);
\end{tikzpicture}
\end{align*}
\fi
Compare against the atoms for
\[
\overline{\states}_{\lambda, 1}
\sqcup
\overline{\states}_{\lambda, s_1}
\sqcup
\overline{\states}_{\lambda, s_2}
\sqcup
\overline{\states}_{\lambda, s_1 s_2}
\sqcup
\overline{\states}_{\lambda, s_2 s_1}
\sqcup
\overline{\states}_{\lambda, s_2 s_1 s_2}
\]
given in Example~\ref{ex:atom_21}.
Note that each inversion that was present in the atom is now either an inversion for $w_0 w$, an $L$-matrix $\tt{a}'_2$, or a $K$-matrix $\tt{k}_3'$.
\end{example}

\begin{theorem}
\label{thm:sum_of_atoms}
We have
\[
    Z(\states_{\lambda, w}; \zz) = \sum_{y \leqslant w} Z(\overline{\states}_{\lambda, y}; \zz).
\]
\end{theorem}

\begin{proof}
We can show this combinatorially by following using the same idea as the proof of~\cite[Thm.~3.9]{BSW20}.
We consider the paths taken by two colors, and replace vertices accordingly.
Indeed, the first time we see two colors touch at a $\tt{a}_2$ or $\tt{a}_2^{\dagger}$ vertex or at a $K$-matrix $\tt{b}_2$, we replace it by a $\tt{a}_2^{\dagger}$ or a $\tt{k}_2$ respectively.
Every subsequent interaction between two colors becomes an $\tt{a}_2'$ or a $\tt{k}_3'$ for the $K$-matrix.
\end{proof}

As an immediate corollary of Theorem~\ref{thm:sum_of_atoms}, we have $D_w(\zz, \lambda) = \sum_{y \leqslant w} A_y(\zz, \lambda)$ (see Theorem~\ref{thm:lskeys}).

\begin{example}
Let $\lambda = (2, 1)$.
We compute all admissible states in $\states_{\lambda, s_1 s_2}$:
\ifexamples
\newcommand{\cs}{.35} 
\newcommand{\lw}{0.6mm} 
\newcommand{\thescale}{0.35} 
\newcommand{\base}{
  \foreach \y in {1,5} {
    \draw[-] (8,\y) to [out=0, in=0] (8,\y+2);
  }
  \foreach \x in {1,3,5,7}{
    \draw[-] (\x,0) -- (\x,8);
    \foreach \y in {0,2,4,6,8}
      \draw[fill=white] (\x,\y) circle (\cs);
  }
  \foreach \y in {1,3,5,7} {
    \draw[-] (0,\y) -- (8,\y);
    \foreach \x in {0,2,4,6,8}
      \draw[fill=white] (\x,\y) circle (\cs);
  }
}
\begin{align*}
&
\begin{tikzpicture}[scale=\thescale,baseline=40]
  \base
  \draw[-, UQpurple, line width=\lw] (8,7) to[out=0, in=0] (8,5);
  \draw[-, UQgold, line width=\lw] (8,3) to[out=0, in=0] (8,1);
  \draw[-,blue,line width=\lw] (5,8) -- (5,7) -- (8,7);
  \draw[-,blue,line width=\lw] (8,5) -- (0,5);
  \draw[-,red,line width=\lw] (1,8) -- (1,7) -- (5,7) -- (5,3) -- (8,3);
  \draw[-,darkred,line width=\lw] (8,1) -- (0,1);
  \draw[fill=red] (1,8) circle (\cs);
  \foreach \x in {2,4}
    \draw[fill=red] (\x,7) circle (\cs);
  \draw[fill=red] (5,6) circle (\cs);
  \draw[fill=red] (5,4) circle (\cs);
  \draw[fill=red] (6,3) circle (\cs);
  \draw[fill=red] (8,3) circle (\cs);
  \foreach \x in {0,2,4,6,8}
    \draw[fill=darkred] (\x,1) circle (\cs);
  \draw[fill=blue] (5,8) circle (\cs);
  \foreach \x in {6,8}
    \draw[fill=blue] (\x,7) circle (\cs);
  \foreach \x in {0,2,4,6,8}
    \draw[fill=blue] (\x,5) circle (\cs);
\end{tikzpicture}
&&
\begin{tikzpicture}[scale=\thescale,baseline=40]
  \base
  \draw[-, UQpurple, line width=\lw] (8,7) to[out=0, in=0] (8,5);
  \draw[-, UQgold, line width=\lw] (8,3) to[out=0, in=0] (8,1);
  \draw[-,blue,line width=\lw] (5,8) -- (5,7) -- (8,7);
  \draw[-,blue, line width=\lw] (8,5) -- (0,5);
  \draw[-,red,line width=\lw] (1,8) -- (1,7) -- (3,7) -- (3,3) -- (8,3);
  \draw[-,darkred, line width=\lw] (8,1) -- (0,1);
  \draw[fill=red] (1,8) circle (\cs);
  \draw[fill=red] (2,7) circle (\cs);
  \foreach \y in {4,6}
    \draw[fill=red] (3,\y) circle (\cs);
  \foreach \x in {4,6,8}
    \draw[fill=red] (\x,3) circle (\cs);
  \foreach \x in {0,2,4,6,8}
    \draw[fill=darkred] (\x,1) circle (\cs);
  \draw[fill=blue] (5,8) circle (\cs);
  \foreach \x in {6,8}
    \draw[fill=blue] (\x,7) circle (\cs);
  \foreach \x in {0,2,4,6,8}
    \draw[fill=blue] (\x,5) circle (\cs);
\end{tikzpicture}
&&
\begin{tikzpicture}[scale=\thescale,baseline=40]
  \base
  \draw[-, UQgold, line width=\lw] (8,7) to[out=0, in=0] (8,5);
  \draw[-,blue,line width=\lw] (5,8) -- (5,5) -- (0,5);
  \draw[-,red,line width=\lw] (1,8) -- (1,7) -- (8,7);
  \draw[-,darkred,line width=\lw] (8,5) -- (5,5) -- (5,3) -- (7,3) -- (7,1) -- (0,1);
  \draw[fill=red] (1,8) circle (\cs);
  \foreach \x in {2,4,6,8}
    \draw[fill=red] (\x,7) circle (\cs);
  \foreach \x in {6,8}
    \draw[fill=darkred] (\x,5) circle (\cs);
  \draw[fill=darkred] (5,4) circle (\cs);
  \draw[fill=darkred] (6,3) circle (\cs);
  \draw[fill=darkred] (7,2) circle (\cs);
  \foreach \x in {0,2,4,6}
    \draw[fill=darkred] (\x,1) circle (\cs);
  \foreach \y in {6,8}
    \draw[fill=blue] (5,\y) circle (\cs);
  \foreach \x in {0,2,4}
    \draw[fill=blue] (\x,5) circle (\cs);
\end{tikzpicture}
\allowdisplaybreaks
\\ &
\begin{tikzpicture}[scale=\thescale,baseline=40]
  \base
  \draw[-, UQpurple, line width=\lw] (8,7) to[out=0, in=0] (8,5);
  \draw[-, UQgold, line width=\lw] (8,1) to[out=0, in=0] (8,3);
  \draw[-,blue,line width=\lw] (5,8) -- (5,5) -- (0,5);
  \draw[-,red,line width=\lw] (1,8) -- (1,7) -- (8,7);
  \draw[-,darkred
, line width=\lw] (8,5) -- (7,5) -- (7,3) -- (8,3);
  \draw[-,darkred, line width=\lw] (8,1) -- (0,1);
  \draw[fill=red] (1,8) circle (\cs);
  \foreach \x in {2,4,6,8}
    \draw[fill=red] (\x,7) circle (\cs);
  \draw[fill=darkred] (8,5) circle (\cs);
  \draw[fill=darkred] (7,4) circle (\cs);
  \draw[fill=darkred] (8,3) circle (\cs);
  \foreach \x in {0,2,4,6,8}
    \draw[fill=darkred] (\x,1) circle (\cs);
  \draw[fill=blue] (5,8) circle (\cs);
  \draw[fill=blue] (5,6) circle (\cs);
  \foreach \x in {0,2,4}
    \draw[fill=blue] (\x,5) circle (\cs);
\end{tikzpicture}
&&
\begin{tikzpicture}[scale=\thescale,baseline=40]
  \base
  \draw[-, UQgold, line width=\lw] (8,3) to[out=0, in=0] (8,1);
  \draw[-,blue,line width=\lw] (5,8) -- (5,7) -- (7,7) -- (7,5) -- (0,5);
  \draw[-,red,line width=\lw] (1,8) -- (1,7) -- (3,7) -- (3,3) -- (8,3);
  \draw[-,darkred, line width=\lw] (8,1) -- (0,1);
  \draw[fill=red] (1,8) circle (\cs);
  \draw[fill=red] (2,7) circle (\cs);
  \foreach \y in {4,6}
    \draw[fill=red] (3,\y) circle (\cs);
  \foreach \x in {4,6,8}
    \draw[fill=red] (\x,3) circle (\cs);
  \foreach \x in {0,2,4,6,8}
    \draw[fill=darkred] (\x,1) circle (\cs);
  \draw[fill=blue] (5,8) circle (\cs);
  \draw[fill=blue] (6,7) circle (\cs);
  \draw[fill=blue] (7,6) circle (\cs);
  \foreach \x in {0,2,4,6}
  \draw[fill=blue] (\x,5) circle (\cs);
\end{tikzpicture}
&&
\begin{tikzpicture}[scale=\thescale,baseline=40]
  \base
  \draw[-, UQgold, line width=\lw] (8,3) to[out=0, in=0] (8,1);
  \draw[-,blue,line width=\lw] (5,8) -- (5,7) -- (7,7) -- (7,5) -- (0,5);
  \draw[-,red,line width=\lw] (1,8) -- (1,7) -- (5,7) -- (5,3) -- (8,3);
  \draw[-,darkred, line width=\lw] (8,1) -- (0,1);
  \draw[fill=red] (1,8) circle (\cs);
  \foreach \x in {2,4}
    \draw[fill=red] (\x,7) circle (\cs);
  \foreach \y in {6,4}
    \draw[fill=red] (5,\y) circle (\cs);
  \foreach \x in {6,8}
    \draw[fill=red] (\x,3) circle (\cs);
  \foreach \x in {0,2,4,6,8}
    \draw[fill=darkred] (\x,1) circle (\cs);
  \draw[fill=blue] (5,8) circle (\cs);
  \draw[fill=blue] (6,7) circle (\cs);
  \draw[fill=blue] (7,6) circle (\cs);
  \foreach \x in {0,2,4,6}
    \draw[fill=blue] (\x,5) circle (\cs);
\end{tikzpicture}
\end{align*}
\fi
\end{example}

\section{Key algorithm and Proctor patterns}
\label{sec:proctor}

In this section, we will use our model to construct a key algorithm on \emph{reverse} King tableaux~\cite{King75,King76} for $G = \Sp_{2n}$ or Sundaram tableaux~\cite{Sundaram90} for $G = \SO_{2n+1}$.
We begin by recalling the weight preserving bijection between states of the uncolored type $C$ model and symplectic Proctor patterns~\cite[Thm.~4.2]{Proctor94} given in~\cite[Ch.~1]{Ivanov12}.
We then give the analogous bijection between the states of the uncolored type $B$ model and odd orthogonal Proctor patters~\cite[Thm.~7.1]{Proctor94}.
Similar to~\cite{BSW20}, the order of our variables is different by $i \leftrightarrow \overline{n + 1 - i}$, which is the reason we naturally work with reverse King tableaux.
As a consequence, these bijections with our model provides a new proof of~\cite[Thm.~4.2, Thm.~7.1]{Proctor94}.

\subsection{Symplectic patterns and King tableaux}

We consider the case for $G = \Sp_{2n}$, which is the Lie group of Cartan type $C_n$.
We note that these patterns were first given by \v{Z}elobenko~\cite{Zel62}.
A \defn{symplectic Proctor pattern} is a pattern of non-negative integers of the form
\[
\begin{array}{cccccccccc}
a_{1,1} && a_{1,2} && a_{1,3} && \cdots && a_{1,n}
\\ & b_{1,1} && b_{1,2} && b_{1,3} && \cdots && b_{1,n}
\\ && a_{2,2} && a_{2,3} && \cdots && a_{2,n}
\\ &&& b_{2,2} && b_{2,3} && \cdots && b_{2,n}
\\ &&&& \ddots && \ddots && \vdots & \vdots
\\ &&&&&& a_{n-1,n-1} && a_{n-1,n}
\\ &&&&&&& b_{n-1,n-1} && b_{n-1,n}
\\ &&&&&&&& a_{n,n}
\\ &&&&&&&&& b_{n,n}
\end{array}
\]
that satisfies the interlacing conditions
\begin{align*}
\min \{a_{i,j}, a_{i+1,j} \} & \geq b_{i,j} \geq \max \{a_{i,j+1}, a_{i+1,j+1} \},
\\
\min \{b_{i-1,j-1}, b_{i,j-1} \} & \geq a_{i,j} \geq \max \{b_{i-1,j}, b_{i,j} \}.
\end{align*}
The weight of a symplectic Proctor pattern $P$ is given by
\[
\wt(P) := \prod_{i=1}^n z_i^{A_i-2B_i+A_{i+1}},
\qquad
\text{ where } A_i = \sum_{j=i}^n a_i \text{ and } B_i = \sum_{j=i}^n b_i.
\]
We consider $A_{n+1} = 0$.
Let $\PP^C_{\lambda}$ denote the set of symplectic Proctor patterns with top row $\lambda$.

A \defn{King tableau}~\cite{King75,King76} is a filling of a Young diagram with entries in the ordered alphabet $1 < \bon < 2 < \btw < \cdots < n < \bn$ such that the rows are weakly increasing and columns are strictly increasing and the smallest entry in row $i$ is $i$.
The weight of a King tableau $T$ is
\[
\wt(T) = \prod_{i=1}^n x_i^{m_i - m_{\ibar}},
\]
where $m_k$ is the number of times $k$ appears in $T$.
Let $\king_{\lambda}$ denote the set of King tableaux of shape $\lambda$.
A \defn{reverse} King tableau is a King tableau with respect to the alphabet in the reverse order or alternatively the entries in rows (resp.\ columns) are weakly (resp.\ strictly) decreasing and the largest entry in row $i$ is $\overline{n + 1 - i}$.

As discussed in~\cite{Proctor94}, there is a natural bijection $\Theta^C \colon \PP^C_{\lambda} \to \king_{\lambda}$ by extending the usual bijection between GT patterns and semistandard tableaux.
Indeed, the partition $\lambda^{(k)}$ of the $k$-th row indicates the subtableau consisting of all of the letters greater than the $k$-th letter in the alphabet.
For instance, if $k = 3$, then we restrict to the letters $\overline{n-1} < n < \bn$.

\begin{proposition}[{\cite[Ch.~1]{Ivanov12}}]
\label{prop:bijection_typeC}
Let $\states^C_{\lambda}$ denote the uncolored model for $G = \Sp_{2n}$.
There exists a weight-preserving bijection
\[
\Psi^C \colon \states^C_{\lambda} \to \PP^C_{\lambda}.
\]
\end{proposition}

Ivanov constructed the bijection $\Psi^C$ in Proposition~\ref{prop:bijection_typeC} explicitly by extending the usual bijection between the five-vertex model and Gelfand--Tsetlin (GT) patterns (see, \textit{e.g.},~\cite[Sec.~2.1]{BSW20}) and using the $1$ edges between the ${}^{\Delta}_{\Gamma}$ (resp.~${}^{\Gamma}_{\Delta}$) rows to define the $\{a_{ij}\}_{i,j}$ (resp.~$\{b_{ij}\}_{i,j}$) values by the GT pattern bijection.
More precisely, the $i$-th row of vertical edges in the model is the $01$-sequence of the partition in the $i$-th row of the symplectic Proctor pattern read from right-to-left.
We can make the analogous injection on the colored model $\overline{\states}_{\lambda, w}$ or $\states_{\lambda, w}$ by considering the positions of the colored vertical edges or equivalently by forgetting about the colors in our model as an intermediate step.
Note that this is unaffected by the difference between the uncolored version of our lattice model and that in~\cite{Ivanov12} (see Remark~\ref{rem:uncolored_semidual}).

\begin{example}
\label{ex:patterns_typeC}

Consider the states given in Example~\ref{ex:atom_21} for $\overline{\states}_{\lambda, w}$.
Then under the bijection $\Psi^C$, the states correspond to the following symplectic Proctor patterns:
\begin{align*}
1: & \quad
\begin{array}{cccc}
2 && 1 \\
& 1 && 0 \\
&& 1 \\
&&& 0
\end{array}
\qquad
s_1: \quad
\begin{array}{cccc}
2 && 1 \\
& 2 && 0 \\
&& 2 \\
&&& 0
\end{array}
\qquad
s_2: \quad
\begin{array}{cccc}
2 && 1 \\
& 1 && 0 \\
&& 1 \\
&&& 1
\end{array}
\allowdisplaybreaks \\
s_1 s_2: &  \quad 
\begin{array}{cccc}
2 && 1 \\
& 1 && 0 \\
&& 0 \\
&&& 0
\end{array}
\qquad
\begin{array}{cccc}
2 && 1 \\
& 2 && 1 \\
&& 2 \\
&&& 0
\end{array}
\qquad
\begin{array}{cccc}
2 && 1 \\
& 1 && 1 \\
&& 1 \\
&&& 0
\end{array}
\allowdisplaybreaks \\
s_2 s_1: & \quad 
\begin{array}{cccc}
2 && 1 \\
& 2 && 0 \\
&& 2 \\
&&& 1
\end{array}
\qquad
\begin{array}{cccc}
2 && 1 \\
& 2 && 0 \\
&& 2 \\
&&& 2
\end{array}
\allowdisplaybreaks \\
s_1 s_2 s_1: &  \quad 
\begin{array}{cccc}
2 && 1 \\
& 2 && 0 \\
&& 0 \\
&&& 0
\end{array}
\qquad
\begin{array}{cccc}
2 && 1 \\
& 2 && 0 \\
&& 1 \\
&&& 0
\end{array}
\qquad
\begin{array}{cccc}
2 && 1 \\
& 2 && 0 \\
&& 1 \\
&&& 1
\end{array}
\qquad
\begin{array}{cccc}
2 && 1 \\
& 2 && 1 \\
&& 1 \\
&&& 0
\end{array}
\allowdisplaybreaks \\
s_2 s_1 s_2: &  \quad 
\begin{array}{cccc}
2 && 1 \\
& 2 && 1 \\
&& 2 \\
&&& 1
\end{array}
\qquad
\begin{array}{cccc}
2 && 1 \\
& 2 && 1 \\
&& 2 \\
&&& 2
\end{array}
\qquad
\begin{array}{cccc}
2 && 1 \\
& 1 && 1 \\
&& 1 \\
&&& 1
\end{array}
\allowdisplaybreaks \\
w_0: &  \quad
\begin{array}{cccc}
2 && 1 \\
& 2 && 1 \\
&& 1 \\
&&& 1
\end{array}
\end{align*}
\end{example}

\subsection{odd orthogonal patterns and Sundaram tableaux}

Here we instead assume $G = \SO_{2n+1}$, which is the Lie group of Cartan type $B_n$.
An \defn{odd orthogonal Proctor pattern} is a symplectic Proctor pattern such that the values $b_{i,n}$, for all $1 \leq i \leq n$, at the right ends are also allowed to be half integers.
We remark that these patterns were first announced by Gelfand and Tsetlin without proof in~\cite{GT50}.
Let $\PP^B_{\lambda}$ denote the set of odd orthogonal Proctor patterns with top row $\lambda$.

A Sundaram tableau~\cite{Sundaram90} is a King tableau with an additional letter $\infty > \bn$ that is allowed to repeat down columns but can only appear once in a row.
The weight of a Sundaram tableau is the same as for a King tableau; in particular, we ignore $\infty$ in the weight computation.
We denote the set of Sundaram tableau of shape $\lambda$ by $\sundaram_{\lambda}$.
A reverse Sundaram tableau is defined analogously to a reverse King tableau.
Likewise, we have a natural bijection $\Theta^B \colon \PP^B_{\lambda} \to \sundaram_{\lambda}$, as noted in~\cite{Proctor94}, by the same description as $\Theta^C$ except if the rightmost entry in the odd orthogonal Proctor pattern is a half integer, we replace the leftmost entry in the corresponding row with an $\infty$.

Recall that the model for both type $B$ and $C$ Demazure atoms (as well as for Demazure characters) have the same states, but the $\tt{k}_1$ entry of the $K$-matrix has a binomial weight $z^{-2} + z^{-1}$ in type $B$ as opposed to the monomial weight $z^{-2}$ for type $C$.
Thus, following~\cite{BSW20}, we can introduce a marking to the states for this $K$-matrix entry, where if the bend is marked, then it has a Boltzmann weight of $z^{-1}$ and otherwise the Boltzmann weight is  $z^{-2}$.
This yields a bijection between the marked states and the monomials of the partition function as opposed to a product of binomials.

\begin{proposition}
\label{prop:bijection_typeB}
Let $\widehat{\states}_{\lambda}$ denote the set of marked states for the uncolored type $B$ model.
There exists a weight-preserving bijection
\[
\Psi^B \colon \widehat{\states}_{\lambda} \to \PP^B_{\lambda}.
\]
\end{proposition}

\begin{proof}
We extend the bijection $\Psi^C$ given above to the desired bijection by replacing $b_{i,n}$ with $b_{i,n} - \frac{1}{2}$ if the $K$-matrix entry $\tt{k}_1$ is marked.
\end{proof}

Similar to the case when $X = C$, we can extend $\Psi^B$ to an injection with the domain the colored model $\overline{\states}_{\lambda, w}^B$ or $\states^B_{\lambda, w}$.

\begin{example}
Let us take $w = s_2$. Then we have the one state in the model $\overline{\states}^B_{\lambda, w}$ with one $\tt{k}_1$ U-turn that we can mark, which corresponds to the following pair of odd orthogonal Proctor patterns and reverse Sundaram tableaux:
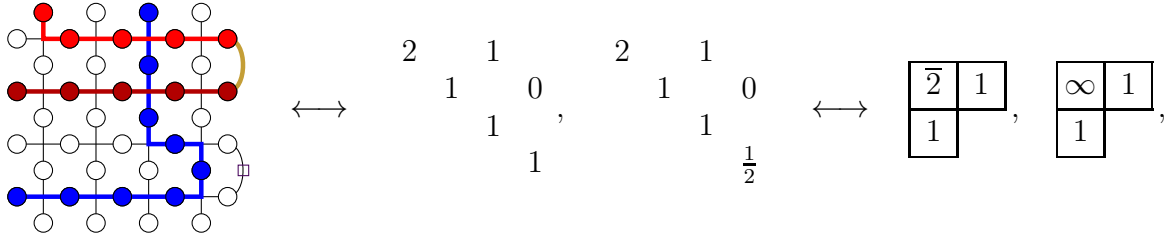
\[
\newcommand{\cs}{.35} 
\newcommand{\lw}{0.6mm} 
\newcommand{\thescale}{0.35} 
\begin{tikzpicture}[scale=\thescale,baseline=40]
  \foreach \y in {1,5} {
    \draw[-] (8,\y) to [out=0, in=0] (8,\y+2);
  }
  \foreach \x in {1,3,5,7}{
    \draw[-] (\x,0) -- (\x,8);
    \foreach \y in {0,2,4,6,8}
      \draw[fill=white] (\x,\y) circle (\cs);
  }
  \foreach \y in {1,3,5,7} {
    \draw[-] (0,\y) -- (8,\y);
    \foreach \x in {0,2,4,6,8}
      \draw[fill=white] (\x,\y) circle (\cs);
  }
  \draw[-, UQgold, line width=\lw] (8,7) to[out=0, in=0] (8,5);
  \draw[-,blue,line width=\lw] (5,8) -- (5,3) -- (7,3) -- (7,1) -- (0,1);
  \draw[-,red,line width=\lw] (1,8) -- (1,7) -- (8,7);
  \draw[-,darkred, line width=\lw] (8,5) -- (0,5);
  \draw[fill=red] (1,8) circle (\cs);
  \foreach \x in {2,4,6,8}
    \draw[fill=red] (\x,7) circle (\cs);
  \foreach \x in {0,2,4,6,8}
  \draw[fill=darkred] (\x,5) circle (\cs);
  \foreach \y in {8,6,4}
    \draw[fill=blue] (5,\y) circle (\cs);
  \draw[fill=blue] (6,3) circle (\cs);
  \draw[fill=blue] (7,2) circle (\cs);
  \foreach \x in {0,2,4,6}
    \draw[fill=blue] (\x,1) circle (\cs);
  \draw[color=UQpurple] (8.6+0.2,2+0.2) rectangle (8.6-0.2,2-0.2);
\end{tikzpicture}
\quad \longleftrightarrow \quad
\begin{array}{cccc}
2 && 1 \\
& 1 && 0 \\
&& 1 \\
&&& 1
\end{array},
\quad
\begin{array}{cccc}
2 && 1 \\
& 1 && 0 \\
&& 1 \\
&&& \frac{1}{2}
\end{array}
\quad\longleftrightarrow\quad
\ytableaushort{{\btw}1,1}\,,
\quad
\ytableaushort{{\infty}1,1}\,,
\]
where the box in the unique state of $\overline{\states}^B_{s_2,\lambda}$ denotes the possible marking.
\end{example}

\subsection{Key algorithm}

Using the bijections $\Psi^X$, we now give a simple algorithm for computing the (right) key $w$ of a reverse King (resp.\ Sundaram) tableau $T$ for $X = C$ (resp.\ $X = B$) using the bijections from~\cite{Proctor94}.

Let $\mcT$ denote a set of tableaux such that $\chi_{\lambda}(\zz) = \sum_{T \in \mcT} \xx^{\wt(T)}$.
The \defn{(right) key} of $\mcT$ is a map $\key \colon T \to W$ such that
\[
\mcT = \bigsqcup_{w \in W} \{T \in \mcT \mid \key(T) = w\}
\qquad \text{ and } \qquad
A_w(\zz, \lambda) = \sum_{\substack{T \in \mcT \\ \key(T) = w}} \zz^{\wt(T)}.
\]

We consider the case when $\mcT$ is the set of \emph{reverse} King (resp.\ Sundaram) tableaux of shape $\lambda$ for $X = C$ (resp.~$X = B$).
Consider some $T \in \mcT$, which we first convert to the corresponding type of Proctor pattern and then to a state in the uncolored lattice model.
That is, the state that we have is $S = \bigl( (\Psi^X)^{-1} \circ (\Theta^X)^{-1} \bigr)(T)$.
From Theorem~\ref{thm:sum_of_atoms}, Theorem~\ref{thm:partition_atom_BC}, and the fact that $D_{w_0}(\zz, \lambda) = \chi_{\lambda}(\zz)$, we know there is a unique way to color the corresponding state $S$.
This coloring gives us a signed permutation $w$ by reading the left boundary of the colored state, which then defines $\key(T) = w$.

\begin{theorem}
Let $\lambda$ be a partition.
Let $X = C$ (resp.~$X = B$), and let $\mcT$ denote the set of reverse King (resp.\ Sundaram) tableaux of shape $\lambda$.
The map $\key \colon \mcT \to W$ defined above is a right key map.
\end{theorem}

For the remainder of this section, we assume $X = C$ unless stated otherwise.
We give a conjecture that would justify that the key map given above is natural.
We first need some notation from crystal theory.
For more details on crystals, we refer the reader to~\cite{BS17}.
A highest weight crystal $B(\lambda)$ is a certain edge-colored weighted connected digraph that satisfies certain additional properties and encodes the action of the corresponding quantum group of $G$.
In particular, we have operators $e_i$ defined by $e_i(b) = b'$ for every edge $b' \xrightarrow[\hspace{20pt}]{i} b$.
Let $u_{\lambda}$ denote the highest weight element of $B(\lambda)$, which is the unique source in the digraph.

Fix some $w \in W$, and chose some reduced expression $w = s_{i_1} s_{i_2} \dotsm s_{i_{\ell}}$.
A \defn{Demazure crystal} is the crystal induced from $B(\lambda)$ by the set of vertices
\[
B_w(\lambda) := \{ b \in B(\lambda) \mid e_{i_{\ell}}^{a_{\ell}} \cdots e_{i_2}^{a_2} e_{i_1}^{a_1} b = u_{\lambda} \text{ for some } a_1, \dotsc, a_{\ell} \in \ZZ_{\geq 0} \}.
\]
This does not depend on the choice of reduced expression for $w$~\cite{K93}.
The \defn{atom crystal} is the crystal induced from $B(\lambda)$ by the set of vertices
\[
\overline{B}_w(\lambda) := B_w(\lambda) \setminus \bigcup_{v < w} B_v(\lambda).
\]
A common model for $B(\lambda)$ for type $C_n$ is the set of Kashiwara--Nakashima (KN) tableaux, which has a natural crystal structure~\cite{KN94}.

The Sheats bijection~\cite{Sheats99} is a weight-preserving bijection between King tableaux and a variant of KN tableaux called DeConcini tableaux~\cite{DeConcini79}.
There is a natural map from KN tableaux to DeConcini tableaux, and so we will abuse terminology slightly and call the composite of these maps the Sheats bijection from KN tableaux to King tableaux.
However, we have the weight twisting by $i \leftrightarrow \overline{n+1-i}$ coming from the bijection from KN tableaux to DeConcini tableaux.
To go between reverse King tableaux and normal King tableaux, we simply replace $i \leftrightarrow \overline{n+1-i}$ entry-wise, where $\overline{\ibar} = i$.
In order to get the alphabets to match, we require interchanging $i \leftrightarrow \ibar$ by using the tableau switching algorithm from~\cite{BSS96}.
This leads to the following conjecture.

\begin{conjecture}
\label{conj:atom_structure}
Let $\mcT_{\lambda}$ denote the set of reverse King tableaux of shape $\lambda$.
There exists a crystal structure on $\mcT_{\lambda}$ such that $\{T \in \mcT_{\lambda} \mid \key(T) = w\}$ equals the corresponding crystal atom $\overline{B}_w(\lambda)$.
Additionally, applying the tableau switching algorithm from $\mcT_{\lambda} \to \king_{\lambda}$ is a crystal isomorphism with the crystal structure on $\king_{\lambda}$ given by Lee~\cite{Lee19}.
Moreover, the composition of
\begin{enumerate}
\item applying the tableau switching algorithm to interchange $i \leftrightarrow \ibar$ for all $i$, \label{map:switch}
\item replacing $i \leftrightarrow \overline{n+1-i}$ entry-wise, and \label{map:reverse}
\item applying the Sheats bijection
\end{enumerate}
is a crystal isomorphism on the corresponding crystal atoms that sends $\key$ to the key map defined in~\cite{JL19,Santos19}.
\end{conjecture}

We note that first two maps~(\ref{map:switch}) and~(\ref{map:reverse}) commute.

\begin{example}
\label{ex:evidence}
Consider the states given in Example~\ref{ex:atom_21} and the symplectic Proctor patterns under $\Psi^C$ in Example~\ref{ex:patterns_typeC}.
The corresponding reverse King tableaux with the alphabet $1 < \bon < 2 < \btw$ are
\ytableausetup{boxsize=1.3em}
\begin{align*}
1: &  \quad \ytableaushort{21,1}\,,
\qquad
s_1:  \quad \ytableaushort{22,1}\,,
\qquad
s_2:  \quad \ytableaushort{{\btw}1,1}\,,
\allowdisplaybreaks \\
s_1 s_2: &  \quad \ytableaushort{{\bon}1,1}\,,\ \ytableaushort{22,{\bon}}\,,\ \ytableaushort{2{\bon},1}\,,
\qquad
s_2 s_1:  \quad \ytableaushort{{\btw}2,1}\,,\ \ytableaushort{{\btw}{\btw},1}\,,
\allowdisplaybreaks \\
s_1 s_2 s_1: &  \quad \ytableaushort{{\bon}{\bon},1}\,,\ \ytableaushort{2{\bon},1}\,,\ \ytableaushort{{\btw}{\bon},1}\,,\ \ytableaushort{2{\bon},{\bon}}\,,
\qquad
s_2 s_1 s_2: \quad \ytableaushort{{\btw}2,{\bon}}\,,\ \ytableaushort{{\btw}{\btw},{\bon}}\,,\ \ytableaushort{2{\bon},1}\,,
\allowdisplaybreaks \\
w_0: &  \quad \ytableaushort{{\btw}{\bon},{\bon}}\,,
\end{align*}
Next, we change the alphabet by $i \leftrightarrow \ibar$ using the tableau switching algorithm, so our new alphabet is $\bon < 1 < \btw < 2$, to yield
\begin{align*}
1: &  \quad \ytableaushort{21,1}\,,
\qquad
s_1:  \quad \ytableaushort{22,1}\,,
\qquad
s_2:  \quad \ytableaushort{{\btw}1,1}\,,
\allowdisplaybreaks \\
s_1 s_2: &  \quad \ytableaushort{11,{\bon}}\,,\ \ytableaushort{22,{\bon}}\,,\ \ytableaushort{2{\bon},1}\,,
\qquad
s_2 s_1:  \quad \ytableaushort{2{\btw},1}\,,\ \ytableaushort{{\btw}{\btw},1}\,,
\allowdisplaybreaks \\
s_1 s_2 s_1: &  \quad \ytableaushort{1{\bon},{\bon}}\,,\ \ytableaushort{2{\bon},1}\,,\ \ytableaushort{{\btw}{\bon},1}\,,\ \ytableaushort{2{\bon},{\bon}}\,,
\qquad
s_2 s_1 s_2: \quad \ytableaushort{2{\btw},{\bon}}\,,\ \ytableaushort{{\btw}{\btw},{\bon}}\,,\ \ytableaushort{2{\bon},1}\,,
\allowdisplaybreaks \\
w_0: &  \quad \ytableaushort{{\btw}{\bon},{\bon}}\,.
\end{align*}
We then apply the map that sends $i \leftrightarrow \overline{n+1-i}$ entry-wise to obtain
\begin{align*}
1: &  \quad \ytableaushort{{\bon}{\btw},{\btw}}\,,
\qquad
s_1:  \quad \ytableaushort{{\bon}{\bon},{\btw}}\,,
\qquad
s_2:  \quad \ytableaushort{1{\btw},{\btw}}\,,
\allowdisplaybreaks \\
s_1 s_2: &  \quad \ytableaushort{{\btw}{\btw},2}\,,\ \ytableaushort{{\bon}{\bon},2}\,,\ \ytableaushort{{\bon}{\btw},2}\,,
\qquad
s_2 s_1:  \quad \ytableaushort{{\bon}1,{\btw}}\,,\ \ytableaushort{11,{\btw}}\,,
\allowdisplaybreaks \\
s_1 s_2 s_1: &  \quad \ytableaushort{{\btw}2,2}\,,\ \ytableaushort{{\bon}2,{\btw}}\,,\ \ytableaushort{12,{\btw}}\,,\ \ytableaushort{{\bon}2,2}\,,
\qquad
s_2 s_1 s_2:  \quad \ytableaushort{{\bon}1,2}\,,\ \ytableaushort{11,2}\,,\ \ytableaushort{1{\btw},2}\,,
\allowdisplaybreaks \\
w_0: &  \quad \ytableaushort{12,2}\,,
\end{align*}
Finally, we use the Sheats bijection to obtain the KN tableaux (with the alphabet now $1 < 2 < \btw < \bon$)
\begin{align*}
1: &  \quad \ytableaushort{11,2}\,,
\qquad
s_1:  \quad \ytableaushort{12,2}\,,
\qquad
s_2:  \quad \ytableaushort{11,{\btw}}\,,
\allowdisplaybreaks \\
s_1 s_2: &  \quad \ytableaushort{12,{\btw}}\,,\ \ytableaushort{22,{\bon}}\,,\ \ytableaushort{22,{\btw}}\,,
\qquad
s_2 s_1:  \quad \ytableaushort{1{\btw},2}\,,\ \ytableaushort{1{\btw},{\btw}}\,,
\allowdisplaybreaks \\
s_1 s_2 s_1: &  \quad \ytableaushort{2{\bon},{\btw}}\,,\ \ytableaushort{1{\bon},2}\,,\ \ytableaushort{1{\bon},{\btw}}\,,\ \ytableaushort{2{\bon},{\bon}}\,,
\qquad
s_2 s_1 s_2:  \quad \ytableaushort{2{\btw},{\bon}}\,,\ \ytableaushort{{\btw}{\btw},{\bon}}\,,\ \ytableaushort{2{\btw},{\btw}}\,,
\allowdisplaybreaks \\
w_0: &  \quad \ytableaushort{{\btw}{\bon},{\bon}}\,,
\end{align*}
which are precisely the crystal atoms (see~\cite[Ex.~16]{Santos19}).
\end{example}

We remark that applying $w_0$ to the weight means we have $w_0 \wt = -\wt$, and in terms of the bijection with Proctor patterns, we instead use the initial alphabet $\bon < 1 < \btw < 2$.
In this case, we are taking dual atoms from the lowest weight element and replacing $e_i \mapsto f_i$ in the definition of a Demazure crystal.
Equivalently, we are applying the Lusztig involution (or the contragradient dual) to the crystal atom.
However, this does not remove the fact we are working with reverse King tableaux unlike in~\cite{BSW20}.

\begin{example}
We note that applying tableau switching $i \leftrightarrow \ibar$ is not the same as taking the Proctor patterns for the alphabet changed under $w_0$.
Indeed, the following tableaux are fixed under the tableau switching but are interchanged using the alphabet $\bon < 1 < \btw < 2$:
\[
\ytableaushort{2{\bon},1}\,,
\qquad\qquad
\ytableaushort{21,{\bon}}\,.
\]
\end{example}

If we instead took the order of the spectral parameters from~\cite{Gray17,Ivanov12}, the weight would instead be twisted by the longest element $w_A$ of natural $S_n \subseteq W$.
Yet this would give us King tableaux as our image under the bijection $\Phi^C$.
This would mean we would have a key algorithm on King tableaux for a twisted version of atoms working with an extremal weight crystal $B(w_A \lambda)$ in the terminology of~\cite{K93}.
We have similar results for $G = \SO_{2n+1}$ and Sundaram tableaux, as well as conjecture a crystal structure that is an extension of a solution of Conjecture~\ref{conj:atom_structure}.


\appendix
\section{\textsc{SageMath} code}

We give our \textsc{SageMath}~\cite{sage} code that we used to prove Proposition~\ref{prop:YBE}.

\begin{lstlisting}
def type_C_R_matrix(gamma_L1=True, gamma_L2=True, max_m=4):
    BR = ZZ
    c = [1] * max_m

    # The auxiliary space is the horizontal lines and the quantum space
    # is the vertical line. The input the left and bottom sides.
    def L_gamma(aux_in, q_in, aux_out, q_out, z):
        if set([aux_in, q_out]) != set([aux_out, q_in]):
            return 0
        if aux_in == aux_out == q_out == q_in == 0:  # blank
            return 1
        if aux_out == q_out == 0:  # right turn corner
            return z*c[aux_in-1]
        if aux_in == q_in == 0:  # left turn corner
            return 1*c[aux_out-1]
        if q_in == q_out == 0:  # horizontal line
            return z*c[aux_in-1]
        if aux_in == aux_out == 0:  # vertical line
            return 0
        # Must be a crossing
        assert aux_in > 0 and aux_out > 0 and q_in > 0 and q_out > 0
        if aux_in == aux_out:
            if q_in == q_out == aux_in == aux_out:
                return z*c[aux_in-1]^2
            if q_in < aux_out:
                return z*c[aux_in-1]*c[q_in-1]
        elif aux_in == q_in:
            if q_in < q_out:
                return z*c[aux_in-1]*c[aux_out-1]
        return 0

    # Infinite recursion if we set L_gamma = L_gamma_mirror,
    #   so we introduce an auxiliary function name.
    L_gamma_mat = L_gamma

    def L_delta(aux_in, q_in, aux_out, q_out, z):
        if set([aux_in, q_in]) != set([aux_out, q_out]):
            return 0
        if aux_in == aux_out == q_out == q_in == 0:  # blank
            return z
        if aux_out == q_in == 0:  # right turn corner
            return 1/c[aux_in-1]
        if aux_in == q_out == 0:  # left turn corner
            return z/c[aux_out-1]
        if q_in == q_out == 0:  # horizontal line
            return 1/c[aux_in-1]
        if aux_in == aux_out == 0:  # vertical line
            return 0
        # Must be a crossing
        assert aux_in > 0 and aux_out > 0 and q_in > 0 and q_out > 0
        if aux_in == aux_out:
            if q_in == q_out == aux_in == aux_out:
                return 1/c[aux_in-1]^2
            if q_in > aux_in:
                return 1/c[aux_in-1]/c[q_in-1]
        elif aux_in == q_out:
            if q_in > q_out:
                return 1/c[aux_in-1]/c[aux_out-1]
        return 0

    if gamma_L1:
        L1_wt = L_gamma_mat
    else:
        L1_wt = L_delta
    if gamma_L2:
        L2_wt = L_gamma_mat
    else:
        L2_wt = L_delta

    states_to_vars = {(in_top,in_bot,out_top,out_bot): 0
                      for in_top in range(max_m)
                      for in_bot in range(max_m)
                      for out_top in range(max_m)
                      for out_bot in range(max_m)
                      }
    zi, zj = BR['zi,zj'].fraction_field().gens()
    base = zi.parent()
    S = PolynomialRing(base, 'x', len(states_to_vars))
    vars_to_states = []
    for i, st in enumerate(states_to_vars):
        states_to_vars[st] = S.gen(i)
        vars_to_states.append(st)
    def R_wt(in_top, in_bot, out_top, out_bot):
        return states_to_vars.get((in_top,in_bot,out_top,out_bot), 0)

    data = [list(range(max_m)), list(range(max_m)), list(range(max_m))])
    states = list(cartesian_product(data)
    print("Building matrices")

    L1 = matrix(base, [[L1_wt(s[0], s[2], t[0], t[2], zj) if s[1] == t[1]
                       else 0 for t in states] for s in states])
    L2 = matrix(base, [[L2_wt(s[1], s[2], t[1], t[2], zi) if s[0] == t[0]
                       else 0 for t in states] for s in states])
    R = matrix(S, [[R_wt(s[0], s[1], t[1], t[0]) if s[2] == t[2] else 0
                    for t in states] for s in states])
    print("Computing RLL - LLR")
    RLL = R * (L1 * L2) - (L2 * L1) * R

    print("Setting up equations")
    def extract_coeffs(p):
        d = p.dict()
        zero = p.parent().zero()
        return [d.get(gen.exponents()[0], zero) for gen in S.gens()]
    M = matrix(base, [extract_coeffs(RLL[i,j])
                      for i in range(len(states))
                      for j in range(len(states))])
    print("Computing kernel")
    ker = [b for b in M.right_kernel().basis()]
    assert len(ker) == 1
    ret = {}
    for i, val in enumerate(ker[0]):
        if val != 0:
            ret[vars_to_states[i]] = SR(val).factor()
    return ret
\end{lstlisting}

\bibliographystyle{alpha} 
\bibliography{typeC}

\end{document}